\documentclass[a4paper, english, 10pt]{smfart} 

\usepackage{placeins} 

\usepackage[all]{xy}

\usepackage{setspace}
\usepackage[british]{babel}

\usepackage[OT2,T1]{fontenc}
\usepackage{amssymb} 
\usepackage{mathrsfs} 

\usepackage{stmaryrd}
\usepackage{amsthm}

\usepackage{amsmath}
\usepackage[latin1]{inputenc}

\allowdisplaybreaks[1]

\usepackage{graphicx}
\usepackage{manfnt} 

 \usepackage{booktabs} 

\usepackage[margin=2.7cm]{geometry}

\usepackage[all]{xy}

\usepackage{math tools} 

\usepackage{enumitem}

\usepackage[pagebackref=true, colorlinks=true, linkcolor=black, citecolor=black, urlcolor=black]{hyperref}
\usepackage{smfthm}

\usepackage[msc-links, alphabetic]{amsrefs}

\usepackage{math tools} 

\usepackage{enumitem}

\DeclareSymbolFont{cyrletters}{OT2}{wncyr}{m}{n}
\DeclareMathSymbol{\Sha}{\mathalpha}{cyrletters}{"58}

\newtheorem{theoA}{Theorem}

\newtheorem*{coro*}{Corollary}
\newtheorem*{conj*}{Conjecture}
\newtheorem*{lemm*}{Lemma}

\newcommand{\smallmat}[4]{\bigl(\begin{smallmatrix}#1&#2\\#3&#4\end{smallmatrix}\bigr)}
\providecommand{\twomat}[4]{\left(\begin{array}{cc}#1&#2\\#3&#4\end{array}\right)}
\providecommand{\smalltwomat}[4]{\left(\begin{smallmatrix}#1&#2\\#3&#4\end{smallmatrix}\right)}

\theoremstyle{definition}

\theoremstyle{remark}
\newtheorem{remark*}{Remark}

\NumberTheoremsIn{subsection}

\numberwithin{equation}{subsection}
\numberwithin{table}{subsection}

\newcommand{\ot}{\otimes}

\newcommand{\ts}{\times}

\newcommand{\beq}{\begin{equation}\begin{aligned}}
\newcommand{\eeq}{\end{aligned}\end{equation}}

\newcommand{\beqq}{\begin{equation*}\begin{aligned}}
\newcommand{\eeqq}{\end{aligned}\end{equation*}}

\newcommand{\lb}[1]{\label{#1}}

\newcommand{\one}{\mathbf{1}}
\newcommand{\Q}{\mathbf{Q}}
\newcommand{\qqq}{\mathbf{q}}

\newcommand{\GL}{\mathrm{GL}}
\newcommand{\G}{\mathrm{G}}
\renewcommand{\H}{\mathrm{H}}

\newcommand{\GO}{\mathrm{GO}}

\newcommand{\X}{\mathscr{X}}

\newcommand{\sm}{\mathrm{sm}}

\newcommand{\cH}{\mathscr{H}}

\newcommand{\cV}{\mathscr{V}}

\newcommand{\cZ}{\mathscr{Z}}

\newcommand{\calL}{\mathscr{L}}

\newcommand{\cW}{\mathscr{W}}

\newcommand{\R}{\mathbf{R}}
\newcommand{\Z}{\mathbf{Z}}

\newcommand{\uw}{\underline{w}}

\newcommand{\ul}{\underline{l}}
\newcommand{\uk}{\underline{k}}
\newcommand{\ukappa}{\underline{\kappa}}

\newcommand{\frakp}{\mathfrak{p}}

\newcommand{\rW}{\mathrm{W}}
\newcommand{\rf}{\mathrm{f}}

\newcommand{\ad}{\mathrm{ad}}
\newcommand{\ab}{\mathrm{ab}}

\newcommand{\into}{\hookrightarrow}
\newcommand{\cd}{\cdot}

\newcommand{\Y}{\mathscr{Y}}

\newcommand{\la}{\langle}
\newcommand{\ra}{\rangle}

\newcommand{\lm}{\lambda}

\newcommand{\sg}{\sigma}
\newcommand{\Sg}{\Sigma}

\newcommand{\cL}{\mathscr{L}}

\newcommand{\bff}{\mathbf{f}}
\newcommand{\bI}{\mathbf{I}}

\newcommand{\cK}{\mathscr{K}}

\newcommand{{\calG}}{\mathscr{G}}

\newcommand{\cS}{\mathscr{S}}

\newcommand{\calS}{\mathcal{S}}

\newcommand{\bcalS}{\baar{\mathcal{S}}}

\newcommand{\C}{\mathbf{C}}

\newcommand{\N}{\mathbf{N}}
\newcommand{\OO}{\mathscr{O}}
\newcommand{\A}{\mathbf{A}}

\newcommand{\bks}{\backslash}
\newcommand{\baar}{\overline}

\newcommand{\eps}{\varepsilon}
\newcommand{\vphi}{\varphi}
\newcommand{\vpi}{\varpi}

\newcommand{\Up}{\mathrm{U}}

\newcommand{\wtil}{\widetilde}
\newcommand{\Res}{\mathrm{Res}}
\newcommand{\B}{\mathbf{B}}

\newcommand{\W}{\mathscr{W}}

\newcommand{\ord}{\mathrm{ord}} 
\newcommand{\hol}{\mathrm{hol}} 

\newcommand{\cl}{\mathrm{cl}} 

\newcommand{\vol}{\mathrm{vol}}
\newcommand{\Tr}{\mathrm{Tr}}

\newcommand{\Gal}{\mathrm{Gal}}

\newcommand{\Ker}{\mathrm{Ker}\,}

\newcommand{\Hom}{\mathrm{Hom}\,}
\newcommand{\End}{\mathrm{End}\,}

\newcommand{\llb}{\llbracket}
\newcommand{\rrb}{\rrbracket}

\newcommand{\Spec}{\mathrm{Spec}\,}

\newcommand{\id}{\mathrm{id}}

\def\Xint#1{\mathchoice
      {\XXint\displaystyle\textstyle{#1}}%
      {\XXint\textstyle\scriptstyle{#1}}%
      {\XXint\scriptstyle\scriptscriptstyle{#1}}%
      {\XXint\scriptscriptstyle\scriptscriptstyle{#1}}%
      \!\int}
   \def\XXint#1#2#3{{\setbox0=\hbox{$#1{#2#3}{\int}$}
        \vcenter{\hbox{$#2#3$}}\kern-.5\wd0}}
   
   \def\dashint{\Xint-}

\title[$p$-adic $L$-functions for $\GL_{2}\ts {\rm GU}(1)$]{$p$-adic $L$-functions via local-global interpolation:\\ the case of $\GL_{2}\ts {\rm GU}(1)$}

\author{Daniel Disegni} 
\address{Department of Mathematics, Ben-Gurion University of the Negev, Be'er Sheva 84105, Israel}
\email{disegni@bgu.ac.il}

\begin{document}

\begin{abstract} 
Let $F$ be a totally real field and let $E/F$ be a CM quadratic extension. We construct a $p$-adic $L$-function 
attached to Hida families for the group  $\GL_{2/F}\ts \Res_{E/F}\GL_{1}$. It is characterised by an exact interpolation property for critical Rankin--Selberg  $L$-values,  at  classical points corresponding to representations $\pi\boxtimes \chi$ with the weights of~$\chi$ smaller than the weights of~$\pi$.

Our $p$-adic $L$-function agrees with previous results of Hida when $E/F$ splits above~$p$ or $F=\Q$, and it is new otherwise. 
Exploring a method that should bear further fruits, we build it  as a ratio of families of global and local Waldspurger zeta integrals, the latter constructed using the local Langlands correspondence in families.

In an appendix of possibly independent recreational interest, 
 we give a reality-TV-inspired  proof of an identity concerning double factorials.  
\end{abstract}	
\thanks{This work was supported by ISF grant 1963/20 and BSF grant 2018250.}

\maketitle

\tableofcontents

\section{Introduction}

This paper is a case study in the construction of  $p$-adic $L$-functions by the `soft' method of  glueing  ratios of matching families of global and local zeta integrals. The local integrals are constructed and then inserted into the global context by using the local Langlands correspondence in families (see \cite{LLC} and  references therein). The method, whose deployment seems new for non-abelian families, should be of wide applicability; we give a brief introductory description in \S~\ref{12}.

\medskip

The specific context and arithmetic interest of our work is the following. Let $F$ be a totally real field, let $E/F$ be a CM quadratic extension, and let  $p$ be a rational prime. We construct a meromorphic function $\cL_{p}(\cV)$ on Hida families for $\GL_{2/F}\ts\Res_{E/F}\GL_{1}$
that interpolates  critical values
 $${L(1/2, \pi_{E}\ot\chi)\over L(1, \pi, \ad)},$$ for $p$-ordinary automorphic representations $\pi\boxtimes \chi$ 
such that $\chi$ has lower weights than $\pi$. (The precise statement is Theorem \ref{thm A} below; note that in our normalisation, the above numerators are not necessarily central values.) The function $\cL_{p}(\cV)$ is new (if not  surprising) at least when $E/F$ does not split above~$p$;  for a discussion of  previous related works see \S~\ref{previous}. 

The interpolation property of $\cL_{p}(\cV)$ holds at all classical points satisfying the weight condition and lying outside the polar locus (on which we have partial control),  and it provides an entirely explicit and complete characterisation of the function, in the spirit of \cite{Hi-genuine}. Its generality and precision are key to some arithmetic applications in \cite{univ}, which motivated our choice of case. In that paper, we prove, first, the $p$-adic Beilinson--Bloch--Kato conjecture in analytic rank~$1$ for (conjugate-)selfdual motives attached to representations $\pi_{E}\ot\chi$ as above;
and secondly,  one divisibility in an Iwasawa Main Conjecture for the cyclotomic derivative ${\rm d}^{\sharp}\cL_{p}(\cV)$ of $\cL_{p}(\cV)$ along a  selfdual locus.  Both results,  new or partly new even when $F=\Q$ and $E/F$ splits at $p$, rely on a $p$-adic Gross--Zagier formula for ${\rm d}^{\sharp}\cL_{p}(\cV)$. In turn, that formula is proved by analytically continuing  formulas from \cite{pyzz, nonsplit} for the central derivatives of certain cyclotomic $p$-adic $L$-functions $\cL_{p}(V_{(\pi, \chi)}, s)$ attached to those representations $\pi\boxtimes\chi$ as above that have minimal weights. The continuation argument  thus requires to exactly    identify the collection $\{(\cL_{p}(V_{(\pi, \chi)} , s)_{(\pi,\chi)}\}$ of  single-variable functions as a set of specialisations of  a multivariable analytic function, which is indeed our~$\cL_{p}(\cV)$.

\medskip

 It would be interesting to extend our results to the non-ordinary case by the method of  \cite{Urban, AI}.
As for further arithmetic directions in the ordinary case,\footnote{This discussion has no ambition of being either a comprehensive research programme or a comprehensive survey of the growing literature on the subject. Moreover, it  entirely leaves out not only the case of non-ordinary families, but also  the  $p$-adic $L$-function with complementary (to \eqref{wt ineqq}) interpolation range introduced in \cite{BDP}.}  the main remaining goal is perhaps
 the full Iwasawa Main Conjecture for $\cL_{p}(\cV)$.  This was proved by Skinner--Urban \cite{SU} and Wan \cite{Wan} in the split case; in the non-split case, results toward it (when $F=\Q$) were recently obtained by  B\"uy\"ukboduk--Lei \cite{BL}. A second goal
is the remaining divisibility in the Main Conjecture for the cyclotomic derivative of $\cL_{p}(\cV)$ (cf. \cite[Theorem E]{univ}); in view of the $p$-adic Gross--Zagier formula of \cite{univ}, this is equivalent to a suitable generalisation of Perrin-Riou's main conjecture for Heegner points, which in its original form was recently proved by Burungale--Castella--Kim \cite{BCK}.

\subsection{Statement of the main result} We move toward stating our main theorem, leaving  a few  of the  detailed definitions of the objects involved to the body of the paper.  

\subsubsection{$p$-adic automorphic representations} \lb{sec 111}
Consider the algebraic groups over $F$
\beq\lb{grps}
\G=\GL_{2/F}, \qquad \H:=\Res_{E/F} \GL_{1/F}.\eeq
 If $v_{0} $ is a place of $\Q$, we denote  $$\Sg_{v_{0}}=\Hom(F, \baar{\Q}_{v_{0}}).$$
 A (numerical) $v_{0}$-adic weight for $\G$ is a tuple  $\uw:= (w_{0}, w= (w_{\tau})_{\tau\in \Sg_{v_{0}}})$ of integers, all of the same parity, such that $w_{\tau}\geq 0$ for all $\tau$. It is said cohomological if $w_{\tau}\geq 2$ for all $\tau$. 
 A weight for $\H$ is a tuple $\ul=(l_{0},l=( l_{\tau})_{\tau\in \Sg_{v_{0}}})$ of integers of the same parity. Finally, if $\uw$ and $\ul$ are weights for $\G$ and $\H$, the associated \emph{contracted weight} for $\G\ts \H$ is\footnote{A contracted weight is the same as a weight for $(\G\ts\H)^{1}:=\{(g, h): \det (g) = N_{E/F}(h)\} \subset \G\ts\H$. This is in fact the true group governing our constructions.}
$$(w_{0}+l_{0}, w, l).$$
 
If $\uw$ is a $p$-adic  weight (say, for $\G$) and $\iota\colon \baar{\Q}_{p}\into \C$ is an embedding, we denote $\uw^{\iota}:= (w_{0}, (w_{\tau})_{\iota\circ \tau\colon F\into \C})$. (In fact $\uw^{\iota}$ only depends on $\iota_{|L}$ if $\uw$ is rational over the finite extension $L$ of $\Q_{p}$ in the sense that  $\Gal(\baar{\Q}_{p}/L)$ fixes $\uw$.) Let $\A$ be the ring of ad\`eles of $F$. 
 An automorphic representation of archimedean weight $\uw$ is a complex automorphic representation $\pi$ of $\G(\A)$ such that $\pi_{\infty}=\pi_{\uw}:=\ot_{\tau\colon F\into \R} \pi_{(w_{0}, w_{\tau})}$, where $\pi_{(w_{0}, w_{\tau})}$ is the discrete series of $\GL_{2}(F_{\tau})$ of weight $w_{\tau}$ and central character $z\mapsto z^{w_{0}}$.
 If $L$ is $p$-adic, we define an \emph{automorphic representation of $\G(\A)$ of weight $\uw$ over $L$} to be a representation  $\pi$ of $\G(\A^{\infty})$ on an $L$-vector space, such that for every $\iota\colon L\into \C$, the representation
  $$\pi^{\iota}:=\pi\ot_{L, \iota} \pi_{\infty, \uw^{\iota}}$$
  of $\G(\A)$ is automorphic.\footnote{This definition is slightly different from, but equivalent to, the one adopted in \cite{univ}, whose flexibility won't be needed here.} 
  
  To  the representation $\pi$ over $L$ is attached a $2$-dimensional representation $V_{\pi}$ of $G_{F}:=\Gal(\baar{F}/F)$; denoting by $V_{\pi,v}$ its restriction to a decomposition group at a place $v$ of $F$, the representation $V_{\pi}$ is characterised by $L(V_{\pi,v},s)=L(s+1/2, \pi_{v})$  for all $v$ (this is the `Hecke' normalisation of the Langlands correspondence, cf. \cite[\S 3.2]{deligne}). 
We say that  $\pi$ is \emph{ordinary}\footnote{The literature often adds to the definition the extra restriction that $\alpha_{\pi,v}^{\circ}$ should be unramified, and calls `nearly ordinary' what we call `ordinary'.} if for each place $v\vert p$ of $F$
there is  a nontrivial $G_{F_{v}}$-stable filtration 
\beq 0 \to V_{\pi, v}^{+} \to V_{\pi,v} \to V^{-}_{\pi,v}\to 0\eeq
such that the character 
$\alpha_{\pi, v}^{\circ}\colon F_{v}^{\ts}\to L^{\ts}$ corresponding
to $V^{+}_{\pi, v}(-1)$ has values in $\OO_{L}^{\ts}$.

Let $L$ be a $p$-adic field splitting $E$,   suppose chosen for each $\tau\colon F\into L$ an extension $\tau'$ to $E$,\footnote{This will only intervene in the numerical labelling of the weights.} and let $\tau'^{c}=\tau'\circ c$ for the complex conjugation $c$ of $E/F$.
A Hecke character of $\H$ of weight $\ul $ over the $p$-adic field $L$ is a locally algebraic character $\chi\colon E^{\ts}\bks \A_{E}^{\infty, \ts}\to L^{\ts}$ such that 
$$\chi(t_{p}) = \prod_{\tau\colon F\into L} \tau'(t_{p})^{(l_{\tau}+l_{0})/2}  \tau'^{c}(t_{p})^{(-l_{\tau}+l_{0})/2}$$ for all $t_{p}$ in some neighbourhood of  $1\in E_{p}^{\ts}$.  We let $V_{\chi}$ be the $1$-dimensional $G_{E}$-representation corresponding to $\chi$.

\subsubsection{$L$-values}
Let $\pi$, respectively $\chi$, be a complex automorphic representation of $\G(\A)$, respectively $\H(\A)$, and let $\pi_{E}$ denote the base-change of $\pi$ to $E$.
Let us also introduce the convenient notation 
\beq \lb{virtual} \text{``$V_{(\pi, \chi), v}:= (V_{\pi, v}\ot{\rm Ind}_{F_{v}}^{E_{v}}\chi_{v})\ominus \ad(V_{\pi,v})(1)$''} \eeq
(to be thought of as referring to a `virtual motive').

Let  $\eta\colon F^{\ts}\bks \A^{\ts} \to \{\pm 1\}$ be the character associated with $E/F$, and let
\beqq
\calL(V_{(\pi, \chi), v},0)&:= {\zeta_{F,v}(2) L(1/2,\pi_{E, v}\ot\chi_{v}) \over {L(1, \eta_{v})L( 1, \pi_{v}, \ad)}}
\cdot  \begin{cases} 1 \quad &\text{if $v\nmid\infty $}\\ \pi^{-1}  &\text{if
 $v\vert \infty$}\end{cases}
\quad \in \C,\\
\calL(V_{(\pi, \chi)},0)&:= \prod_{v} \calL(V_{(\pi, \chi), v},0),
\eeqq
where the product (in the sense of analytic continuation) is over all places.  These are the $L$-values we will interpolate. 

\subsubsection{Interpolation factors} \lb{sec psi intro}
Let $L$ be a finite extension of $\Q_{p}$,  let $\pi$ be an ordinary automorphic representation of $\G(\A)$ over $L$, with a locally algebraic central character $\omega_{\pi}\colon \A^{\infty , \ts}\to L^{\ts}$,  let $\chi\colon \H(F)\bks \H(\A)\to L^{\ts}$ be a locally algebraic character, and set $\omega_{\chi}:= \chi_{|\A^{\infty, \ts}}$. Let $\iota\colon L\into \C$ be an embedding and let $\psi=\prod_{v}\psi_{v}\colon F\bks \A\to \C^{\ts}$ be the standard additive character such that $\psi_{\infty}(x)=e^{2\pi i \Tr_{F_{\infty}/\R}(x)}$. 

If $v\vert p$, let  $\ad (V_{\pi,v})(1)^{++}:= \Hom(V_{\pi,v}^{-}, V_{\pi, v}^{+})$, and define
\beq\lb{ev}
e_{v}(V_{(\pi^{\iota}, \chi^{\iota})})=
 {\prod_{w\vert v}\gamma(\iota{\rm WD}[V^{+}_{\pi, v|G_{E, w}} \ot V_{\chi, w}], \psi_{E,w})^{-1}
 \over \gamma(\iota{\rm WD}[\ad(V_{\pi,v})(1)^{++}], \psi_{v})^{-1}}
 \cd \cL( V_{(\pi^{\iota}, \chi^{\iota}),v})^{-1}, \eeq
where  $\iota {\rm WD}$ is the functor from potentially semistable Galois representations to complex Weil--Deligne representations of \cite{fontaine}, the inverse Deligne--Langlands $\gamma $-factor is $\gamma(W, \psi_{v})^{-1} = {L(W)/ \eps(W, \psi_{v}) L(W^{*}(1))}$,\footnote{The normalisations of $L$- and $\eps$-factors are as in \cite{tate-nt}.} and  $\psi_{E, w}=\psi_{v}\circ \Tr_{E_{w}/F_{v}}$.

Let $k_{0}\in \Z$ be such that the archimedean component of $\omega=\omega_{\pi}\omega_{\chi}$ is $\omega_{\infty}(x)=x^{k_{0}}$. We define 
\beqq e_{\infty}(V_{(\pi^{\iota}, \chi^{\iota})}) = i^{k_{0}[F:\Q]},
\eeqq
and
\beq\lb{epinf}
 e_{p\infty}(V_{(\pi^{\iota}, \chi^{\iota})}) :=  e_{\infty}(V_{(\pi^{\iota}, \chi^{\iota})}) \cd \prod_{v\vert p}e_{v}(V_{(\pi^{\iota}, \chi^{\iota})})
\eeq

 \subsubsection{Hida families}\lb{sec HF} Let $U_{\G}^{p}\subset \G(\A^{p\infty })$ be any open compact subgroup, and for $R=\Z_{p}, \Q_{p}$, let  ${\bf T}^{\rm sph, \ord}_{U_{\G}^{p}, R}$ be the $p$-(nearly) ordinary spherical Hecke $R$-algebra acting on ordinary $p$-adic modular cuspfoms for $\G$ of tame level $U_{\G}^{p}$.
 
   A \emph{cuspidal Hida family} $\X_{\G}$ is an irreducible component of the space $\Y_{\G, U_{\G}^{p}}:=\Spec {\bf T}_{U_{\G}^{p}, \Q_{p}}^{\rm sph, \ord}$ for some $U_{\G}^{p}$. It is a scheme finite flat over $\Spec \Z_{p}\llb T_{1}, \ldots, T_{[F:\Q]+1+\delta_{F,p}}\rrb\ot_{\Z_{p}}\Q_{p}$ (where $\delta_{F, p}$ is the $p$-Leopoldt defect of $F$), coming with a dense   ind-finite subscheme $$\X_{\G}^{\cl}\subset \X_{\G}$$ 
 of \emph{classical points},  and a locally free sheaf $\cV_{\G}$ of rank~$2$ endowed with an $\OO_{\X_{\G}}$-linear action of $G_{F}$. To each $x\in \X_{\G}^{\cl}$ is associated an automorphic representation $\pi_{x}$ of $\G(\A^{\infty})$ over $\Q_{p}(x)$, and the fibre $\cV_{\G |x}\cong V_{\pi_{x}}$.  The (numerical) weight of $x$ is  defined to be the weight of $\pi_{x}$.
 
 Let $U^{p}_{\H}\subset \H(\A^{p\infty})$ be an open compact subgroup. We define 
\beq \lb{YH}
\Y_{\H}=\Y_{\H, U_{\H}^{p}}:=\Spec \Z_{p}\llb \H(F)\bks \H(\A^{p\infty})/U_{\H}^{p}\rrb\ot_{\Z_{p}}\Q_{p},\eeq
where the topology on $\H(F)\bks \H(\A^{p\infty})/U_{\H}^{p}$ is profinite; it comes with a universal character $\chi_{\rm univ}\colon \H(F)\bks\H(\A^{\infty})\to \OO(\Y_{\H})^{\ts}$, identified with a $G_{E}$-representation $\cV_{\H}$ of rank~$1$, and a dense ind-finite  subscheme $\Y_{\H}^{\rm cl}\subset \Y_{\H}$, whose points $y$ correspond to $U_{ \H}^{p}$-invariant locally algebraic  Hecke characters $\chi_{y}$ of $\H$ over $\Q_{p}(y)$.  The weight of $y$ is  defined to be the weight of $\chi_{y}$.

Finally, the ordinary eigenvariety for $\G\ts\H$ of level $U^{p}_{\G}\ts U^{p}_{\H}$ is 
$$\Y_{\G\ts\H}= \Y_{\G\ts\H, U^{p}_{\G}\ts U^{p}_{\H}} := \Y_{\G}\hat{\ts} \Y_{\H} := 
\Spec {\bf T}_{U_{\G}^{p}, \Z_{p}}^{\rm sph, \ord}
\hat{\ot}_{\Z_{p}}
 \Z_{p}\llb \H(F)\bks \H(\A^{p\infty})/U_{\H}^{p}\rrb\ot_{\Z_{p}}\Q_{p}.$$
 Its subset of classical points is $\Y_{\G\ts \H}^{\cl}:=\Y_{\G}^{\cl}\ts \Y_{\H}^{\cl}$.
 A {Hida family} for $\G\ts\H$ is an irreducible component of $\Y_{\G\ts\H, U^{p}_{\G}\ts U^{p}_{\H}}$
for some $U^{p}_{\G}$, $U^{p}_{\H}$. 

We now isolate an interesting subspace of  $\Y_{\G\ts\H}$.
 Let $\omega_{\G}\colon F^{\ts}\bks \A^{\infty, \ts}\to \OO(\Y_{\G})^{\ts}$ be the character giving the action of the centre of $\G(\A)$ on $p$-adic modular forms,  let $\omega_{\H}:=\chi_{\rm univ|\A^{\infty, \ts}}$, and let 
 $$\omega:= \omega_{\G}\omega_{\H} \colon F^{\ts}\bks \A^{\infty,\ts}\to \OO(\Y_{\G\ts\H})^{\ts}.$$ 
 The \emph{self-dual locus} 
$$\Y_{\G\ts\H}^{\rm sd}\subset \Y_{\G\ts\H}$$
is the closed subspace defined by $\omega=\one$. If $\X$ is a Hida family for $\G\ts\H$, 
we denote $\X^{\cl}:= \X\cap \Y_{\G\ts\H}^{\cl}$,  $\X^{\rm sd}
:= \X\cap \Y_{\G\ts\H}^{\rm sd}$, and $\X^{\cl, {\rm sd}}:= \X^{\cl}\cap \X^{\rm sd}$.

\subsubsection{Main theorem}
 Throughout this paper, if $\X$ is a scheme over a  characteristic-zero field $L$, we  identify a geometric point $x\in \X(\C)$ with a pair $(x_{0}, \iota)$, where $x_{0}\in \X$ is the scheme point image of (as a synonym, \emph{underlying}) $x$ 
 and $\iota\colon L(x_{0})\into \C$ is an embedding.
 If $\X$ is integral, we denote by $\cK(\X)$ the local ring of the generic point, which we call the field of meromorphic functions on $\X$.
 
If $\X$ is a Hida family for $\G\ts\H$, we define $\X^{\cl, {\rm wt}}\subset \X^{\cl}$ to be the  subset of points $(x_{0}, y_{0})$ whose contracted weight $(k_{0}, w, l)$ satisfies
\beq \lb{wt ineqq}
|l_{v}| \leq w_{v}-2, \qquad |k_{0}| \leq w_{v}-| l_{v}|-2\qquad \text{for all $v \in \Sg_{p}$}.
\eeq
We denote $\X^{\cl,{\rm sd},  {\rm wt}}:= \X^{\cl, {\rm sd}}\cap \X^{\cl, {\rm wt}}$.

\begin{theoA} \lb{thm A} Let $\X$ be a Hida family for $\G\ts \H$ whose self-dual locus $\X^{\rm sd}$ is non-empty.  There exists a unique meromorphic function 
 $$\cL_{p}(\cV)\in \cK(\X),$$
 whose polar locus $\mathscr{D}$  does not intersect $\X^{\cl, {\rm sd}, {\rm wt}}$, 
 such that for each $(x, y)\in \X^{\cl, {\rm wt}}(\C)- \mathscr{D}(\C)$  
 we have 
 \beq
\label{interpol Lp intro} 
\calL_{p}(\cV)(x, y) =
e_{p\infty}(V_{(\pi_{x}, \chi_{y})}^{})
 \cdot  \calL(V_{(\pi_{x} ,\chi_{y})}, 0).
 \eeq
 
Here, if $(x_{0}, y_{0})\in \X^{\cl}$ is the  point underlying $(x,y)$ and $\iota\colon \Q_{p}(x_{0}, y_{0}) \into \C$ is the corresponding embedding, we have denoted $\pi_{x} =\pi_{x_{0}}^{\iota}, \chi_{y}=\chi_{y_{0}}^{\iota}$, and the interpolation factor is
as in \eqref{epinf}.
  \end{theoA}
  
The value of the  interpolation factor agrees with the general conjectures of Coates  and Perrin-Riou (see \cite{coates}). 
(The notation $\cV$ is meant to evoke some `universal virtual Galois representation interpolating \eqref{virtual}'.)

\subsubsection{Previous related work} \lb{previous} When $E/F$ splits above $p$, Theorem \ref{thm A}  may be essentially deduced from the main result of  \cite{Hi} (see also \cite{quad}). Hida's method uses the Rankin--Selberg integral, whereas ours uses Waldspurger's variant  \cite{wald}  based on the Weil representation (as discussed below). 

The numerator of our $L$-value is a special case of the standard $L$-function for $\GL_{2}\ts \GL_{1}$ over $E$, and when so considered our $p$-adic $L$-function is a multiple of the restriction to some base-change locus of  one constructed by Januszewski  \cite{janu}  using the method of modular symbols; however that function is not uniquely characterised by its interpolation property, which involves unspecified periods.

Finally, when $F=\Q$,  variants of $\cL_{p}(\cV)$ were constructed by Hida \cite[Theorem 5.1a]{hida-ii} and, more recently,  by Loeffler and  B\"uy\"ukboduk--Lei (see \cite[\S~B.4]{BL}) under various local restrictions.

\subsection{Idea of proof, organisation of the paper, and discussion of the method} \lb{12}
The proof combines the strategy  of  Hida \cite{Hi} with an enhanced version of that of \cite[Proof of Theorem A]{pyzz}, where we had constructed the `slices' $\cL_{p}(\cV)(x, -)$ for $x\in \X_{\G}^{\cl}$ of weight~$2$. 

We start from Waldspurger's  \cite{wald} integral representation of Rankin--Selberg type
\beq \lb{WRS}
(f, I(\phi, \chi)) = \cL(V_{(\pi, \chi)}, 0) \cdot \prod_{v} R_{v}(W_{v},\phi_{v},\chi_{v}) \eeq
where $( \ , \ )$ is a normalised  Petersson product, $f$ is a form in $\pi$ with Whittaker function $\ot_{v}W_{v}$, the form $I(\phi, \chi)$ is a mixed theta-Eisenstein series depending on a certain Schwartz function $\phi$, and the $R_{v}$ are normalised local integrals.

 In \S~\ref{sec 2}, we discuss the general setup. In \S~\ref{sec 3}, we make a judicious choice of $\phi_{v}$  at the places $v\vert p\infty$ and interpolate the ordinary projection of ${I}(\phi, \chi)$ into a $\Y_{\G\ts\H}$-adic modular form. In \S~\ref{sec 4}, we  interpolate $R_{v}$ for $v\nmid p\infty$ using sheaves of local Whittaker functions over $\X$ provided by the local Langlands correspondence in families (\S~\ref{sec 44}); we compute $R_{v}$ for $v\vert p\infty$ (\S~\ref{sec 43}), which yield the interpolation factors in \eqref{interpol Lp intro}; and finally (\S~\ref{sec 45}), we  use \eqref{WRS} to define $\cL_{p}(\cV)$ as a glued quotient of the global and local (away from $p\infty$) families of zeta integrals.

In Appendix \ref{double fact}, we give a TV-inspired  bijective proof of a combinatorial lemma occurring in \S~\ref{sec eis}.

\medskip

The method of constructing $p$-adic $L$-functions as ratios of arbitrary matching families of global and local zeta integrals should be applicable whenever an integral representation for the corresponding complex $L$-function is available, at least if the groups involved are products of general linear groups: for example, for Rankin--Selberg  $L$-functions. It can be compared to the `hard' constructions from much of the existing literature on $p$-adic $L$-functions, which rely on fine choices of local data at all places, computation of the associated integrals, and bounds on the ramification of the data (see \cite{hsieh 3} for an excellent example of the state of the art). To be sure, the two approaches should be viewed as complementary rather than alternative: while the `soft' construction provides a  flexibility useful for some applications (such as in \cite{pyzz}), explicit choices and computations can still be plugged into it, and are likely still indispensable to address finer issues such as integrality.

For  another brief  general discussion of our method focused on the role of the local Langlands correspondence in families (LLCF), as well as  some results on local interpolation, we refer to  \cite[\S~1.2, \S~5]{LLC};\footnote{In Appendix \ref{app B} we correct a couple of blundered statements from \cite{LLC}.}  see also the very recent work of Cai--Fan \cite{Cai} for a related study  in the context of periods attached to spherical varieties.  Abelian antecedents of the construction, for which the LLCF is not needed,  can be found in \cite{LZZ} and \cite{pyzz}.

The local-global approach may in principle introduce poles coming from zeros of the families of local integrals. In our specific setup, the Waldspurger local integrals are not easy to control (at least for this author) away from the selfdual locus. This is why Theorem \ref{thm A}, while sufficient for the arithmetic applications in \cite{univ}, is not as strong as it could be: one may at least expect that the condition that $\X^{\rm sd}$ be non-empty is superfluous, and that the polar locus of $\cL_{p}(\cV)$ should not intersect $\X^{\cl, {\rm wt}}$. 
 As noted by a referee, approaching $\cL_{p}(\cV)$ via the well-understood Rankin--Selberg integrals  for $\GL_{2}\ts \GL_{2}$ would likely yield such a strengthening.

\subsection*{Acknowledgments} I am grateful to David Callan,  Haruzo Hida, Ming-Lun Hsieh and Xinyi Yuan  for useful correspondence, and to the referees for a very careful reading.

\section{$p$-adic modular forms and Hida families}
\lb{sec 2}
The material of this section is largely due to Hida (see \cite[\S\S~1--3, 7]{Hi} and references therein).

\subsection{Notation and preliminaries}
The  notation introduced in the present subsection (or in the introduction)
 will be used throughout the paper unless otherwise noted; in particular, the groups $\G$ and $\H$ defined in \eqref{grps}.
\subsubsection{General notation} The following notational choices are largely standard.
\begin{itemize}
\item The fields $F$ and $E$ are as fixed in the introduction unless specified otherwise; if $*$ denotes a place  of $\Q$ or a finite set thereof, we denote by $S_{*}$ the set of places of $F$ above $*$;
\item we denote by $D_{F}$, $D_{E}$ and $D_{E/F}$, respectively, the absolute discriminants of $F$ and  $E$ and relative discriminant of $E/F$; for a finite place $v$ of $F$  we denote by $d_{v}\in F_{v}$ a generator of the different ideal of $F$ and by $D_{v}\in F_{v}$ a generator of the relative discriminant ideal; 
\item we denote by $<$ the partial order on $F$ given by $x<y$ if and only if $\tau(x)<\tau(y)$ for all $\tau\in \Sg_{\infty}$; we denote $\R^{+}:= \{x\in \R\, | \, x>0\}$ and $F^{+}:=\{x\in F\, |\, x>0\}\subset F^{\ts}$;
\item $\A$ is the ring of ad\`eles of $F$; if $S$ is a finite set of places of a number field $F$, we denote by $\A^{S}=\prod'_{v\notin S} F_{v}$, and $F_{S}:= \prod_{v\in S}F_{v}$; when $S$ consists of the set of places of $F$ above some finite set of places of $\Q$ (for instance the place $p$) we use the same notation with those places of $\Q$ instead of $S$ (for instance, $F_{p}=F_{S_{p}}$). We denote $F_{\infty}^{+}=\{x\in F_{\infty}\ | \ x_{\tau}>0 \text{ for all $\tau\in \Sg_{\infty}$}\}$ and  $\A^{+}:=\A^{\infty}\ts F_{\infty}^{+}$. 
\item we denote by $\psi\colon F\bks \A\to \C^{\ts}$ the standard additive character as in \S~\ref{sec psi intro};
\item if $R/R_{0}$ is a ring extension, $A$ is an $R_{0}$-algebra, and $X$ is an $R_{0}$-scheme, we denote 
$$A_{R}:=A\ot_{R_{0}}R, \qquad X_{R}:= X\ts_{\Spec R_{0}}R.$$
\item we denote by $G_{K}$ the absolute Galois group of a field $K$;  
\item if $K$ is a finite extension of $F$, its class number is denoted by
$$h_{K}:= |K^{\ts}\bks  \A_{K}^{\infty, \ts}/\widehat{\OO}^{\ts}_{K}|$$ 
\item  for a place $v$ of  $F$, we denote by $\vpi_{v}$ a fixed uniformiser at $v$, by $q_{F,v} $ the cardinality of the residue field; we denote $q_{F, p}:=(q_{F, v})_{v\in S_{p}}$;
\item the class field theory isomorphism is normalised by sending uniformisers to geometric Frobenii; for $K$  a number field (respectively a local field), we will then identify characters of $G_{K}$ with characters of $K^{\ts}\bks {\A}_{K}^{\times}$ (respectively $K^{\times})$ without further comment;
\item If $I$ is a finite index set and $x=(x_{i})_{i}$, $y=(y_{i})$ are real vectors, we define $(xy)_{i}=x_{i}y_{i}$ and $x^{y}:= \prod_{i}x_{i}^{y_{i}} $ whenever that makes sense. Moreover, we  often identify an integer $w_{0}$ with the constant vector $(w_{0})_{i\in I}\in \Z^{I}$. 
\item we denote by $\one[\cd ]$ the $\{0, 1\}$-valued function on logical propositions such that $\one[\phi]=1$ if and only if $\phi$ is true;
\end{itemize}

\subsubsection{Subgroups of $\GL_{2}$ and special elements} 
We denote by $Z$, $A$, and $N$ respectively the centre, diagonal torus, and upper unipotent subgroup of $\G=\GL_{2/F}$; we let $P=AN$ and $P^{1}:=P\cap {\rm SL}_{2/F}$.  We define a map  ${\bf a}\colon \GL_{1/F}\to \G$
by $${\bf a}(y):= \twomat y{}{}1.$$
We denote by $$w:= \twomat {}1{-1}{}\in \GL_{2}(F)$$
or its image in $\GL_{2}(R)$ for any $F$-algebra $R$. (The context will prevent any confusion with the notation for the weights of $\G$.)
For $r\in \Z_{\geq 1}^{S_{p}}$,
we define
$$w_{r_{v},v}:= \twomat {}{1}{-\vpi_{v}^{r_{v}}}{}\in \GL_{2}(F_{v}), \qquad w_{r, p}:= \prod_{v\vert p} w_{r_{v},v}\in \GL_{2}(F_{p}),$$
as well as  a sequence of compact subgroups 
\begin{align*}
U_{v, r_{v}}:=U^{1}_{1}(\vpi_{v}^{r_{v}}) &:= \{ \smalltwomat abcd \in \GL_{2}(\OO_{F,v})\, :\,  a-1\equiv d-1 \equiv c\equiv 0 \pmod{\vpi_{v}^{r_{v}}}\}\subset \GL_{2}(F_{v})\\
U_{p, r}&:= \prod_{v\in S_{p}} U_{v, r_{v}}.
\end{align*}

For $\theta\in (\R/2\pi\Z)^{S_{\infty}}$, we denote $r_{\theta}:=\left( \twomat {\cos \theta_{v} }{\sin\theta_{v}}{-\sin \theta_{v}}{\cos \theta_{v}}\right)_{v}\in {\rm SO}(2, F_{\infty})$

\subsubsection{Hecke algebras}\lb{213}
Let $S$ be a finite set of non-archimedean places of $F$ and let  $U^{S}=\prod_{v\notin S}U_{v}\subset \GL_{2}(\A^{S\infty})$ be an open compact subgroup. For each finite set of finite places $S$, we define the Hecke algebra
$$\cH_{U^{S}}:=C_{c}(U^{S}\bks \G(\A^{S\infty}) /U^{S}, \Z).$$
It carries an involution 
\beq\lb{inv Hk}
T\mapsto T^{\curlyvee}\eeq
arising from the map $g\mapsto g^{-1}$ on the group $\G$. 

Let $A_{p}:= A(F_{p})\subset \G(F_{p})$ be the diagonal torus, and let $A_{p}^{+}$ be the set of $t=\smalltwomat {t_{1}} {}{} {t_{2}}$ such that $v(t_{1}) \geq v(t_{2})$ for all $v\vert p$. The involution 
\beq \lb{inv A}
t\mapsto t^{\curlyvee}:= \det(t)^{-1} t \eeq
preserves $A_{p}^{+}$. For $S$ a finite set of places of $F$ disjoint from $S_{p}\cup S_{\infty}$, we define the ordinary Hecke algebra 
$$\cH_{U^{Sp}}^{\ord}:= \cH_{U^{Sp}, \Z_{p}}\ot \Z_{p}[A_{p}] $$
over $\Z_{p}$, which will act on spaces of ordinary modular forms (here and in the rest of the text, a subscript $Sp$ is shorthand for $S\cup S_{p}$).  It is endowed with the involution $\curlyvee$ deduced from \eqref{inv Hk} and  \eqref{inv A}.

If $U=\prod_{v}U_{v}\subset \G(\A^{\infty})$, respectively $U^{p}=\prod_{v\nmid p}U_{v}\subset \G(\A^{p\infty})$ are open compact subgroups, and $S$, respectively $S^{p}$, is the set of places such that $U_{v}$ is not maximal, we define
$$\cH_{U}^{\rm sph}:= \cH_{U^{S}}, \qquad \cH_{U^{p}}^{\rm sph, \ord}:= \cH_{U^{Sp}}^{\ord}.$$
(These depend on $S$, but their images in endomorphisms rings of spaces of modular forms does not.)

\subsubsection{Measures}  We use the same notation and conventions for Haar measures and integration as in \cite[\S~1.6]{yzz} and \cite[\S~1.9]{pyzz}.  In particular, we have a regularised integration functional  
$$\int^{*}_{E^{\ts}\bks {\A}_{E}^{\ts} /\A^{\ts}} f(t) \, dt , $$
which satisfies the following. 
\begin{lemm} \lb{int *}
Let $f$ be a smooth function on $\A_{E}^{\ts}$ that is invariant under $E_{\infty}^{\ts}$. Let $\mu\subset \OO_{E}^{\ts}$ be a finite index subgroup fixing $f$ (under the scaling action).
Then 
$$\int^{*}_{E^{\ts}\bks {\A}_{E}^{\ts} /\A^{\ts}}  \sum_{x\in E^{\ts}}  f(xt) \, dt 
= { 2L(1,\eta)\over h_{E}^{}}   { [\OO_{E}^{\ts}:\mu_{}] } \int_{ \A_{E}^{\infty, \ts}}  \sum_{\alpha\in \mu_{}} 
f(\alpha t)\, d^{\bullet}t,$$
where $d^{\bullet} t  $ is the  Haar measure giving volume $1$ to $ \widehat{\OO}^{\ts}_{E}$.
\end{lemm}
\begin{proof} Let $U\subset \A_{E}^{\infty, \ts}$ be any compact open subgroup fixing $f$. Since both sides are independent of $\mu$, we may assume  that   $\mu= \OO_{E}^{\ts}\cap U$. 
 By \cite[(1.6.1) and following paragraphs]{yzz}, we have 
\beq
\lb{derer}
\int^{*}_{E^{\ts}\bks {\A}_{E}^{\ts} /\A^{\ts}} f(t) \, dt
= { \vol ({E^{\ts}\bks {\A}_{E}^{\ts} /\A^{\ts}})}\ \dashint_{E^{\ts}\bks {\A}_{E}^{\ts} /\A^{\ts}} f(t) \,dt
= {\vol ({E^{\ts}\bks {\A}_{E}^{\ts} /\A^{\ts}}) \over |E^{\ts}\bks  \A_{E}^{\infty, \ts}/U | } 
\sum_{t\in E^{\ts}\bks  \A_{E}^{\infty, \ts}/U}f(t).
\eeq
Now  by a coset identity, 
$$ |E^{\ts}\bks  \A_{E}^{\infty, \ts}/U | =
 h_{E} { |\widehat{\OO}^{\ts}_{E}/ U| \over [\OO_{E}^{\ts}:\mu_{}]}$$
 and by \cite[\S 1.6.3]{yzz},  $\vol ({E^{\ts}\bks {\A}_{E}^{\ts} /\A^{\ts}})=2L(1, \eta)$. 
Hence \eqref{derer} equals 
$$ {2L(1, \eta)\over h_{E}}
{ [\OO_{E}^{\ts}:\mu_{}]  \over  |\widehat{\OO}^{\ts}_{E}/ U|}
\sum_{t\in E^{\ts}\bks  \A_{E}^{\infty, \ts}/U} f(t)$$
 If we compose with the operator $f(\cd)\mapsto \sum_{x\in E^{\ts}} f(x\cd) = \sum_{x\in \mu_{}\bks E^{\ts}} \sum_{\alpha\in \mu_{}}  f(\alpha x\cd)$, we obtain
 $$ 
{ 2L(1,\eta)\over h_{E}^{}}
{ [\OO_{E}^{\ts}:\mu_{}]  \over  |\widehat{\OO}^{\ts}_{E}/ U|}
 \sum_{ t\in \A_{E}^{\infty, \ts}/U}  \sum_{\alpha\in \mu_{}}f(\alpha t) = 
{ 2L(1,\eta)\over h_{E}^{}}  
{ [\OO_{E}^{\ts}:\mu_{}] }
 \int_{ \A_{E}^{\infty, \ts}}  \sum_{\alpha\in \mu_{}} f(\alpha t)\, d^{\bullet}  t .$$
 \end{proof}

\subsection{Modular forms and their $q$-expansions}
Let $\frak h\subset \C$ be the upper half-plane. We  view $\G(F_{\infty})$ as acting on $\frak h^{\Sg_{\infty}}$ by M\"obius transformations, and identify 
$$C^{+}_{\infty}:= (\R_{+ }{\rm SO}(2, \R))^{{\Sg_{\infty}}}\subset \G(F_{\infty})$$
with the neutral connected component  of the stabiliser of ${\rm i} := (\sqrt{-1}, \ldots, \sqrt{-1})\in \frak h^{{\Sg_{\infty}}}$.

\subsubsection{Nearly holomorphic modular forms} Let $\uw$ be an ($\infty$-adic, \S~\ref{sec 111}) weight for $\G$,   let $U\subset \G(\A^{\infty})$ be a compact open subgroup, and let $m=(m_{\tau})\in \Z_{\geq 0}^{\Sg_{\infty}}$ . A complex \emph{nearly holomorphic (Hilbert) modular form} of weight $\uw$,  level $U$, degree $\leq m$ is a function 
$$f \colon  \G(\A) \to \C$$
satisfying the following two conditions:
\begin{enumerate}
\item for all $g\in \G(\A)$, $\gamma\in \G(F)$, $k \in U C_{\infty}^{+}$,
$$f(\gamma g k) = j_{\uw}(k_{\infty}, {\rm i})^{-1}f(g),$$
where
 for $z\in \frak h^{{\Sg_{\infty}}}$,
$$j_{\uw}({\smalltwomat abcd}, z_{}) := (ad-bc)^{(w_{0}-w)/2} (cz+d )^{w}.$$
\item  There is a Whittaker--Fourier expansion
\beq \lb{FW f}
f\left(\twomat yx{}1\right) &= |y| \sum_{a\in F} W^{\C}_{f,a}(y)(Y)  \,  {\qqq}^{a}
\eeq
for all $y\in {\A}^{+}$, $x\in {\A}$, where:
\begin{itemize}
\item we have
 $$W_{f,0}^{\C}(y)= y_{\infty}^{(w_{0}+w-2)/2} \rW_{f, 0}(y) , \qquad W_{f,a}^{\C}(y) =
 (ay_{\infty})^{(w_{0}+w-2)/2}   \rW_{f, a}(y) \quad (a\neq 0)$$ 
for   polynomials 
 $$\rW_{f, a}(y)\in \C[(T_{\tau})_{\tau\in \Sg_{\infty}}]$$ of degree $\leq m_{\tau}$ in the variables $T_{\tau}$, 
 evaluated at  $Y:=(Y_{\tau})_{\tau\in \Sg_{\infty}}$ with $Y_{\tau}= (4\pi y_{\tau})^{-1}$; 
\item we denote
$$\qqq^{a}  := \psi(ax) \psi_{\infty}({\rm i} a y_{\infty}).$$
\end{itemize}
\end{enumerate}
The polynomial  $\rW_{f,a}(y)$ only depends on the class of $ay$ modulo ${\bf a}^{-1}(U)$, so that defining 
\beq \lb{Wfa} 
\rW_{f}(a):= W_{f, 1}(a)
\eeq
 for $a\in \A^{+}$, for all $a\in F^{+}$ and $y\in \A^{+}$ we have
$W_{f,a}(y)=W_{f} (ay)$. 
We say that $f$ is \emph{cuspidal} if $W_{0}(y)=0$ for all $y$.

If $f$ is nearly holomorphic of degree $0$ (that is, $\leq (0, \ldots, 0)$), we simply say that $f$ is a (holomorphic) modular form.    If $R\subset\C$ is any subring, we denote by  
$$S_{\uw}(U, R)\subset M_{\uw}(U, R)\subset N_{\uw}^{ \leq m}(U, R)$$
respectively the  spaces of  cuspidal forms and  holomorphic forms of  level $U$ and weight $\uw$, and  of  nearly holomorphic forms of level $U$, weight $\uw$,  and degree $\leq m=(m_{\tau})$, such that for all $a\in \A^{+}$, the polynomials $W_f({a})$ have coefficients in $R$. We write $N_{\uw}^{}(U, R):= \varinjlim_{m} N_{\uw}^{ \leq m}(U, R)$, and 
$\Box(R):=\varinjlim_{U} \Box(U, R)$
if $\Box$ stands for the notation for any of the spaces of forms defined above (or below).

Finally, we define the space $S^{\rm a}_{\uw}(U, \C)$ of \emph{antiholomorphic cuspforms} to be the $\C$-vector space  image of  $S_{\uw}(U, \C)$ under complex conjugation. The formula
\beq \lb{hol anti}
f\mapsto f^{\rm a}:=\twomat {-1}{}{}{1} f 
\eeq
(where $\smalltwomat  {-1}{}{}{1} \in \G(F)\subset \G(\A)$ acts, as usual, by right translation)
defines  $\C$-linear bijections from $S_{\uw}(U, \C)$ to $S^{\rm a}_{\uw}(U, \C)$ and viceversa. 

\subsubsection{Twisted modular forms}  \lb{TMF}
A \emph{twisted {nearly holomorphic (Hilbert) modular form}} of weight $\uw$,  level $U$, degree $\leq m=(m_{\tau})$, is a function 
$$f \colon \G(\A)\ts \A^{\ts} \to \C$$
satisfying the following two conditions:
\begin{enumerate}
\item for all $g\in \G(\A)$, $\gamma\in \G(F)$, $k \in U C_{\infty}^{+}$,
$$f(\gamma g k, \det(\gamma)^{-1} u) = j_{\uw}(k_{\infty}, {\rm i})^{-1} f(g, u);$$
\item  there is a  Whittaker--Fourier expansion 
\beq 
f\left(\twomat yx{}1 , u\right) &= |y| \sum_{a\in F} W^{\C}_{f, a}(y, u)(Y)  \,  {\qqq}^{a} 
\eeq
for all  $x\in {\A}$ and $y, u\in {\A}^{\ts}$ such that $(uy)_{\infty}>0$, where:
 $$W_{f, 0}^{\C}(y, u)= y_{\infty}^{(w_{0}+w-2)/2} \rW_{f, 0}(y, u) , \qquad W_{f, a}^{\C}(y, u) =
 (ay_{\infty})^{(w_{0}+w-2)/2}  \rW_{f, a}(y, u) \quad (a \neq 0)$$ 
for   polynomials 
 $$ \rW_{f, a}(y, u)\in \C[(T_{\tau})_{\tau\colon F\into \R}]$$ of degree $\leq m_{\tau}$ in the variables $T_{\tau}$, 
 evaluated at  $Y:=(Y_{\tau})_{\tau\colon F\into \R}$ with $Y_{\tau}= (4\pi y_{\tau})^{-1}$. 
\end{enumerate}

    If $R\subset\C$ is any subring, we denote by  
$ M^{\rm tw }_{\uw}(U, R)\subset N_{\uw}^{{\rm tw}, \leq m}(U, R)$
 the  spaces of holomorphic and   nearly holomorphic forms of level $U$, weight $\uw$,  and degree $\leq m=(m_{\tau})$, such that 
  all the polynomials $\rW_{f, a}(y, u)$ have coefficients in $R$.

\subsubsection{Contracted product}
 For any open compact subgroup  $U_{F}\subset \widehat{\OO}_{F}^{\ts}$,  let
\beq\lb{cUF} \mu_{U_{F}}:= F^{\ts}\cap U_{F}, \qquad \nu_{U_{F}}:=|\{\pm 1\} \cap \mu_{U_{F}}|, \qquad \qquad c_{U_{F}}= {\nu_{U_{F}} \cd 2^{[F:\Q]}h_{F}\over [\OO_{F}^{\ts}:\mu_{U_{F}}^{2}]}.
\eeq
Let $\vphi\colon \A^{\ts}\to \C$ be a Schwartz function, invariant under a subgroup of the form $\mu_{U_{F}'}^{2}\subset F^{\ts}$ as above. Then the sum
\beq \lb{sum star}
\sum^{\star}_{u\in F^{\ts}} \vphi(u):=  c_{U_{F}}\sum_{u\in \mu_{U_{F}}^{2}\bks F^{\ts}}  \vphi(u) \eeq
is well-defined independently of $U_{F}\subset U_{F}'$, and for any such choice the support of the sum is finite.

If $f_{1}$, $f_{2}$ are twisted nearly holomorphic forms, 
we may thus define a (plain) nearly holomorphic form $f_{1} \star f_{2}$ by
\beq\lb{star def}
f_{1} \star f_{2}(g):= \sum_{u\in F^{\ts}}^{\star} f_{1}(g, u)f_{2}(g, u).
\eeq

\subsubsection{Differential operators}
We attach to  a nearly holomorphic (genuine or twisted) form $f$  the function  
\beqq
f^{\frak h}\colon \G(\A^{\infty})\ts \frak h^{{\Sg_{\infty}}} &\to \C \\
(g^{\infty}, z=g_{\infty} {\rm i})& \mapsto j_{\uw}(g_{\infty}, {\rm i}) f(g_{\infty});
\eeqq
the map $f\mapsto f^{\frak h}$ is injective.

The Maass--Shimura differential operators on functions on $\frak h^{{\Sg_{\infty}}}$ are defined as follows. For $\tau\colon F\into \R$, $w \in \Z$, let 
$$\delta_{w}^{\tau, \frak h}:= {1\over 2\pi i} \left({w\over 2i y_{\tau}}+ {\partial \over \partial z_{\tau}}\right), \qquad d^{\tau}:= {1\over 2\pi i} {\partial\over \partial z_{\tau}},$$
a differential operator on the upper half-plane $\frak h$. 
For $w,k \in \Z_{\geq 0}^{{\Sg_{\infty}}}$, let 
$$\delta_{w}^{k, \frak h}:=\prod_{\tau} \delta^{\tau, \frak h}_{w_{\tau}+2k_{\tau}}\circ \cdots \circ \delta_{w_{\tau}+2}^{\tau, \frak h}\circ \delta_{w_{\tau}}^{\tau, \frak h}, \qquad d^{k}:=\prod_{\tau}(d^{\tau})^{k_{\tau}}.$$
Then for any ring  $\Q\subset R\subset \C$, this  operator defines a map
$$\delta_{w}^{k}\colon N_{\uw}^{{\rm (tw)} ,\leq m}(U,R)\to N_{\uw+(0;2k)}^{{\rm (tw)}, \leq m+k }(U, R)$$
such that $\delta_{w}^{k}(f)^{\frak h}=\delta_{w}^{k, \frak h}(f^{\frak h})$.  (For a proof of the intuitive fact that  the archimedean operator $\delta_{\uw}^{k}$ indeed preserves the rationality properties of finite Whittaker--Fourier coefficients, see \cite[Proposition 1.2]{Hi}, whose calculations  also apply to the twisted case.) The subscript $w$ will be omitted if it is clear from the context.

By \cite[(1.16)]{shimura}, for all $k\in \Z_{\geq 0}^{\Sg_{\infty}}$ we have
\beq \lb{delta d}
\delta_{w}^{k}=\sum_{0 \leq j\leq k}  \prod_{\tau\in \Sg_{\infty}} {k_{\tau}\choose j_{\tau}} {\Gamma(w_{\tau}+k_{\tau})\over \Gamma(w_{\tau}+j_{\tau})} (-4\pi y_{\tau})^{j_{\tau}-k_{\tau}} \, d^{j_{}}.
\eeq

If $w\geq 2m+1$, any $f\in N_{\uw}^{\rm (tw),  \leq m}(U, R)$ can be written uniquely as 
$$f=\sum_{0\leq r\leq m} \delta_{w-2r}^{r}f_{r}$$
with $f_{r}\in M_{\uw+(0; -2r)}^{\rm (tw)  }(U, R)$. (The proof in \cite[Lemma 7]{sh76} carries over to our context.)
Thus the linear map 
\beq \lb{ehol}
e^{\rm hol}\colon  N_{\uw}^{\rm (tw),  \leq m}(U, R) &\to M_{\uw}^{\rm (tw)}(U, R)\\
f&\mapsto f_{0}\eeq
is well-defined.

\subsection{$p$-adic modular forms}  We study the completions of spaces of modular forms for certain $p$-adic norms.
\subsubsection{Arithmetic $q$-expansion} Let  Let $\uw$ be a weight for $\G$ and  let $U\subset \G(\A^{\infty})$ be a compact open subgroup.
The \emph{$q$-expansion} map
$$f\mapsto (a\mapsto \rW_{f}(a)=\eqref{Wfa})$$
sends $S_{\uw}(U,\C)$ to  $ \C^{\A^{+}/U_{F}F_{\infty}^{+}}$,
 where $U_{F}={\bf a}^{-1}(U)$.
  By the $q$-expansion principle (see \cite[Proposition 2.1.1]{pyzz} for a version in our setting), the map is injective. We denote its image by ${\rm S}_{\uw}(U,\C)$ and view the map $S_{\uw}(U,\C)\to {\rm S}_{\uw}(U,\C)$ as an identification.

 If $R$ is any ring admitting embeddings into $\C$, we denote by  
$${\rm S}_{\bullet}(U, R) \subset  R^{\A^{+}/U_{F}F_{\infty}^{+}}$$
the set of those sequences 
\beq \lb{wff}
\rf=(\rW_{\rf}(a))_{a}\eeq
such that for any $\iota\colon R\into \C$,
the sequence $\rf^{\iota}:=(\iota \rW_{\rf}(a))_{a}$ is the $q$-expansion of a cuspform 
 $$f^{\iota}\in {S}_{\bullet}(U, \C)=\bigoplus_{\uw} S_{\uw}(U, \C).$$
(In \eqref{wff}, the notation $\rW_{\rf}$ can be thought of as simply  synonymous to $\rf$; it is introduced in order to match the identification of the previous paragraph.)
  By \cite[Theorem 2.2 (i)]{Hi} (together with a consideration of Galois actions mixing the weights), for any such ring $R$ we  have ${\rm S}_{\bullet}(U, R)= {\rm S}_{\bullet}(U, \Z)\ot R$. For more general rings, the previous equality is taken to be the definition of  ${\rm S}_{\bullet}(U, R)$.

\subsubsection{$p$-adic modular forms} \lb{232}
Let $L$ be a finite extension of $\Q_{p}$ splitting $F$. A \emph{$p$-adic  $L$-valued (cohomological) weight} $\uw=(w_{0}, (w_{\tau})_{\tau\colon F\into L})$ is a tuple of integers, all having the same parity, such that $w_{\tau}\geq 1$ for all $\tau\colon F\into L$. As in \S~\ref{sec 111}, if $\uw$ is an $L$-valued weight and  $\iota\colon L\into \C$ is an embedding, we define the complex weight  $\uw^{\iota}=(w_{0}, (w_{\tau})_{\iota\circ \tau})$.

Let $U\subset \G(\A^{\infty})$  be a compact open subgroup, and let $\uw $ be an $L$-valued weight. We define ${\rm S}_{\uw}(U, L)$ to be the set of $q$-expansions $\rf$ such that for every $\iota\colon L\into \C$ the expansion $\rf^{\iota}$ belongs to ${\rm S}_{\uw^{\iota}}(U,\C)$. The \emph{$p$-adic $q$-expansion} of $\rf=(\rW_{\rf}({a}))_{a}\in {\rm S}_{\uw}(U, L)$ is the sequence 
$$f=(W_{f}(a)):=(W_{\rf}(a)),\qquad W_{\rf}(a) :=  a_{p}^{(w_{0}+w-2)/2}  \rW_{\rf}({a}),$$
so that 
\beq\lb{a p inf}
W_{f^{\iota},a}^{\C}(y):=(ay)_{\infty}^{(w_{0}+w^{\iota}-2)/2}  \, \iota\left( (ay_{p})^{(-w_{0}-w+2)/2} W_{f}(a)\right) 
\eeq
is the Whittaker--Fourier coefficient of $f^{\iota}$ as in \eqref{FW f}.  (In other words, we have  two embeddings ${\rm S}_{\uw}(U, L)\into L^{\A^{+}/U_{F}F_{\infty}^{+}}$: the $q$-expansion
 $\rf \mapsto (\rW_{\rf}(a))_{a}$, and the $p$-adic $q$-expansion $\rf\mapsto (W_{\rf}(a))_{a}$.)

Let $U^{p}\subset \G(\A^{p\infty})$ be a compact open subgroups, let $U_{F}^{p}:={\bf a}^{-1}(U^{p})$, and for any $L$-valued weight $\uw$ denote
\beqq
{\rm S}_{\uw}(U^{p}, L) &:=\varinjlim_{n} {\rm S}_{\uw}(U^{p}U_{p, n}, L).
   \eeqq
 The space of cuspidal \emph{$p$-adic  modular forms}
$${\bf S}(U^{p}, L) \subset   L^{\A^{\infty,\times}/U_{F}^{p}}\subset   L^{\A^{\infty,\times}}=L^{\A^{+}/F^{+}_{\infty}}$$ 
is the completion of ${\rm S}_{\uw}(U^{p}, L)$ for the norm $||\rf||:= \sup_{a} |W_{\rf}(a)|$, for any $\uw$.
By a fundamental result of Hida (see \cite[paragraph after Theorem 3.1]{Hi}), the space ${\bf S}(U^{p}, L)$ is independent of the choice of~$\uw$. In particular, if $L$ is  Galois over $\Q_{p}$, this space is stable by the action of $\Gal(L/\Q_{p})$ and so it is of the form   ${\bf S}(U^{p}, \Q_{p})\ot_{\Q_{p}}L$ for a space ${\bf S}(U^{p}, \Q_{p})$.

\subsubsection{Nearly holomorphic forms as $p$-adic modular forms} We may attach a $p$-adic $q$-expansion to a nearly holomorphic form with coefficients in a $p$-adic subfield of $\C$. 

Let $L$ be a finite extension of $\Q_{p}$ and let $\uw$ be a $p$-adic $L$-valued weight. We say that 
$$f=(W_{f}(a))\in L^{\A^{+}/F_{\infty}^{+}}$$ 
is a \emph{$p$-adic nearly holomorphic cuspform} of weight $\uw$  and level $U^{p}\subset \G(\A^{p\infty})$ if the following condition holds.  For each $\iota\colon L\into \C$,
there exists a cuspidal nearly holomorphic form 
$$f^{\iota}\in N_{\uw^{\iota}}^{\leq \lfloor{(w+1)/2}\rfloor}(U^{p}U_{p, n},\C)$$ 
for some $n\in \Z_{\geq 1}^{S_{p}}$, whose Whittaker--Fourier polynomials have constant terms satisfying
\beq \lb{f nh p}
\rW_{f^{\iota}, 1}(a)(0)  = \iota\left(a_{p}^{(-w_{0}-w+2)/2}W_{f}(a)\right).
\eeq
The notion of a $p$-adic \emph{twisted} nearly holomorphic cuspform is defined similarly by the identity 
$\rW_{f^{\iota}, a}(y, u)(0)  = \iota\left((ay)_{p}^{(-w_{0}-w+2)/2}W_{f}(a)(y, u)\right)$.

\begin{prop} If $f$ is a  $p$-adic nearly holomorphic cuspform over $L$ of level $U^{p}$, then it belongs to  the space ${\bf  S}(U^{p},L)$ of $p$-adic modular cuspforms of level $U^{p}$.
\end{prop}
\begin{proof} This is the first assertion of \cite[Proposition 7.3]{Hi}.
\end{proof}

\subsubsection{Hecke operators and ordinary projection} 
The space  $N_{\uw}(U, \C)$ is endowed with  the usual action of $\cH_{U}$.  By writing down the effect of this action on Whittaker--Fourier coefficients of cuspforms, we may descend it  to a bounded action of $\cH_{U^{p}, L}$ on ${\rm S}_{\uw}(U^{p},L)$, hence  on ${\bf S}(U^{p}, L)$, for any $p$-adic field~$L$.

For $t\in A_{p}^{+}$ or $y\in \prod_{v\vert {p}}\OO_{F, v}-\{0\}$, and any $n\in \Z_{\geq 1} ^{S_{p}}$, define the double coset operators
\beq\label{aut-Up}
 \Up_{t}&:=[U_{p, n} \,  t \, U_{p, n}],  \qquad &\Up_{t}^{\circ, \uw} := t_{1}^{2-w}\det(t)^{(-w_{0}+w-2)/2} \Up_{t}, \\
 \Up_{y}&:= \Up_{\smalltwomat y{}{}1},\qquad & \Up_{y}^{\circ, \uw}:= y^{(-w_{0}-w+2)/2 } \Up_{y}.
\eeq
If $L$ is a finite extension of $\Q_{p}$, 
then for all $y\in \prod_{v\vert p}\OO_{F, v}-\{0\}$ we also  define the operator
\beqq \Up_{y}^{\circ}\colon {\bf S}(U^{p}, L) &\to {\bf S}(U^{p}, L)\\
W_{\Up_{y}^{\circ}f}(c)&:= W_{f}(cy).\eeqq
This is compatible with the previous definition in the following sense (see \cite[(2.2b)]{Hi}, where  $\Up_{y}$ is denoted by $T(y)$): if   $f$ is a $p$-adic nearly holomorphic form of weight $\uw$ over $L$, then for all $\iota\colon L\into \C$    we have 
$$(\Up_{y}^{\circ}f)^{\iota}= \Up_{y}^{\circ, \uw^{\iota}} f^{\iota}.$$
The superscript $\uw$ will be omitted when understood from the context. The ordinary projector is 
\beq 
e^{\ord}:=\lim_{n\to \infty} (\Up_{p}^{\circ})^{n!} \quad \in \quad \End_{L}({\bf S}(U^{p}, L))\eeq
for any tame level $U^{p}$ and $p$-adic field $L$.  Its image is denoted by 
 $${\bf S}^{\ord}(U^{p}, L):= e^{\ord}  {\bf S}(U^{p},L).$$
The operator $e^{\rm ord}$ preserves ${\rm S}_{\uw}(U^{p},L)$, and we denote  ${\rm S}^{\ord}_{\uw}(U^{p}, L):= e^{\ord}  {\bf S}(U^{p},L)$, ${\rm S}^{\ord}(U^{p}, L):= \bigoplus_{\uw}  {\rm S}^{\ord}_{\uw}(U^{p}, L)$.

 If $f^{\C}$ is a complex modular form  arising as $f^{\C}=f^{\iota}$ for a form $f\in {\rm S}(L)$
for  some finite extension $L$ of $\Q_{p}$ and some $\iota \colon L\into \C$, we define 
 $$e^{\rm ord, \iota} (f^{\C}) := (e^{\ord} f)^{\iota}.$$

\subsubsection{Differential operators after ordinary and holomorphic projections}
Let $L$ be a finite extension of $\Q_{p}$, and let $f_{1}$, $f_{2}$ be $p$-adic twisted  nearly holomorphic forms over $L$.
For any $\iota\colon L\into \C$ and $k \in \Z_{\geq 0}^{{\Sg_{\infty}}}$,
we have 
\beq \lb{e ord hol} 
[e^{\ord} (f_{1}\star d^{k}f_{2})]^{\iota}= e^{\ord, \iota} [e^{\hol} (f_{1}^{\iota} \star \delta^{k}f_{2}^{\iota})] ;\eeq
the proof of  \cite[Proposition 7.3]{Hi} carries over to the twisted case.

\subsection{Hida families} We gather the fundamental notions concerning Hida families and the associated  sheaves of modular forms.
  \subsubsection{Weight space} \lb{wt sp}
    Let $U_{F, p}^{\circ}=\prod_{v\vert p} U_{F, v}^{\circ}\subset \OO_{F, p}^{\ts}$ be a compact open subgroup (which will be fixed once and for all in \S~\ref{wt H}).   
    Let $U_{F}^{p}\subset \A^{p\infty, \ts}$ be a compact open subgroup, and consider the topological groups (with the profinite topology) 
    $$[Z]_{U_{F}^{p}}:= Z(F)U_{F}^{p}\bks Z(\A^{\infty}), \qquad [Z]_{U_{F}^{p}} \ts U_{F, p}^{\circ};$$
 the latter is isomorphic to $\Delta \ts \Z_{p}^{1+[F:\Q]+\delta_{F,p}}$, where $\Delta$ is a finite group and $\delta_{F, p}$ is the $p$-Leopoldt defect of $F$.  It is embedded into $A(\A^{\infty})$ by 
    $$(z, y_{p})\mapsto \twomat {zy_{p}}{}{}{z}.$$
   The \emph{weight space} (of tame level  $U_{F}^{p}$) is 
\beq
\frak W=\frak W_{U_{F}^{p}}:=\Spec \Z_{p}\llb [Z]_{U_{F}^{p}} \ts U_{F, p}^{\circ}\rrb_{\Q_{p}}.
 \eeq
A point $\ukappa\in \frak W$ is identified with the pair  of characters 
\beq\lb{ukkk}
\left(\kappa_{0}:=\ukappa_{|[Z]_{U_{F}^{p}}} , \qquad \kappa= \ukappa_{|U_{F, p}^{\circ}}
\right).
\eeq
We have an involution defined by 
$$\ukappa^{\curlyvee}(t):=\kappa_{0}(\det t)^{-1} \ukappa(t).$$ 

If $\uk$ is a $p$-adic weight for $\G$,
 we say that $\ukappa $ is \emph{classical of weight $\uk$} if for all $v\vert p$, 
$$\kappa_{0}^{\rm sm}(z_{p}):=\kappa_{0}(z_{p})z_{p}^{-k_{0}}, \qquad \kappa_{v}^{\rm sm}(y):=\kappa_{v}(y)\prod_{\tau|v} \tau(y)^{(-k_{0}-{k}_{\tau}+2)/2}$$
 are  smooth characters  of $F_{p}^{\ts}$, respectively $U_{F, v}^{\circ}$; in the second equation, $\kappa_{v}:=\kappa_{|U^{\circ}_{F,v}}$, and  the product runs over the $\tau\in \Sg_{p}$ inducing the place $v\in S_{p}$.  For a classical weight $\ukappa$, we define $\kappa^{\sm}:=\otimes_{v\vert p} \kappa_{v}^{\sm}$ and
\beq \lb{kp'}
\kappa':= \kappa^{\sm}\kappa_{0}^{\sm,-1} = \kappa^{\curlyvee, \sm},\eeq
 a smooth character of $U_{F, p}^{\circ}$.

We denote by 
$$\frak W^{\cl}\subset \frak W$$
the set of points of classical weight, which has the structure of an ind-\'etale ind-finite scheme over $\Q_{p}$. If $\ukappa$ is classical of weight $\uk=(k_{0}, k)$, then $\ukappa^{\curlyvee}$ is classical of weight $\uk^{\vee}=(-k_{0}, k)$.  We let $\frak W^{\cl, \geq 2}$ be the set of classical points satisfying $k \geq 2$. 

\subsubsection{Hida schemes} In light of the examples of the previous and following paragraph, 
 it  will be convenient to introduce a suitable category of spaces.\footnote{The treatment proposed here is minimal and  somewhat ad hoc, but it will be sufficient for our purposes. We believe that a more systematic treatment of the geometry of Hida theory should be based on the theory of uniformly rigid spaces developed in \cite{kappen}.} Define the category of  \emph{Hida rings} to consist of finite flat $\Z_{p}\llb X_{1}, \ldots, X_{n}\rrb$-algebras $A^{\circ}$  (for some $n$) 
and $\Z_{p}$-algebra morphisms,  and the category of \emph{Hida algebras} to be the image of Hida rings  under the functor $\ot_{\Z_{p}}\Q_{p}$. Define the category of affine Hida schemes to be dual to the category of Hida algebras. A \emph{Hida scheme} is an open subset of an affine Hida scheme.  If $A_{i}^{\circ}$ are Hida rings (for $i=1, 2$) and   $\X_{i}=\Spec (A_{i}^{\circ}\ot_{\Z_{p}} \Q_{p})$ (for $i=1,2$), we define 
$$\X_{1}\hat{\ts} \X_{2} := \Spec (A^{\circ}_{1}\hat{\ot}A^{\circ}_{2})_{\Q_{p}},$$
where $\hat{\ot}$ is the completed tensor product.

\subsubsection{Hida families} 
Let 
$${\bf T}_{U^{p},\Q_{p}}^{\rm sph, \ord}\subset {\bf T}_{U^{p}, \Q_{p}}^{\ord}\subset \End({\bf S}^{\ord}(U_{p}, \Q_{p}))$$ 
be the images of the Hecke algebras $\cH_{U^{p}, \Q_{p}}^{\rm  sph, \ord}$, $\cH_{U^{p}, \Q_{p}}^{\ord}$ from \S~\ref{213}. 
We let 
$$\Y_{\G}=\Y_{\G, U^{p}}:= \Spec {\bf T}_{U^{p}, \Q_{p}}^{\rm sph, \ord}$$
be the \emph{ordinary eigenvariety} for $\G$ of tame level $U^{p}$. (The subscript $U^{p}$ will be omitted when unimportant or understood from the context.)   The space $\Y_{\G, U^{p}}$
 is a union of finitely many irreducible components, called \emph{Hida families} of tame level (dividing)~$U^{p}$. It carries an involution $\curlyvee$ deduced from the one on $\cH_{U^{p}, \Q_{p}}^{\rm  sph, \ord}$.

Letting $U_{F} ^{p}:= U^{p}\cap Z(\A^{p\infty})$, we have a \emph{weight-character} map
$$\ukappa_{\G}\colon \Y_{\G, U^{p}} \to \frak W_{U_{F}^{p}}$$
that, when identified with a pair $(\kappa_{\G, 0}, \kappa_{\G}) $ of $\OO(\Y_{\G})^{\ts}$-valued characters as in \eqref{ukkk}, is   $\kappa_{\G, 0}(z)=$ the  Hecke operator  acting by  right translation by $z$ on modular forms, $\kappa_{G}(y_{p})= \Up_{y_{p}}^{\circ}$.  The weight map  is finite and flat  and it intertwines the involutions $\curlyvee$.

The set of classical points of $\Y_{\G} $ is $$\Y_{\G}^{\cl}:=\Y_{\G} \ts_{\frak W}\frak W^{\cl, \geq 2}\subset \Y_{\G}.$$ 
If $x_{0}\in \Y_{\G}^{\cl}$, we denote by $\pi_{x_{0}}$ the automorphic representation of $\G(\A)$ over $\Q_{p}(x_{0})$ on which  $\cH^{\rm sph}_{U^{p}}$ acts 
by the $\Q_{p}(x_{0})$-character corresponding to $x_{0}$. If $x\in \Y_{\G}^{\cl}(\C)$ corresponds to $(x_{0}\in \Y_{\G}^{\cl}, \iota\colon \Q_{p}(x_{0})\into \C)$, we denote $\pi_{x}:= \pi_{x_{0}}^{\iota}$. 

\subsubsection{Families of ordinary forms}\lb{244}
By construction, for each $U^{p}{}'\subset U^{p}$, the ordinary eigenvariety $\Y_{\G, U^{p}}$, respectively the weight space $\frak W_{U_{F}^{p}}$, carries a (coherent) sheaf 
$$\cS^{U^{p}{}'},$$
respectively  $\cS^{U^{p}{}'}_{\frak W}:=\kappa_{\G, *} \cS^{U^{p}{}'}$, 
whose modules of global sections is ${\bf S}^{\ord}(U^{p}{}', \Q_{p})$. 
We set $\cS_{(\frak W)}:=\varinjlim_{U^{p}{}'}\cS_{(\frak W)}^{U^{p}}.$
By Hida's Control Theorem (see \cite[Corollary 3.3]{Hi}), the restriction of $\cS_{\frak W}^{U^{p}}$ to $\frak W^{\cl}$ is the sheaf attached to ${\rm S}^{\ord}(U^{p}, \Q_{p})$.

For each $x\in \Y_{\G, U^{p}}^{\cl}$ of weight $\uw$, there exists a unique (up to isomorphism) ordinary automorphic representation $\pi_{x}$
 of $\G(A)$ over $L:=\Q_{p}(x)$ of weight $\uw$, such that there is an $\cH_{U^{p}, L}^{\ord}$-isomorphism
  $$\cS_{|x}\cong \pi_{x}^{\ord}:= e^{\ord}\pi_{x}:= [\lim_{n} (\Up_{p}^{\circ, \uw})^{n!}]\, \pi_{x};$$
the isomorphism is unique up to scalars. This defines a bijection between $\Y_{\G, U^{p}}^{\cl}(\baar{\Q}_{p})$ and the set of isomorphism classes of  ordinary automorphic representation $\pi$ of $\G(\A)$ over $\baar{\Q}_{p}$ with $\pi^{U^{p}}\neq 0$.

\begin{lemm}  \lb{Z-adic} Let $U^{p}\subset \G(\A^{p\infty}) $ be a compact open subgroup and let $U_{F}^{p}:= U^{p}\cap Z(\A^{p\infty})$. Let $\cZ$ be a Hida scheme endowed with a map $\vphi \colon \cZ\to \frak {W=\frak W}_{U^{p}_{F}}$.  Then we have an $\OO_{\cZ}$-linear  injective $q$-expansion map
\beq \lb{shv qexp}
\cS_{\cZ}^{U^{p}} := \cS^{U^{p}}_{\frak W} \ot_{\OO_{\frak W}} \OO_{\cZ}\longrightarrow
 \OO_{\cZ}^{\A^{\infty, \ts}/U_{F}^{p}} \subset   \OO_{\cZ}^{\A^{\infty, \ts}}
\eeq
characterised by the property that for every $\kappa \in \frak W^{\cl}$, every closed point $z\in \vphi^{-1}(\kappa)$,  and every $a\in \A^{\infty, \ts}$, we have 
$$W_{\bff}({a})(z)=W_{\bff(z)}(a),$$
 where the right-hand side is the $p$-adic $q$-expansion coefficient of the classical modular form ${\bff}(z)\in 
 (\vphi^{*}\cS_{\frak W})_{|z}
\subset   {\rm S}^{\ord}(U^{p}, \Q_{p}(z))$.
 
 Moreover, the image of \eqref{shv qexp} equals the space of those sequences $(W(a))_{a}$ for which  there exists a  set of closed points $\Sg\subset \vphi^{-1}(\frak W^{\cl})$ that is dense in $\cZ$, such that that for all $z\in \Sg$ the sequence $(W(a)(z))_{a}$ is the $p$-adic $q$-expansion of a modular form $\bff_{z}\in (\vphi^{*}\cS_{\frak W})_{|z}.$
\end{lemm}

Note that the sheaf $\cS^{U^{p}}$ on $\Y_{\G}$ is identified with $(\cS_{\frak W}^{U^{p}}\ot_{\OO_{\frak W}}\Y_{\G})^{\OO_{\Y_{\G}}}$, the subsheaf of  invariants for the diagonal $\OO_{\frak W}$-linear action of  $\OO_{\Y_{\G}}$. In particular, we deduce from \eqref{shv qexp} a $q$-expansion map 
\beq\lb{qexp Y}
\cS^{U^{p}}\to \OO_{\Y_{\G}}^{\A^{\infty, \ts}/U_{F}^{p}}.
\eeq

\begin{proof}
It suffices to construct \eqref{shv qexp} for $\cZ=\frak W$ as the general case follows by base-change. Let $A^{\circ}:= \Z_{p}\llb [Z]_{U_{F}^{p}} \ts U_{F, p}^{\circ}\rrb$, and let ${\bf S}^{\ord}(U^{p}, \Z_{p})$ be the space of ordinary forms with $\Z_{p}$-coefficients; this is an $A^{\circ}$-module and a $\Z_{p}$-lattice in $\cS_{\frak W}^{U^{p}}(\frak W)= {\bf S}^{\ord}(U^{p}, \Q_{p})$. 
 For $\kappa\in \frak W^{\cl}$, let $\frakp_{\kappa}\subset A^{\circ}$ be the corresponding prime ideal.  Let $n\in \N$ and let $M$ range among finite subsets of $\frak W$; the filtered system of ideals 
 $$I_{n,M}:= (p^{n})+\bigcap_{\kappa \in M} \frakp_{\kappa}$$
forms a fundamental system of neighbourhoods of $0\in A^{\circ}$, i.e. $A^{\circ}=\varprojlim_{n, M} A^{\circ}/I_{n, M}$.  The $p$-adic  $q$-expansion maps ${\bf S}^{\ord}(U^{p}, \Z_{p})\ot_{A^{\circ}}A^{\circ}/\frakp_{\kappa}\to \Z_{p}(\kappa)^{\A^{\infty, \ts}}$ yield a compatible family of maps 
$${\bf S}^{\ord}(U^{p}, \Z_{p})\ot_{A^{\circ}}A^{\circ}/I_{n, M} \to (A^{\circ}/I_{n, M})^{\A^{\infty, \ts}}$$
and after taking projective limits, the desired map ${\bf S}^{\ord}(U^{p}, \Z_{p}) \to (A^{\circ})^{\A^{\infty, \ts}}$. It is injective by the $q$-expansion principle and the preservation of injectivity under inverse limits.

We now consider the second statement. It is clear that, for any fixed $\Sg$ as in the lemma, the space $\tilde{\cS}_{\cZ}^{U^{p}, \Sg} \subset \OO_{\cZ}^{\A^{\infty, \ts}}$  described contains the image of \eqref{shv qexp}; we show the opposite containment. We may assume that $\cZ=\Spec B^{\circ}_{\Q_{p}}$ for a Hida ring $B^{\circ}$, and  consider \eqref{shv qexp} as a map
\beq\lb{shv qe2} 
{\bf S}^{\ord}(U^{p}, \Z_{p})\ot_{A^{\circ}}B^{\circ} \to \tilde{\bf S}_{B^{\circ}}^{\ord}(U^{p})^{\Sg} :=(B^{\circ})^{\A^{\infty, \ts}}\cap \tilde{\cS}_{\cZ}^{U^{p}} (\cZ).
\eeq For $z\in \Sg$, let $\frakp_{z}\subset B^{\circ}$ be the corresponding prime ideal.  Let $n\in \N$ and let $N$ range among finite subsets of $\Sg$; then  the filtered system of ideals 
 $J_{n,N}:= (p^{n})+\bigcap_{z\in N} \frakp_{z}$
forms a fundamental system of neighbourhoods of $0\in B^{\circ}$. By assumption, for each $z\in \Sg$ the map  \eqref{shv qe2} is an isomorphism modulo $\frakp_{z}$; hence it is an isomorphism modulo $J_{n, N}$ for all $(n, N)$, hence an isomorphism.
\end{proof}
We call elements of $\cS_{\cZ}^{U^{p}}$ (respectively $\cS_{\cZ}^{U^{p}}\ot_{\OO_{\cZ}} \cK(\cZ)$) \emph{$\cZ$-adic ordinary modular cuspforms} (respectively \emph{meromorphic} $\cZ$-adic ordinary modular cuspforms) of weight $\vphi\colon \cZ\to \frak W$.

\subsubsection{Weight-character map for $\H$}\lb{wt H} Let  $U_{\H}^{p}\subset \H(\A^{p\infty})$ be an open compact subgroup, and let $$\Y_{\H}=\Y_{\H, U_{\H}^{p}}:=\Spec \Z_{p}\llb \H(F)\bks \H(\A^{p\infty})/U_{\H}^{p}\rrb\ot_{\Z_{p}}\Q_{p}$$ 
as in \eqref{YH}. A (Hida) \emph{family} for $\H$ is a connected component of $\Y_{\H}$. 

Fix    a sufficiently small open compact subgroup 
$U_{F, p}^{\circ, \sqrt{}}=\prod_{v\vert p} U_{F, v}^{\circ, \sqrt{}}\subset \OO_{F,p}^{\ts}$
 and an injective group homomorphism 
$$j'' \colon (U_{F, p}^{\circ, \sqrt{}})
 \to \OO_{E, p}^{\ts,1}:=  \{ t\in \OO_{E,p}^{\ts}\  |\  N_{E_{p}/F_{p}}(t)=1\},$$
  and let 
$$U_{F, p}^{\circ}:= (U_{F,p}^{\circ, \sqrt{}})^{2}\subset \OO_{F, p}^{\ts},$$
which is now fixed as promised in \S~\ref{wt sp}.
Let $\sqrt{\ }\colon U_{F, p}^{\circ} \to U_{F, p}^{\circ, \sqrt{}} $ be the (uniquely determined, up to shrinking $U_{F, p}^{\circ, \sqrt{}}$) square root, 
 and let $j, j'\colon U_{F, p}^{\circ}\to \OO_{E , p}^{\ts}$ be the maps\footnote{If $v\vert p$ splits in $E$, then for $a\in U_{F, v}^{\circ}$ we have $j(a)=(a, 1)$ under some isomorphism $E_{v}^{\ts}\cong F_{v}^{\ts}\ts F_{v}^{\ts}.$}
\beq \lb{j wt h}
j'(a):= j''(\sqrt{a})/\sqrt{a},\qquad   j(a)= j'(a)a.
\eeq

For any open compact $U_{E}^{p}\subset \A_{E}^{p\infty, \ts}$ and $U_{F}^{p}:= U_{E}^{p}\cap \A^{p\infty, \ts}$,  define a map
\beq\lb{wt H eq}
\ukappa_{\H}\colon \ \Y_{\H, U_{E}^{p}} &\to \frak W_{U_{F}^{p}}\\
y &\mapsto \ukappa_{\H}(y)=\left( \kappa_{0}=\chi_{y|[Z]_{U_{F}^{p}}}, \kappa: = \chi_{y}\circ j
\right).
\eeq
The set of  classical points is
$$\Y_{\H}^{\cl}:= \Y_{\H} \ts_{\frak W} \frak W^{\cl}.$$ 
Note that if $y\in \Y_{\H}$ is a classical point such that  $\ukappa_{\H}(y)$ has  weight $(l_{0}, l)$, then $\chi_{y}$ has weight $(l_{0}, l)$ as defined in the introduction. 

\subsubsection{Hida families for $\G\ts\H$} These are defined as in \S~\ref{sec HF}.

\subsubsection{Universal automorphic sheaf on a Hida family}
Let $\X_{\G}$ be a Hida family for $\G$, and let 
$\X_{\G}^{\cl}:=  \X_{\G} \cap \Y_{\G}^{\cl} $.
For each sufficiently small  $U^{p}$, we may view  $\X\subset \Y_{\G, U^{p}}$ and we define
 $$\Pi^{U^{p}}=\Pi_{\X_{\G}}^{U^{p}}:={\cS}^{U^{p}}_{|\X_{\G}}.$$ 
 For each $x\in \X_{\G}^{\cl}$, by  Hida's Control Theorem  (see for instance \cite[Corollary 3.3]{Hi}) and the theory of newforms, we have an isomorphism of  $\cH_{U^{p}}$-modules,
\beq \lb{HCT}
\Pi_{|x}^{U^{p}}  \cong \pi_{x}^{U^{p}, \ord}:= e^{\ord} \pi_{x}^{U^{p}}.\eeq

 Let $U^{p}_{\X_{\G}}$ be minimal such that $\X_{\G}$ is a component of $\Y_{\G, U^{p}_{\X_{\G}}}$.
 By \cite[\S 3]{Hi}, there is a unique 
 \beq\lb{new bfff}
 \bff_{0}=\bff_{0,\X_{\G}}\in \Pi^{U^{p}_{\X_{\G}}}(\X_{\G})
 \eeq 
 (the \emph{normalised primitive form} over $\X_{\G}$) such that $W_{{\bff}_{0}}(1)=1\in \OO(\X_{\G})$ for the $q$-expansion map deduced from \eqref{qexp Y}. Any $\bff \in \Pi^{U^{p}}$ can be written as $\bff =T \bff_{0}$ for some Hecke operator $T$ supported at the places $v\nmid p\infty$ such that $U^{p}$ is not maximal.

\subsubsection{Universal Galois sheaf on a Hida family and local-global compatibility} \lb{sec LG} Let $\X_{\G}$ be a Hida family for $\G$. By results of Hida and Wiles (see \cite[Proposition 3.2.4]{univ}), there exist an open subset $\X_{\G}'\subset \X_{\G}$ containing $\X_{\G}^{\cl}$ and 
a locally free sheaf $\cV_{\G}$ of rank~$2$,  endowed with a Galois action 
$$G_{F}\to \End_{\OO_{\X_{\G}'}} (\cV_{\G})$$
such that for all  $x\in \X_{\G}^{\cl}$, the fibre $\cV_{\G|x}$ is the Galois representation attached to $\pi_{x}$ by the global Langlands correspondence.

Let $S$ be a finite set of finite places of $F$, disjoint from $S_{p}$, such that for all $v\notin S$ the tame level $U^{p}=U^{Sp}U_{S}$ of $\X_{\G}$ is maximal at $v$. We define 
\beq\lb{Pidef}
\Pi_{\X_{\G}}^{U^{Sp}}:= \varinjlim_{U'_{S}} \Pi_{\X_{\G}}^{U^{Sp}U'_{S}}, \eeq
which is a finitely generated  $\OO_{\X_{\G}}[\G(F_{S})]$-module. 
On the other hand,  \cite[Theorem 4.4.1]{LLC} attaches to the restriction $\cV_{\G, v}:= \cV_{\G|G_{F_{v}}}$ 
an $\OO_{\X_{\G}'}[\G(F_{S})]$-module
 $$\Pi(\cV_{ \G, v}),$$
  which is torsion-free and \emph{co-Whittaker} in the sense of \cite[Definition 4.2.2]{LLC}. 
\begin{prop}\lb{LGC} After possibly replacing $\X_{\G}'\subset\X_{\G}$ with a smaller open subset still  containing $\X_{\G}^{\cl}$, there exists  a line bundle $\Pi^{\circ}_{\X_{\G}'} $ over $\X_{\G}'$ with trivial $\G(F_{S})$-action, such that 
$$\Pi^{U^{Sp}}_{|\X_{\G}'} \cong \Pi^{\circ}_{\X_{\G}'} \ot \bigotimes_{v\in S} \Pi(\cV_{\G,v})$$
as $\OO_{\X'_{\G}}[\G(F_{S})]$-modules.
\end{prop}
\begin{proof}
By the local-global compatibility of the Langlands correspondence for Hilbert modular forms (see \cite{carayol-hilbert} or \cite[Theorem 2.5.1]{univ}), for all  $x\in \X_{\G}^{\cl}$ and all places $v$, 
the $\G(F_{v})$-representation $\pi_{x,v}$ corresponds, under local Langlands, to  the Weil--Deligne representation  $V_{x, v}$
attached to $\cV_{\G|x|G_{F_{v}}}$.
Then the result follows from \cite[Theorem 4.4.3]{LLC}.\end{proof}

\section{Theta--Eisenstein family}
\lb{sec 3}

In this section, we define the kernel of the Rankin--Selberg convolution giving the $p$-adic $L$-function.

\subsection{Weil representation}
We recall the definition of the Weil representation for groups of similitudes; this  subsection is largely  identical to \cite[\S 3.1]{pyzz}.

\subsubsection{Local case} Let $V=(V,q)$ be a quadratic space of even dimension  over a local field $F$ of characteristic not 2. Fix  a nontrivial additive character $\psi$ of $F$. 
For $u\in F^\times$, we denote by $V_u$ the quadratic space $(V,uq)$. We let $\GL_2(F)\times \GO(V)$ act on the usual  space of Schwartz functions
${\calS}'(V\times F^\times)$  as follows (here $\nu\colon \GO(V)\rightarrow\mathbf{G}_m$ denotes the similitude character):
\begin{itemize}
\item $r(h)\phi(x,u)=\phi(h^{-1}x,\nu(h)u)$ \quad for $h\in \GO(V)$;
\item$r(n(b))\phi(x,u)=\psi(buq(x))\phi(x,u)$ \quad for $n(b)\in N(F)\subset\GL_2(F)$;
\item $r\left(\begin{pmatrix}a&\\& d\end{pmatrix}\right)\phi(x,u)=\chi_{V_{u}}(a)|{a\over d}|^{\dim V\over 4}\phi(at, d^{-1}a^{-1}u)$;
\item $r(w)\phi(x,u)=\gamma(V_u)\hat{\phi}(x,u)$ for $w=\begin{pmatrix}&1\\-1 &\end{pmatrix}.$
\end{itemize}
Here $\chi_{V}=\chi_{(V,q)}$ is the quadratic character attached to $V$,  $\gamma(V,q)$ is a fourth root of unity, and $\hat{\phi}$ denotes Fourier transform in the first variable with respect to the self-dual measure for the character $\psi_{u}(x)=\psi(ux)$. We will need to note the following facts (see for instance \cite{JL}): $\chi_{V}$  is trivial if $V$ is a quaternion algebra over $F$ or $V=F\oplus F$, and $\chi_{V}=\eta$ if $V$ is a separable quadratic extension $E$ of $F$ with associated character $\eta$.

\medskip

\subsubsection{Fock model and reduced Fock model} 
Assume that $F=\R$ and $V$ is positive definite. Then we will prefer to consider a modified version of the previous setting.   Let the \emph{Fock model} $\calS(V\times \R^{\times}, \C)$ be the space of functions spanned by those of the form 
$$H(u)P(x) e^{-2\pi |u| q(x)},$$ where $H$ is a compactly supported smooth function on $\R^{\times}$ and $P$ is a complex  polynomial function on $V$.  This space is not stable under the action of $\GL_{2}(\R)$, but it is so under the restriction of the induced $(\mathfrak{gl}_{2, \R},{\bf O}_{2}(\R))$-action on the usual Schwartz space (see \cite[\S 2.1.2]{yzz}). 

We will also need to consider the \emph{reduced Fock space} $\bcalS(V\times \R^{\times})$ spanned by functions of the form
$$\phi(x,u)=(P_{1}(uq(x))+\mathrm{sgn}(u)P_{2}(uq(x))) e^{- 2\pi |u| q(x)}$$
where $P_{1}$, $P_{2}$ are polynomial functions with rational coefficients.

By \cite[\S 4.4.1, 3.4.1]{yzz}, there is a surjective quotient map 
\begin{equation}\label{fock}
\begin{aligned}
\calS(V\times\R^{\times} , \C)&\to\bcalS(V\times\R^{\times})\otimes_{\Q} \C\\
 \Phi& \mapsto\phi(x,u)=\baar{\Phi}(x,u)= \int_{\R^{\times}}\dashint_{{\bf O}(V)} r(ch) {\Phi}(x,u)\, dh \, dc.
 \end{aligned}
\end{equation}
We let $\calS(V\times\R^{\times} ) \subset \calS(V\times\R^{\times} , \C)$ be the preimage of $\bcalS(V\times\R^{\times})$.
For the sake of uniformity, when $F$ is non-archimedean we set $\bcalS(V\times F^{\times})=\calS(V\times F^{\times}):=\calS'(V\times F^{\times})$.

\subsubsection{Global case} Let $({\bf V},q)$ be an even-dimensional quadratic space over the ad\`eles $\A$ of a totally real number field $F$, and suppose that ${\bf V}_{\infty}$ is positive definite; we say that ${\bf V}$ is \emph{coherent} if it has a model over $F$ and \emph{incoherent} otherwise. Given an $\widehat{\OO}_{F}$-lattice ${\bf {V}}^{\circ}\subset {\bf V}$, we define the space $\calS({\bf V}\times \A^{\times})$ as the restricted tensor product of the corresponding local spaces, with respect to the spherical  elements
$$\phi_{v}(x,u)=\one_{{\bf V}^{\circ}_{v}}(x)\one_{\vpi_{v}^{n_{v}}}(u),$$ 
if $\psi_{v}$ has level $n_{v}$. We call such $\phi_{v}$ the \emph{standard Schwartz function} at a non-archimedean place $v$.
We define similarly the reduced  space $\bcalS({\bf V}\times \A^{\times})$, which admits a quotient map
\begin{gather}\label{schwartz quotient}
\calS({\bf V}\times \A^{\times})\to\bcalS({\bf V}\times \A^{\times})
\end{gather}
defined by the product of the maps \eqref{fock} at the infinite places and of the identity at the finite places. The Weil representation of $ \GO({\bf V})\times \G(\A^{\infty})\times(\mathfrak{gl}_{2,F_{\infty}},{\bf O}({\bf V}_{\infty}))$ is  the restricted tensor product of the local representations. 

For a quadratic space ${\bf V}=({\bf V}, q) $ over $\A$, we define $\eps({\bf V})=+1$ (respectively $-1$) if and only if there exists (respectively does not exist) a quadratic space $V$ over $F$ such that $V\ot_{F}\A={\bf V}$.

\subsubsection{The quadratic spaces of interest} \lb{sec bf V}  Let us go back to our usual notation: thus $F$ is our chosen totally real field and $E$ its chosen CM quadratic extension. In this paper, we will consider the quadratic spaces ${\bf V}=({\bf B}, q)$, where ${\bf B}$ is a  quaternion algebra over $\A$, definite at all the archimedean places and split at $p$, and endowed with  an $\A$-embedding $\A_{E}\into {\bf B}$, and $q\colon {\bf B}={\bf V}\to \A$ is its reduced norm. It has  a decomposition 
$${\bf V}={\bf V}_{1}\oplus {\bf V}_{2}$$
where ${\bf V}_{1}=\A_{E}$ (on which the restriction of $q$ coincides with $N_{E/F}$) and ${\bf V}_{2}$ is the $q$-orthogonal complement. Thus $\eps({\bf V})=\eps({\bf V}_{2})$. We denote by $r_{1}$ the restriction of $r$ to a representation of $\A_{E}^{\ts}={\rm GO}({\bf V}_{1})$ on $\bcalS({\bf V}_{1}\ts \A^{\ts})$. 

For each  place $v$, we have 
\beq
\lb{eps B}
\eps({\B}_{v})=\eps({\bf V}_{v})=
\begin{cases} +1 &\text{if $\B_{v}\cong M_{2}(F_{v})$} \\ -1 &\text{if $\B_{v}$ is a division algebra.}
\end{cases}
\eeq
We  have $\eps({\bf V}):=\prod_{v} \eps({\bf V}_{v}) = (-1)^{[F:\Q]} \prod_{v\nmid p} \eps({\bf V}_{v})$.

\subsection{Theta series} Let $\phi_{1} \in \bcalS({\bf V}_{1}\ts \A^{\ts})$. 
We define a function  on $\G(\A)\ts \A^{\ts}$ by 
\beq \theta(g, u, \phi_{1}):= \sum_{x\in E}r(g) \phi_{1}(x, u).
\eeq
It satisfies 
\beq \lb{1236}
\theta(zg, u, \phi_{1}) =\eta(z) \theta(g, u, r(z^{-1}, 1) \phi_{1}) \eeq
for all $z\in {\A}^{\ts}$ (here we view $(z, 1) \in \G(\A)\ts \A_{E}^{\ts}$).

For a complex  weight $\ul$ for $\H$, let
\beq\lb{phi1inf}
\phi^{}_{1, \ul, \infty} &:=\ot_{v\vert \infty}\phi^{}_{1 , \ul,v}, \\
\phi^{}_{1 ,\ul,v}(t, u)^{} &:= \one_{\R^{+}}(u) \begin{cases}
 t^{l_{v}}
|u|^{(-l_{0}+l_{v})/ 2}e^{-2\pi uq(t)} & \text{if $l_{v}\geq 0$}\\
( \baar{t})^{-l_{v}}
|u|^{(-l_{0}-l_{v})/  2}e^{-2\pi uq(t)} & \text{if $l_{v} \leq  0$}
\end{cases}
,\eeq
and let
$$\theta^{}(g, u, \phi_{1}^{\infty};\ul) := \theta(g, u, \phi_{1}^{\infty}\phi^{}_{1, \ul, \infty}).$$
Define $|\ul|:= (l_{0}, (| l_{v}|)_{v})$.
\begin{lemm} The series $\theta^{}(g, u, \phi_{1}^{\infty}; \ul)$ is a twisted modular form of weight $(0, \underline{1})+|\ul |$.
\end{lemm}
\begin{proof}
The usual proof that  classical theta series are automorphic shows that our $\theta$ is twisted automorphic. The archimedean component of the central character is easy to determine by \eqref{1236}. The weight is  computed as in \cite[\S A1 on p. 350]{Xue}.
 \end{proof}

The Whittaker--Fourier   expansion of $\theta({\ul})$ is standard:  for all $g=\smalltwomat yx{}1\in \G(\A)$ with $y\in \A^{+}$, 
\beq\lb{tht fou}
\theta(g, u, \phi_{1}^{\infty}; \ul) =\sum_{a\in F^{\ts}} \sum_{x\in E^{\ts}\colon  uq(x)=a} r(g) \phi_{1}(x, u)
= \eta(y)|y| y_{\infty}^{l_{0}+ | l |\over 2} \sum_{a\in F^{\ts}} \sum_{\substack{x\in E^{\ts} \\  uq(x)=a}}  \phi_{1}^{\infty}(yx, y^{ -1} u) \, \qqq^{a}.
\eeq

 The following expansion result will be used in \S~\ref{47}.

  \begin{lemm}\lb{theta FW}
 Let $\chi\colon E^{\ts}\bks \A_{E}^{\ts}\to \C^{\ts}$ be a locally algebraic character of weight $\ul$, and let $E(g,u) $ be any twisted modular form such that $E(g, u_{\infty} u)=E(g, u)$ for all $u_{\infty}\in F_{\infty}^{+}$. Suppose that  $\phi_{1}^{\infty}(0, u)=0$
 for all $u$. Then for all $g=\twomat yx{}1\in \G(\A)$, we have 
\begin{multline*}
\int^{*}_{E^{\ts}\bks {\A}_{E}^{\ts} /\A^{\ts}}  \chi(t) \theta(g, u, r(t)\phi_{1}^{\infty};   \ul) \star E(g,q(t) u)\, dt  \\ 
= 4 |D_{E/F}|^{1/2}  \sum_{a\in  F^{\ts}}
   \one_{F_{\infty}^{+}}(y_{\infty}^{-1}a)  \eta(y) |y|^{1/2} y_{\infty}^{| l | +l_{0}\over 2}  \int_{{\A}_{E}^{\infty,\ts} }
      \chi^{\infty}(t)  r_{1}(t)\phi_{1}^{\infty}(y,y^{-1} a)
\ E(g, q(t) a) \, d^{\bullet}t\, \qqq^{a} .
 \end{multline*}
 \end{lemm}
\begin{proof}
We may assume that $U_{F}$ is so small that $E(u)$ is invariant under $u\in U_{F}$ and $\nu_{U_{F}}=1$. Taking fundamental domains for $\mu_{U_{F}}^{2}\bks F^{\ts}$, the expression of interest is
\beqq
  &c_{U_{F}}\eta(y) |y|^{1/2} \int^{*}_{E^{\ts}\bks {\A}_{E}^{\ts} /\A^{\ts}} \chi(t)
 \sum_{u\in \mu_{U_{F}}^{2}\bks F^{\ts}} \sum_{a\in  \mu_{U_{F}}^{2}\bks F^{\ts}}\sum_{x\in E^{\ts}}
  \phi_{1}^{}(t^{-1}xy, y^{-1}q(t) u) \one[uq(x)=a] E(q(t) u,g) \, dt \eeqq
  
  Since the integrand is invariant under $E_{\infty}^{\ts}$,  by Lemma \ref{int *} with $\mu=\mu_{U_{F}}$ and a change of variables $a=uq(x)$, this equals
  \beqq
  &  {2 L(1, \eta) c_{U_{F}}\over h_{E} [\OO_{E}^{\ts}:\mu_{U_{F}}]} \eta(y) |y|^{1/2}   
  \int_{ {\A}_{E}^{\infty,\ts} }  \sum_{a\in \mu_{U_{F}}^{2}\bks F^{\ts}} \sum_{\alpha\in \mu_{U_{F}}}  \\
& \phantom{ c_{U_{F}}\eta(y) |y|^{1/2}   \int_{\bks {\A}_{E}^{\infty,\ts} } } { \chi^{\infty}(t\alpha) \phi_{1}^{\infty}(t^{-1}\alpha^{-1}y, y^{-1}aq(t)\alpha^{2})) \one_{F_{\infty}^{+}}(y_{\infty}^{-1}a) y_{\infty}^{ | l | +l_{0}\over 2} \ E(g,q(t\alpha) a) \, \qqq^{\alpha^{2}a} \, d^{\bullet}t}.\eeqq
By the invariance properties under ${U_{F}}$,  this can be brought into  the desired expression by a change of variables $a'=\alpha^{2}a$ and the calculation
$${2 L(1, \eta) c_{U_{F}}\over h_{E} [\OO_{E}^{\ts}:\mu_{U_{F}}]} = 
{L(1, \eta) [\OO_{E}^{\ts}:\OO_{F}^{\ts}] \over h_{E}/h_{F}}
= 4 |D_{E/F}|^{1/2},$$
which follows from the definition of $c_{U_{F}}=\eqref{cUF}$ and  the class number formula
\end{proof}

 \subsection{Eisenstein series}\lb{sec eis}
Let ${\bf V}_{2}$ be a $2$-dimensional  quadratic space over $\A$, totally definite at the archimedean places. Let $\phi_{2}^{}\in \bcalS({\bf V}_{2}\times \A^{\times})$ be a Schwartz function,  and let 
 $\xi\colon  F^{\ts}\bks {\A}^{\ts}\to \C^{\ts}$ be a locally algebraic character  such that $\xi_{\infty}(x) = x^{k_{0}}$ for some integer $k_{0}$ and for  all $x\in F_{\infty}^{+}$. Define
 the automorphic Eisenstein series\footnote{For $k_{0}=0$, this is ${L^{(p\infty)}(1, \eta\xi) / {L}^{(p\infty)}(1, \eta)}$ times
 the series defined in \cite{pyzz}.}
$${E}_{r}(g, u , \phi_{2};\xi)={L^{(p\infty)}(1, \eta\xi)\over {L}^{(p\infty)}(1, \eta)}\sum_{\gamma\in P^{1}(F)\bks SL_{2}(F)} \delta_{\xi, r} (\gamma g w_{r,p}) r(\gamma g) \phi_{2}(0,u)$$
where (with $s\in \C$)
\beq\lb{dlx}
\delta_{\xi, r}(g) &:=\delta_{\xi, r, 0}(g)\\
\delta_{\xi, r,s}(g) &:=\begin{cases} 
\xi(d)^{-1}  |a/d|^{s/2} \psi(k_{0} \theta)  &\textrm{if\ } g= \smalltwomat ab{}d  h   \textrm{\ with\ } h=h^{\infty}r_{\theta}\in U_{p,r} SO(2, F_{\infty})
\\ 0 & \textrm{if\ } g\notin P(\A) U_{p,r} SO(2, F_{\infty}).
\end{cases}
\eeq
(The defining sum is  absolutely convergent for $\Re(s)$ sufficiently large, and otherwise it is interpreted  by analytic continuation.)
It satisfies 
\beqq \lb{Eis cc}
{E}_{r}(zg, u , \phi_{2};\xi) =\eta\xi^{-1}(z) {E}_{r}(g,u,  r(x, 1)\phi_{2},\xi).\eeqq

\subsubsection{Schwartz function at $\infty$}
Let $P_{k_{0}, k}
 \in \R[X]$ be the (rescaled) Laguerre polynomial
\beq\lb{Lag def}
P^{}_{k_{0}, k}(X) &:=(2\pi i)^{-k_{0}}  (4\pi)^{-k} (k+k_{0})! \sum_{j=0}^{k} {k\choose j} {(-X)^{j}\over j!}. \\
\eeq

 For $k=(k_{0}, (k_{v}))\in\Z\ts \Z_{\geq 0}^{{\Sg_{\infty}}}$ such that $k_{v}+k_{0}\geq 0$ for all $v$,  
  define 
  $${E}^{}_{r} (g, u , \phi_{2}^{\infty} ;\xi, k) = {E}^{}_{r}(g, u , \phi_{2}^{\infty}\phi^{}_{2, \infty, k} ; \xi)$$
where $\phi^{}_{2, \infty, k}=\ot_{v\vert \infty}\phi^{}_{2, v, k_{v}}$ with 
\beq\lb{phi2inf}
\phi^{}_{2,v, k_{v}}(x, u) = \one_{\R^{+}}(u) P_{k_{0},k_{v}}(4\pi uq(x))e^{-2\pi uq(x)}.\eeq

The series ${E}_{r}(\phi_{2}^{\infty}; \xi, k)$ belongs to $N_{{\rm tw}, (-k_{0}, k+k_{0})}^{\leq k}( \C)$.

\subsubsection{Whittaker--Fourier expansion} 
The following standard result is essentially \cite[Proposition 3.2.1]{pyzz}.  
\begin{prop}\label{eis-exp}
We have
$${E}_{r}(\smallmat yx{}1,u, \phi_{2}; \xi)= \sum_{a\in F} W_{a,r}( \smallmat y{}{}1, u, \phi_{2}; \xi)\psi(ax)$$
where
$$W_{a,r}(g,u, \phi_{2}; \xi)= \prod_{v} W_{a,r,v}(g,u, \phi_{2,v};\xi_{v})$$
with, for each $v$ and $a\in F_{v}$,
\begin{multline*}
W_{a,r,v}(g,u, \phi_{2,v};\xi_{v})= 
{ L^{(p\infty)}(1, \eta_{v}\xi_{v}) \over L^{(p\infty)}(1, \eta_{v})}
 \int_{F_{v}}\delta_{\xi,r, v}(wn(b)gw_{r,v}) r(wn(b)g ) \phi_{2,v }(0,u)\psi_{v}(-ab) \, db.
\end{multline*}
Here $L^{(p\infty)}(s, \xi'_{v}):=L(s, \xi'_{v})$ if $v\nmid p\infty$ and $L^{(p
\infty)}(s,\xi'_{v}):=1$ if $v\vert p\infty$, and we use the  convention that $r_{v}=0$ if $v\nmid p$.
\end{prop}
(Note that the functions $W_{a,r}(\phi_{2}, \xi)$ correspond to the $W^{\C}_{E_{r}(\phi_{2}, \xi), a}$ of \S~\ref{TMF}. We prefer to use lighter notation in this section.)

We choose convenient normalisations for the local Whittaker functions: let $\gamma_{u, v}=\gamma({\bf V}_{2,v}, uq)$  be the Weil index, and for $a\in F_{v}^{\times}$ set
$$W^{\circ}_{a,r,v}(g,u, \phi_{2,v}; \xi_{v}):=
\gamma_{u,v}^{-1} L^{(p)}(1, \eta_{v}) W_{a,r,v}(g,u, \phi_{2,v}; \xi_{v}) .
$$

Then for the global Whittaker functions we have 
\begin{align}\label{signW}
W_{a,r}(g,u, \phi_{2}; \xi)={-\eps({\bf V}_{2}) \over L^{(p)}(1, \eta)} \prod_{v}W^{\circ}_{a,r,v}(g,u, \phi_{2,v}; \xi_{v})
\end{align}
if $a\in F^{\times}$, where $\eps({\bf V}_{2})=\prod_{v}\gamma_{u, v}$ equals $-1$ if ${\bf V}_{2}$ is coherent or $+1$ if ${\bf V}_{2}$ is incoherent. We similarly define $W_{0, r}^{\circ}(g, u, \phi_{2}, \xi) $ by the identity
\begin{align}\label{whitt0}
W_{0, r}^{}(g, u, \phi_{2}; \xi) =
{-\eps({\bf V}_{2})  \over L^{(p)}(1, \eta)}
W_{0, r}^{\circ}(g, u, \phi_{2}; \xi) .
\end{align}

A simple calculation shows that for all $v$ and $a\neq 0$,
\beq\lb{ay}
W_{a,r,}^{\circ}( \smalltwomat y{}{}1,u, \phi_{2,v}; \xi_{v})= \eta\xi^{-1}(y)|y|^{1/2}W_{ay,r,}^{\circ}( 1, y^{-1}u, \phi_{2,v}; \xi_{v}).\eeq
We will sometimes drop $\phi_{2, v}$ from the notation.

The following sufficient condition for cuspidality will simplify matters a little later on. 
\begin{lemm}\lb{eis cusp} 
 Assume that there is a place $v\nmid p\infty$, at which $\xi_{v}$ is unramified, such that
\beq \lb{phi cond}\phi_{2,v}(0, u)=0\eeq
for all $u$. Then  for all $g={\smalltwomat yx{}1}$ with $y\in \A^{+}, x\in \A$, we have
$$W_{0,r}(g, u, \phi_{2}; \xi)=0.$$ 
\end{lemm}
\begin{proof} This is a special case of \cite[Proposition 6.10]{yzz}. 
\end{proof}

\subsubsection{Archimedean Whittaker functions} We compute them explicitly based on our explicit choice of Schwartz function.
\begin{lemm}\lb{eis exp L} Let $v\vert \infty$, let $\xi_{v}(x)=x^{k_{0}}$ for some $k_{0}\in \Z$, and let $\phi_{2,v}(x, u):= \one_{\R^{+}}(u) (uq(x))^{k}e^{-2\pi uq(x)}$ for some $k\in \Z_{\geq 0}$ with $k\geq -k_{0}$. Let $a\in \R^{\ts}$. 
Then 
$$W_{a,v}^{\circ}( 1,u,  \phi_{2,v, k_{v}};\xi_{v}) =  
  \begin{cases} \displaystyle
 2 (2\pi i)^{k_{0}} {k!  \over (k+k_{0})!}    a^{k+k_{0}} e^{-2\pi a} $$
 & \text{if $a,u>0$}\\
  0 & \text{if $a<0$ or $u<0$ }
\end{cases}
$$
\end{lemm}
\begin{proof} 
We drop the subscript $v$ and write $\delta_{\xi, s}$ for $\delta_{\xi, r,s}$. We have 
$$wn(b) = \twomat {(1 +b^{2})^{-1/2} }{-b(1 +b^{2})^{-1/2 }}{}  {(1 +b^{2})^{1/2} } r_{\theta_{b}}$$
with $e^{ i \theta_{b}}  = (i-b)/  {(1 +b^{2})^{1/2} }$. Then $\delta_{\xi,s} (wn(b)) =  i^{k_{0}} (1+b^{2})^{-s/2}  (1- ib)^{-k_{0}}$.
 Since $L(1,\eta)=\pi^{-1} $
  and $\gamma_{v}=i$, we have 
\beq
i^{-k_{0}}W_{a}^{\circ}(s, 1,u,  \phi_{2, k}) &:=
i^{-1} \pi^{-1}\int_{\R} \delta_{\xi, s}(w n(b) ) r(wn(b))\phi_{2, k}(0, u)\psi(-ab) \, db\\
&= \pi^{-1}\int_{\R} (1+b^{2})^{-s/2}   (1- ib)^{-k_{0}}  \int_{\C}  \phi_{2, k}(x, u) \psi(ubq(x))  d_{u}x \, \psi(-ab) db\\
&= \pi^{-1}  \int_{\R}  (1+b^{2})^{-s/2}  (1- ib)^{-k_{0}} \int_{\C}   u^{k+1}  q(x)^{k} e^{-2\pi u  q(x)} \psi(ubq(x))  d_{1}x\, \psi(-ab) db,
\eeq 
where we recall that $d_{u}x= |u| d_{1}x$ and $d_{1}x$ is twice the usual Lebesgue measure. The integral over $\C$ is 
$$2u^{k+1}\sum_{j=0}^{k} \int_{\R} \int_{\R}{k\choose j} x_{1}^{2j}x_{2}^{k-j} e^{-2\pi u  (1-ib) (x_{1}^{2}+x_{2}^{2})} dx_{1}dx_{2}.$$
Since $\int_{\R} x^{2j}e^{-Ax^{2}} \, dx = {A^{-j-1/2}\Gamma(j+1/2)}$, this equals 
\beqq & 2 u^{k} \sum_{j=0}^{k} {k \choose j} \Gamma(j+1/2) \Gamma(k-j+1/2)\cdot (2\pi u )^{-k-1}
(1-ib)^{-k-1}
= 2^{-k} \pi^{-k} k! \, (1-ib)^{-k-1}
\eeqq
by the combinatorial  identity (see Appendix \ref{double fact})
\beq\lb{dfg} \sum_{j=0}^{k} {k \choose j} \Gamma(j+1/2) \Gamma(k-j+1/2) = \pi k!.\eeq
 Therefore, when $u>0$,
\beqq
W_{a}^{\circ}(s, 1,u,  \phi_{2, k})& =  i  2^{-k} \pi^{-k} k!      \int_{\R} (1+ib)^{-s/2} (1-ib)^{-(s+2k+2k_{0}+2)/2 } e^{-2\pi i a b} db.\eeqq
 The integral  is the same one appearing in \cite[bottom of p. 55]{yzz} with $d=2+2k+2k_{0}$. By \cite[Proposition 2.11]{yzz} (whose normalization differs from ours by $L(1, \eta_{v}\xi_{v})=\pi i $),  we find 
$$W_{a}^{\circ}(0, 1,u,   \phi_{2, k}) = i^{k_{0}}  2^{-k} \pi^{-k-1} k!   {(2\pi)^{1+k+k_{0}}\over \Gamma(1+k+k_{0})}    a^{k+k_{0}} e^{-2\pi a} = 
  2 (2\pi i)^{k_{0}} {k!  \over (k+k_{0})!}    a^{k+k_{0}} e^{-2\pi a} $$
if $a, u>0$, as well as simpler formulas implying the desired ones in the other cases.\end{proof}

We deduce the following. Let 
\beq \lb{Qk def}
Q_{k_{0}, k}(X)= \sum_{j=0}^{k}  {k\choose j} {(k+k_{0})! \over (j+k_{0})!} (-X)^{k-j},\eeq
which satisfies $Q_{k_{0}, k}(0)=1$.

\begin{prop} \lb{W infty E} Let $v\vert\infty$, let $a\in\R$, and  let $k_{0}\in \Z$, $k\in \Z_{\geq 0}$ with $k\geq -k_{0}$. Then for $a\neq 0$ we have 
$$W_{a,v}( \smalltwomat y{}{}1,u,  \phi_{2,v ,k_{v}}; \xi_{v}) =  (-1)^{k}
 \eta_{v}\xi_{v}^{-1}(y) |y|^{1/2}
  \cdot
  2 (ay)^{k+k_{0}} Q_{k_{0},k}((4\pi ay)^{-1}) e^{-2 \pi a y}  
  $$
{if $ay>0$, $uy>0$}, and $W_{a,v}(0, \smalltwomat y{}{}1,u,  \phi_{2,v ,k_{v}}; \xi_{v})=0$ otherwise.
\end{prop}
\begin{proof} After recalling the definition of $P_{k_{0}, k}$ in \eqref{Lag def}, by Lemma \ref{eis exp L} and \eqref{ay} we find the asserted vanishing and that for $ay, uy>0$ we have  (dropping subscripts $v$):
$$W_{a,v}( \smalltwomat y{}{}1,u,  \phi_{2, k}; \xi) =  
 \eta\xi^{-1}(y) |y|^{1/2}
  \cdot
  2  (4\pi)^{-k} (-4\pi)^{-k_{0}} (k+k_{0})! \sum_{j=0}^{k}  {k\choose j}  {(k+k_{0})! \over (j+k_{0})!} (-4\pi ay)^{j+k_{0}}  e^{-2\pi ay} ,$$
  which is equal to the asserted formula.
\end{proof}

\begin{coro} \lb{E delta}
Let $\xi\colon F^{\ts}\bks \A^{\ts}\to \C^{\ts}$ with $\xi(x_{\infty})=x_{\infty}^{k_{0}}$ for some $k_{0}\in \Z$. For each $k\in \Z_{\geq 0}^{\Sg_{\infty}}$ with $k\geq -k_{0}$, we have
$${E}_{r}(g, u, \phi_{2}^{\infty}; \xi, k) = (-1)_{\infty}^{k} \delta^{k} {E}_{r}(g, u, \phi_{2}^{\infty}; \xi, 0).$$
\end{coro}
\begin{proof} 
This follows from Proposition \ref{W infty E} and \eqref{delta d}.
\end{proof}

\subsubsection{Schwartz function at $p$} Let  $U_{F,p}^{\circ} \subset \OO_{F, p}^{\ts}$ be as fixed in \S~\ref{wt H}, and let $\kappa_{2}'\colon U_{F,p}^{\circ}\to \C^{\ts}$ be a smooth character.
 We define
\beq\lb{phi2p}
\phi_{2,p,\kappa_{2}'}(x, u):= \one_{{\bf V}^{\circ}_{2,p}}(x)\cdot \delta_{U_{F,p}^{\circ}}(u) \kappa_{2}'(u),\eeq
where ${\bf V}^{\circ}_{2,p}\subset {\bf V}_{2,p}$ is a fixed $\OO_{F, p}$-lattice and  $\displaystyle \delta_{U_{F,p}^{\circ}}(u):= {\vol(\OO_{F, p}^{\ts})\over \vol(U_{F,p}^{\circ})}\cdot \one_{U_{F,p}^{\circ}}(u)$. It is invariant under $N(\OO_{F,p})$.
Let 
\beq\lb{Errr}
E_{r}( \phi_{2}^{p\infty}; \xi, \kappa_{2}', k):= E_{r}(\phi_{2}^{p\infty}\phi_{2,p,\kappa_{2}' }\phi_{2, \infty, k};\xi), 
\eeq
and denote its normalised Whittaker functions by
$$W_{a, r, v}^{\circ}(g, u, \phi_{2}^{p\infty}; \xi,\kappa_{2}', k);$$
depending on the place $v$ we will drop the unnecessary elements from the notation.

\subsubsection{Non-archimedean Whittaker functions} We study the functions $W_{a,v}^{\circ}$ and the $q$-expansion of $E_{r}$.
\begin{prop}\label{Wv E} 
Let $v$ be a non-archimedean place of $F$. 
\begin{enumerate}
\item\label{vnotp} If $v\nmid p$, then $W_{a,v,r}^{\circ}=W_{a,v}^{\circ}$ does not depend on $r$, and for all $a\in F_{v}$
 $$W_{a,v}^{\circ}(1,u, \xi)=|d_{v}|^{1/2}L(1, \eta_{v}\xi_{v}) (1-\xi_{v}(\vpi_{v}))
\sum_{n=0}^{\infty} \xi_{F,v}(\vpi_{v})^{n} q_{F,v}^{n}\int_{D_{n}(a)}\phi_{2,v}(x_{2},u) \, d_{u}x_{2},$$
where $d_{u}x_{2}$ is the self-dual measure on $({\bf V}_{2,v}, uq)$ and 
$$D_{n}(a)=\{ x_{2}\in {\bf V}_{2,v}\, |\, uq(x_{2})\in a+p_{v}^{n}d_{v}^{-1}\}.$$
(When the sum is infinite, it is to be understood in the sense of analytic continuation from characters $\xi|\cdot|^{s}$ with $s>0$.)
\item For all finite places $v$, $|d|_{v}^{-3/2}|D_{v}|^{-1/2}W_{a,v}^{\circ}(1, u, \xi)\in \Q[\xi, \phi_{2,v}]$, and for almost all $v$ we have
$$|d|_{v}^{-3/2}|D_{v}|^{-1/2}W^{\circ}_{a,v}(1,u, \xi)=
\begin{cases} 1 &  \textrm{\ if \ } v(a)\geq - v(d_{v})\textrm{\ and \ } v(u)=-v(d_{v}) \\
0 & \textrm{\ otherwise.}
\end{cases}$$
\item\label{whitt-p} If $v\vert p$,  then
$$ W_{a,r, v}^{\circ}(1, u; \xi, \kappa_{2}')=
\begin{cases} |d_{v}|^{3/2}|D_{v}|^{1/2}  \xi_{v}(-1) \kappa_{2,v}'(u) & \textrm{\ if \ } v(a)\geq -v(d) \textrm{\ and \ } u\in U_{F, v}^{\circ}, \\  
 0 & \textrm{\ otherwise.} 
\end{cases}
$$
\end{enumerate}
\end{prop}
\begin{proof}
See \cite[Propostion 3.2.3, Lemma 3.2.4]{pyzz} for parts 1 and 2. For part 3, we drop subscripts $v$ and  compute 
$$\delta_{\xi, r}(wn(b) w_{r})=\xi(-1)\one_{\OO_{F}}(b), \qquad \gamma_{u}^{-1} r(wn(b))\phi_{2, \kappa_{2}'}(0, u) = |d_{v}| |D_{v}|^{1/2} \one_{\OO_{F}^{\ts}}(u)\kappa_{2}'(u)$$
(the latter if $b\in \OO_{F}$), so that 
$$W_{a}^{\circ}(1,u;\xi, \kappa_{2}') =  |d_{v}| |D_{v}|^{1/2} \xi(-1)\one_{\OO_{F}^{\ts}}(u)\kappa_{2}'(u) \int_{\OO_{F}}\psi(-ab)\, db,$$
which gives the asserted value.
\end{proof}

\begin{coro} \lb{Eis FW} For $g=\smalltwomat yx{}1 $ with $y\in \A^{+}$, $x\in \A$, we have
\begin{multline*}
{E}_{r} (g,u,\phi_{2}^{\infty}; \xi, k) =   {-\eps({\bf V}_{2}) \over L^{(p)}(1, \eta)} \eta\xi^{-1}(y) |y|^{1/2} \\
\cd \left(W_{0}^{\circ}(y, u)+
\sum_{a\in F^{+}}  2^{[F:\Q]} (ay_{\infty})^{k+k_{0}} 
 W_{a,r}^{\circ , \infty}(0, 1, y^{-1}u, \phi_{2}, \xi)  Q_{k_{0}, k}((4\pi a y_{\infty})^{-1}) \, \qqq^{a}\right)
\end{multline*}
where $W_{a,r}^{\circ , \infty}(0, 1, u, \phi_{2}; \xi) = \prod_{v\nmid \infty}W_{a,r,v}^{\circ}(0, 1, y^{-1}u, \phi_{2, v}; \xi)$.
\end{coro}

Let $\phi_{2}^{p\infty}$ satisfy  \eqref{phi cond}, so that by Lemma \ref{eis cusp}, the corresponding Eisenstein series is cuspidal.
For $\xi$ a locally algebraic $p$-adic character of $\A^{\ts}$ of weight $k_{0}$,
consider the (bounded) sequence of coefficients in $\Q_{p}(\xi)$
$${\rm E}_{r}(u, \phi^{\infty}_{2}; \xi, k) = \left( \lm \cdot \eta\xi^{-1}(y) 2^{[F:\Q]}  (by_{p})^{k+k_{0}}  |D_{F}|^{1/2}|D_{E}|^{1/2}W^{\circ, \infty}_{by}(1, y^{-1}u, \phi_{2}^{\infty};\xi) \right)_{b\in F^{+}}.
$$
where $\lm =  {-\eps({\bf V}_{2}) / |D_{E/F}|^{1/2} L^{(p)}(1, \eta)} $.  This is the $p$-adic $q$-expansion attached to $E_{r}$. 
Analogously to Corollary \ref{E delta}, we have
\beq \lb{E d}
{\rm E}_{r}(u, \phi_{2}; \xi, k) = (-1)_{\infty}^{k} d^{k}  {\rm E}_{r}(u, \phi_{2}; \xi, 0).\eeq

\subsubsection{$p$-adic interpolation of Whittaker functions} For a place $v\nmid p\infty $ of $F$, denote  $\Y_{v}:= \Spec \Q_{p}[F_{v}^{\ts}]$ and (for a later use) $\Y_{\H, v}:=\Spec \Q_{p}[E_{v}^{\ts}]$, the spaces of characters of $F_{v}^{\ts}$ and respectively  $E_{v}^{\ts}$. We say that a meromorphic function $\Phi$ on an integral scheme \emph{has poles controlled by} the (nonzero) meromorphic function $\Phi'$ if $\Phi/\Phi'$ is regular.
\begin{prop}\label{Wv E int} Let $v\nmid p\infty$.  For each $a\in F_{v}^{\ts}$, $y\in F_{v}^{\ts}$, and $\phi_{2,v}\in \calS({\bf V}_{2}^{p\infty}\ts \A^{p\infty, \ts}, \Q_{p})$
there is 
a meromorphic function
	$$\W_{a,v}^{\circ}(y,u, \phi_{2,v})\in \mathscr{K}(\Y_{v}),$$ regular if $\phi_{2, v}$ is standard and otherwise with poles controlled by  $L(1,\eta_{v}\xi_{v})$, satisfying
 	$$\mathscr{W}_{a,v}^{\circ}(y, u, \phi_{2,v};\xi_{v}) = |d_{v}|^{-3/2}|D_{v}|^{-1/2} W_{a,r,v}^{\circ}(\smalltwomat {y_{v}}{}{}1, u,\phi_{2,v}(\xi_{v});  \xi_{v})$$
	for all $\xi_{v}\in \Y_{v}(\C)$ whose underlying scheme point is not a pole.
\end{prop}
\begin{proof} Part 1 is proved as in \cite[Lemma 3.3.1]{pyzz}, except that we write the arbitrary $\phi_{2,v}=c\phi_{2, v}^{\circ}+\phi_{2,v}'$ without the extra factor of equation (3.3.2) \emph{ibid.} Then the argument shows that (only) when $\phi'_{2,v}\neq 0$, there may be a pole controlled by $L(1, \eta_{v}\xi_{v})$. 
\end{proof}

\subsection{Theta-Eisenstein family}\lb{47}
 Fix a compact open subgroup  $U^{p}\subset\G(\A^{p\infty})$ (which will be usually omitted from all the notation), and let $U_{F}^{p}:= U^{p} \cap \A^{p\infty, \ts}$.  Let $\phi^{p\infty}\in \calS({\bf V}^{p\infty}\times \A^{p\infty, \times})$ be a Schwartz function fixed by $U^{p}$.   Let $\xi\colon F^{\ts}\bks \A^{\ts}\to \C^{\ts}$ be a locally algebraic character fixed by $U_{F}^{p}$ and such that $\xi(x_{\infty})=x_{\infty}^{k_{0}}$ for some $k_{0}\in \Z$, and let $k\in\Z_{\geq 0}^{\Sg_{\infty}}$ satisfy $k_{v}+k_{0}\geq 0 $ for all $v$.

We fix a choice of a Schwartz function in $\calS({\bf V}_{2, p}\ts F_{p}^{\ts})$ as follows.
Let $U_{F, p}^{\circ}\subset \OO_{F,p}^{\ts}$ be as fixed in \S~\ref{wt H}.
For $r\in \Z_{\geq 1}^{S_{p}}$ and $\kappa_{1}'\colon U_{F, p}^{\circ} \to \C^{\ts}$ a smooth character, we define
\beq \lb{phi1p}
\phi_{1, r, \kappa_{1}', p}(x_{1}, u)&=\delta_{1, r, p}(x) \one_{U_{F, p}^{\circ}} \kappa_{1}' (u):=
 \prod_{v\vert p} \phi_{1, r_{v}, \kappa_{1}', v}(x_{1,v}, u_{v}), \\
\phi_{1, r_{v}, \kappa_{1}', v}(x_{1,v}, u_{v}) &:={ \vol(\OO_{E, v}, dt)\over \vol(1+\vpi_{v}^{r_{v}}\OO_{E,v}, d^{\ts}t)}    \one_{1+\vpi_{v}^{r_{v}}\OO_{E,v}}(x_{1,v}) \one_{U_{F, p}^{\circ}}(u)\kappa_{1}'(u),
 \eeq
 which is invariant under $N(\OO_{F,p})=\prod_{v\vert p} N(\OO_{F,v})$.

Let $\kappa_{2}'\colon U_{F, p}^{\circ}\to \C^{\ts}$ be a smooth character. For $t\in \A_{E}^{\ts}$, $r\geq\underline{1}$, and  $\phi^{p\infty}=\phi_{1}^{p\infty}\otimes\phi_{2}^{p\infty}\in \calS({\bf V}^{p\infty}\ts \A^{p\infty, \ts})$,  define a form in $N_{(l_{0}-k_{0}, {2}+ l + k_{0}+2k)}^{\leq k}(\C)$ by
\begin{align*}
I_{r}(t, \phi^{p\infty};  \kappa_{1}',\ul,\xi, \kappa_{2}', k)
&:=
{ |D_{F}|^{-1/2}}\cdot 
 \theta(u, r_{1}(t, 1)\phi_{1}^{p\infty}\phi_{1, \kappa_{1}',r, p} ; \ul) \star{E}_{r}(q(t)u, \phi_{2}^{p\infty};\xi, \kappa_{2}', k).
\end{align*}
where the product $\star$ is \eqref{star def}, and $E_{r}(\cd)=\eqref{Errr}$.

Fix a compact open subgroup   $U_{\H}^{p}\subset\A_{E}^{p\infty, \ts}$  (which  will be omitted from all the notation). Let $\ul$ be a complex weight for $\H$, let $\chi\colon E^{\ts}\bks\A_{E}^{\ts}\to \C^{\ts}$ be  a locally algebraic character of weight $\ul$ fixed by $U_{\H}^{p}$,
 and assume that for all $w\vert v\vert p$, the integer $r_{v}\geq 1$ is greater than the conductors of $\chi_{w}$, $\xi_{v}$, $\kappa_{2,v}'$. Then we   define
\beq\lb{Iddef}
 I(\phi^{p\infty}; \chi, \xi, \kappa_{2}', k) := 
\int^{*}_{E^{\ts}\bks \A_{E}^{\ts}/\A^{\ts} }\chi(t) I_{r}(t,\phi^{p\infty}; \kappa_{1,\chi,p}',   \ul,\xi,\kappa_{2}',  k)\, dt,
\eeq
which does not depend on the choice of $r$; here $\kappa_{1,\chi, p}'$ is as in \eqref{kp'}, namely 
\beq \lb{kp1} \kappa_{1,\chi,p}':=\chi_{p} \circ j' \eeq
for $j'\colon U_{F, p}^{\circ}\to \OO_{E, p}^{\ts}$ as in \eqref{j wt h}.

\begin{lemm}\lb{Lemma FW I}  For each $c\in \A^{+}$ satisfying $v(c)\geq 1$ for some $v\vert p$, the $c^{\rm th}$
 Whittaker--Fourier coefficient of $ I(\phi^{p\infty}; \chi, \xi,\kappa_{2}',  k) $ is 
\begin{multline*}
W_{ I}^{\C}(c) =  {-\eps({\bf V}_{2})\cdot 2^{[F:\Q]} \over L^{(p)}(1, \eta)} {|D_{E/F}|^{1/2}\over |D_{F}|^{1/2}} \cdot
 \sum_{a/c\in F, \ 0<a/c<1}
\chi^{\infty}(a) a_{\infty}^{(| l | +l_{0})/2}
 \xi^{\infty, -1}(c-a)(a-c)_{\infty}^{k}
 \\
  \prod_{v\nmid p\infty} J_{v}(a, c, \phi_{v}; \chi_{v}, \xi_{v}) \prod_{v\vert p} J_{v}(a, c; \chi_{v}, \xi_{v}, \kappa_{2,v}').
\end{multline*}
Here, taking $\phi_{1, v} = \phi_{1, r_{v},\kappa_{1,\chi,v}'}$ with $ \kappa_{1,\chi,v}':= \chi_{v}\circ j_{v}$ and $\phi_{2,v}=\phi_{2, \kappa_{2,v}'}$ if $v\vert p$, we define for all $v\nmid p\infty$
\beq \lb{Jv}
J_{v}(a, c, (\phi_{v});\chi_{v},\xi_{v}, (\kappa_{2, v}')) := \int_{E_{v}^{\ts}}
\chi_{v}(t)  r_{1}(t)\phi_{1, v}(1, a)
W^{\circ}_{c-q(t) a, v}(1, q(t)a, (\phi_{2,v}); \xi, (\kappa_{2,v}'))   
 \, d^{\bullet}t_{v},
\eeq
for the measure $d^{\bullet}t_{v}$ giving volume $1$ to $\OO_{E, v}^{\ts}$. 
\end{lemm}
\begin{proof} We  lighten some of the notation. The assumptions of Lemma \ref{theta FW} are satisfied, therefore for $c\in F^{+}$ and  $g=\smalltwomat yx{}1 $ with $y\in \A
^{+}$, $x\in \A$,
\begin{multline*}
 I(g; \chi, \xi,\kappa_{2}',  k) 
 =  {|D_{E/F}|^{1/2}\over |D_{F}|^{1/2}} \eta(y) |y|^{1/2} y_{\infty}^{ | l |+l_{0}\over 2}  \sum_{a\in  F^{+}} 
   \int_{{\A}_{E}^{\infty,\ts} }   \chi^{\infty}(t)  r_{1}(t)\phi_{1}^{\infty}(y,y^{-1} a)
\ {E}(g,q(t) a) \, d^{\bullet}t\, \qqq^{a} \\
=  {|D_{E/F}|^{1/2} \over |D_{F}|^{1/2}}  |y|^{1/2} \eta\chi(y) y_{\infty}^{( | l | -l_{0})/2}  \sum_{a\in F^{+}} 
   \int_{{\A}_{E}^{\infty,\ts} }   \chi^{\infty}(t)  r_{1}(t)\phi_{1}^{\infty}(1, ay)
\ {E}(g,q(t)ay^{2}) \, d^{\bullet}t\, \qqq^{a} \\
\end{multline*}
Now for $\lm = {-\eps({\bf V}_{2})/ L^{(p)}(1, \eta)}$, by Corollary \ref{Eis FW} we have
$${E}(g,q(t) ay^{2})=  \lm \cdot |y|^{1/2} \eta\xi^{-1}(y) \left(W_{0}(y, u)
+
(-1)_{\infty}^{k}\sum_{b\in F^{+}} 2^{[F:\Q]}  (by_{\infty})^{k+k_{0}} 
 Q_{k_{0}, \underline{k}} ((4\pi b y_{\infty})^{-1}) W^{\circ, \infty}_{by}(1, q(t)ay, \xi) \,\qqq^{b}\right)$$
so that 
$$
I(g; \chi, \xi, k) 
=  |y|  \sum_{c\in F^{\ts}}  W_{I, c}^{\C}(y)
 \qqq^{c} 
$$ 
for some coefficients $W_{I, c}^{\C}(y)$, which  we now explicitly calculate if $c$ satisfies $v(c)\geq 1$ for some $v\vert p$. Under this condition, we have
\begin{multline*}
W_{I, c}^{\C}(y)=  {-\eps({\bf V}_{2})\cdot 2^{[F:\Q]}\over L^{(p)}(1, \eta)} {|D_{E/F}|^{1/2}\over |D_{F}|^{1/2}} 
\sum_{\substack{a\in  F^{}\\ 0<a<c }} (ay)_{\infty}^{(| l |-l_{0})/2}    \int_{{\A}_{E}^{\infty,\ts}}
\chi(ay)  \chi^{\infty}(t)  r_{1}(t)\phi_{1}^{\infty}(1, ay)\\
\cdot (-1)_{\infty}^{k} \xi^{-1}((c-a)y)
((c-a)y_{\infty})^{k+k_{0}}
W^{\circ, \infty}_{(c-q(t) a)y}(1, q(t)ay, \xi)   Q_{k_{0}, \underline{k}} ((4\pi (c-a) y_{\infty})^{-1}) 
 \, dt.
 \end{multline*}
where we have noted that, since we have assumed $v(c)\geq 1$ and the choice of $\phi_{1, p}$ implies $v(a)=0$,  the constant term (corresponding to $a=c$) of the Eisenstein series does not contribute. Finally, we rewrite the resulting formula with $c$ in place of $cy$. 
\end{proof}

\begin{lemm}\lb{J int}
Let $a, c\in \A^{\infty, \ts}$, let $\phi_{2}^{p\infty} \in \calS({\bf V}_{2}^{p\infty}\ts \A^{p\infty, \ts})$, and
let $\xi\colon  F^{\ts}\bks \A^{\ts}\to \C^{\ts}$, $\chi\colon E^{\ts}\bks \A_{E}^{\ts} \to \C^{\ts}$ be locally algebraic characters.
\begin{enumerate}
\item For $v\vert p$, if $v(c)\geq 1$ then 
$$ |d_{v}|^{-3/2} |D_{v}|^{-1/2}\cdot  J_{v}(a, c; \chi_{v}, \xi_{v}, \kappa_{2,v}')  =  
\xi_{v}(-1) \one_{U_{F,v}^{\circ}}(a_{v}) \kappa_{1, \chi, v}'\kappa_{2, v}'(a_{v}).
$$
\item
For all but finitely many $v\nmid p$, 
$$ |d_{v}|^{-3/2} |D_{v}|^{-1/2} J_{v}(a, c, \phi_{v}; \chi_{v}, \xi_{v}) =1.$$
\item For all $v\nmid p\infty$, there is a function ${\bf J}_{v}(a, c, , \phi_{v})\in \OO(\Y_{\H, v}\ts \Y_{v})$ such that for all $(\chi_{v}, \xi_{v})\in \Y_{\H, v}\ts\Y_{v}(\C)$, 
$${\bf J}_{v}(a, c, \phi_{v})(\chi_{v}, \xi_{v})= |d_{v}|^{-3/2} |D_{v}|^{-1/2} J_{v}(a, c, \phi_{v}; \chi_{v}, \xi_{v}).$$
\end{enumerate}
\end{lemm}
\begin{proof}
Part 1 follows from the definitions and   Proposition \ref{Wv E}.3. Part 2 follows from Proposition \ref{Wv E}.2 and a simple calculation. Finally, since the integrand in \eqref {Jv} is compactly supported, part 3 follows from Proposition \ref{Wv E int}.
\end{proof}

Let \beq\lb{lm def} \lm:= {-\eps({\bf V}_{2}) \cdot 2^{[F:\Q]}\over |D_{E/F}|^{1/2} L^{(p)}(1, \eta)} 
\in\Q^{\ts}.\eeq

For the sake of simplicity, we momentarily introduce the assumption that  the weight $\ul$ of $\chi$ satisfies $l \geq 0$. We will see in Corollary \ref{coro Iord 2} that this does not affect   our main construction.

\begin{prop}\lb{prrr}
Let $\chi\in \Y_{\H}^{\cl}$ have weight $\ukappa_{1}=\ukappa_{\H}(\chi)$, and let  $\ukappa_{2}=(\kappa_{2,0}, \kappa_{2})\in \frak W^{\cl}$ . Let $U_{F}^{\circ}:= U_{F}^{p}U_{F, p}^{\circ}$.
Write $\xi=\kappa_{2, 0}^{-1}$ and 
\beq\lb{kappa'}
{}&\kappa_{1, 0, p}=\kappa_{1,0, p}^{\rm sm} (\cd )^{l_{0}}, \qquad &\kappa_{2, 0, p}= \kappa_{2, 0, p}^{\rm sm}(\cd )^{-k_{0}} \\
{}&\kappa_{1}=\kappa_{1}^{\rm sm}(\cd )^{(l+l_{0})/ 2}, \qquad &\kappa_{2}=\kappa_{2}^{\rm sm} (\cd)^{k}.\eeq
Assume that $l_{\tau}, k_{\tau}, k_{\tau}+k_{0}\geq 0$ for all $\tau \in \Sg_{p}$.
Let  $$\kappa_{1}', \kappa_{2}'\colon U_{F, p}^{\circ}\to \C^{\ts}$$  be as in \eqref{kp'}.
Let  $n\in \N$  be sufficiently large (depending on $\xi, \chi$).  
 The coefficients 
\begin{multline}\lb{WI c}
W_{(\Up^{\circ})_{p}^{n!}\, {\rm I}(\phi^{p\infty})}(c)(\chi, \ukappa_{2}) := 
    \lm
    \cd 
 \sum_{\substack{a/c\in F \\ 0<a/c< p^{n!}}}
\one_{U_{F, p}^{\circ}}(a_{p})
 \kappa_{1}(a_{p})\kappa_{2}((a-p^{n!}c)_{p})
\\
\cd\kappa_{1,0}^{p\infty}(a)\kappa_{2, 0}^{p\infty}(-a)  
 \prod_{v\nmid p} {\bf J}_{v}(a, c, \phi_{v}; \xi_{v}, \chi_{v}), 
 \end{multline}
 for  $c\in \A^{+}$ with $v(c)\geq 0$ for all $v\vert p$, define a $p$-adic modular form 
 $$(\Up_{p}^{\circ})^{n!}\, {\rm I}(\phi^{p\infty};\chi, \ukappa_{2})\in {\bf S}( \Q_{p}(\chi, \ukappa_{2}))$$
 such that for every $\iota\colon \Q_{p}(\chi, \ukappa_{2})\into \C$, we have
 $$(\Up_{p}^{\circ})^{n!}{\rm I} (\phi^{p\infty};\chi, \ukappa_{2})^{\iota} =    |D_{E}|^{1/2}|D_{F}| (\Up_{p}^{\circ})^{n!}I(\phi^{p\infty};\chi^{\iota}, \xi^{\iota}, \iota\kappa_{2}', k).$$
 
 Moreover,  if $\phi_{2}^{p\infty}$ satisfies \eqref{phi cond}, then
\beq \lb{star id I}
e^{\ord} {\rm I} (\phi^{p\infty};\chi, \ukappa_{2})= e^{\ord, \iota} [e^{\rm hol}  ( |D_{E}|^{1/2}|D_{F}| (\Up_{p}^{\circ})^{n!}I(\phi^{p\infty}; \chi^{\iota}, \xi^{\iota}, \iota\kappa_{2}', k)] .\eeq 

 \end{prop}
 \begin{proof} From Lemmas \ref{Lemma FW I}, \ref{J int}, we find an expression which can be brought into the above form, after replacing  $p^{n!}c$ with  $c$ whenever it occurs as the argument of a smooth function.  
 
 The second assertion follows from \eqref{e ord hol}, Corollary \ref{E delta} and \eqref{E d}.
  \end{proof}

With notation as in Proposition \ref{prrr}, let  $${\bf J}^{p}(a, c, \phi^{p\infty}; \chi, \kappa_{2, 0}^{-1}):=    
\prod_{v\nmid p}
\kappa_{2,0,v}^{-1} (a)\chi_{ v} (-a)  {\bf J}_{v}(a, c, \phi_{v}; \chi_{v}, \kappa_{2, 0,v}^{-1}).$$
 It is easy to verify, using only that   ${\bf J}^{p}(\cd)$ is a Schwartz function of $a\in \A^{p\infty, \ts}$, that the Riemann sums 
 $$\mu_{n}({\bf J}^{p}, c, \phi^{p\infty})(\vphi) := \sum_{a/c\in F^{\ts}}^{} \vphi(a) \cdot \one[0<a/c< p^{n!}]\one[a\in U_{F, p}^{\circ}]\cdot  {\bf J}^{p}(a, c,\phi^{p\infty};  \chi,\kappa_{2, 0}^{-1}),$$ 
for  $\vphi\colon U_{F, p}^{\circ}\to \Q_{p}(\kappa_{2}, \xi)$ locally constant, converge to a  measure (bounded distribution)
$$\mu({\bf J}^{p}, c, \phi^{p\infty}; \chi^{p\infty}, \ukappa_{2}^{p\infty}) $$ 
 on $U_{F, p}^{\circ}$  valued in $\Q_{p}(\chi, \ukappa_{2})$. Then by the same argument  of the standard result in \cite[Theorem 6 on p. 39]{koblitz}, any continuous  $\Q_{p}(\chi, \ukappa_{2})$-valued function is integrable for this measure. The same holds with $\Q_{p}(\chi, \ukappa_{2})$ replaced by $\OO(\Y_{\H}\hat{\ts} \frak W)$ and $\chi, \ukappa_{2}$ by the universal characters, ${\bf J}^{p}$ by the universal function; for this universal situation,  we will use the same notation without $\chi, \ukappa_{2}$. 
  
  \begin{coro}\lb{coro Iord} Let ${S}^{\rm bad}$ be the set of places of $F$  at which $\phi^{p\infty}$ is not the standard Schwartz function.
Denote  $\ukappa_{1}=\ukappa_{\H}(\chi_{\rm univ})$.    Assume that $\phi_{2}^{p\infty}$ satisfies \eqref{phi cond}. The  sequence of coefficients 
\beqq
  W_{{\bf I}^{\ord}(\phi^{p\infty})}(c)(\chi, \ukappa_{2})&:= 
    \lm\cd \one_{\O_{F,p}}[c_{p}]\cd
    \int_{U_{F, p}^{\circ}}  \kappa_{1}\kappa_{2}(a) \, d\mu({{\bf J}^{p}, c, \phi^{p\infty}})(a) \in \mathscr{K}(\Y_{\H}\hat{\ts} \frak W), \qquad c\in \A^{\infty,+},
\eeqq
has  poles controlled by $\prod_{v\in S^{\rm bad}}L(1, \eta_{v}\kappa_{2, 0,v}^{-1})$. It defines an ordinary meromorphic  $\Y_{\H}\hat{\ts}\frak W$-adic modular form (\S\ref{244})
$${\bf I}^{\ord}(\phi^{p\infty};\chi, \ukappa_{2})$$
 of weight $\ukappa_{1}\ukappa_{2}$,
which satisfies the following property. 
 
For all $\chi\in \Y_{\H}^{\cl}(\C)$, $\ukappa_{2}\in\frak W^{\cl}(\C)$ with underlying numerical weights $\ul$, $\uk$ such that   $l_{\tau}, k_{\tau}, k_{\tau}+k_{0}\geq 0$ for all $\tau$,
 we have
\beq\lb{I ord int}
{\bf I}^{\ord}(\phi^{p\infty};\chi, \ukappa_{2})=  e^{\ord, \iota}  |D_{E}|^{1/2}|D_{F}| e^{\rm hol} I(\phi^{p\infty};\chi^{\iota}, \kappa_{2, 0}^{-1, \iota}, \iota\kappa_{2}', k),
\eeq
where $\iota\colon \Q_{p}(\chi, \ukappa_{2})\into \C$ is the embedding attached to the complex geometric point $(\chi, \ukappa_{2})$. 
    \end{coro}
\begin{proof}
The interpolation property \eqref{I ord int} at the level of $q$-expansions follows from Proposition \ref{prrr} and the previous discussion.  
The simplification in the argument of $\kappa_{2}$ in the interpolated coefficient \eqref{WI c} is justified by the fact that $\kappa_{2}(a-p^{n!}c)-\kappa_{2}(a)\to 0$ uniformly in $a$, and that the expression of interest is a bounded function of $\kappa_{2}(\cd)$.   Lemma \ref{244} then shows the existence of the $\Y_{\H}\hat{\ts}\frak W$-adic modular form ${\bf I}^{\ord}(\phi^{p\infty};\chi, \ukappa_{2})$.
 \end{proof}
 
Consider the weight map 
\beqq
\vphi\colon (\Y_{\G}\hat{\ts} \Y_{\H})&\to \frak W\\
(x, y)&\mapsto \ukappa_{x}.
\eeqq
Recycling notation (in a way that should cause no confusion), 
define 
 an ordinary meromorphic $(\Y_{\G}\hat{\ts} \Y_{\H})$-adic modular cuspform of weight $\vphi$ by
\beq \lb{I x y}
{\bf I}^{\ord}(\phi^{p\infty}; x, y):= {\bf I}^{\ord}(\phi^{p\infty}, \chi_{y}, \ukappa_{x} \ukappa_{y}^{-1}),\eeq
where the right-hand side is the form of Corollary \ref{coro Iord}.  We denote by 
$${\bf I}^{\ord, \curlyvee}(\phi^{p\infty},x, y)$$
the pullback of ${\bf I}^{\ord}(\phi^{p\infty},x, y)$ under the involution $\Y_{\G}\ts\Y_{\H}\stackrel{\curlyvee\ts \id}{\longrightarrow} \Y_{\G}\ts\Y_{\H}$.

 \begin{coro}\lb{coro Iord 2}   Assume that $\phi_{2}^{p\infty}$ satisfies \eqref{phi cond}. 
   For all $x\in \Y_{\G}^{\cl}(\C)$ and $y \in \Y_{\H}^{\cl}(\C)$,
of weights  $\uw$, $\ul$ such that for all $\tau\in \Sg_{\infty}$, 
$$ | l_{\tau}| \leq w_{\tau}-2, \qquad |w_{0}+l_{0}|\leq w_{\tau} -2-  | l_{\tau} |, $$
 we have
\beq\lb{I ord int 2}
{\bf I}^{\ord, \curlyvee}(\phi^{p\infty}; x, y)=  e^{\ord, \iota}  |D_{E}|^{1/2}|D_{F}| e^{\rm hol} I(\phi^{p\infty};\chi^{\iota}, \xi_{x, y}^{\iota}, \iota \kappa_{2,x,y}', k_{x,y}),
\eeq
where
\begin{itemize}
\item $\xi=\xi_{x,y} = \omega_{{x}} \omega_{y}$, whose weight we denote by $k_{0}$;
\item $k_{x,y} = (w-  2 - | l | - k_{0}) /2 $;
\item $\kappa_{2,x,y}' = \kappa_{{x}}^{\curlyvee}{}' \cd \kappa_{y}'^{-1}$ (with notation as in \eqref{kp'});
\item  $\iota\colon \Q_{p}(x, y)\into \C$ is the embedding attached to the complex geometric point $(x, y)$. 
\end{itemize}
  \end{coro}
\begin{proof}
The interpolation property at characters satisfying $l\geq 0$ follows from Proposition \ref{prrr} via  Corollary \ref{coro Iord}; the same argument also goes through without the assumption $l\geq 0$, since the weight of the chosen theta-Eisenstein series does not depend on $\ul$ (or $y$) but only on $x$.
  The inequalities on the weights come from the conditions $k, k+k_{0}\geq 0$.
\end{proof}

\section{Zeta integrals}
\lb{sec 4}

In this final section, we interpolate  global and local (away from $p\infty$) zeta integrals, compute the archimedean and $p$-adic integrals, and construct the $p$-adic $L$-function. 

As preliminary, we recall gamma factors introduced in the introduction.
Let $F_{v}$ and $L$  be  $p$-adic fields. The  (inverse) Deligne--Langlands 
gamma factor of a  potentially semistable representation $\rho$ of $\Gal(\baar{F}_{v}/F_{v})$ over $L$, with respect to a 
nontrivial character $\psi_{v}\colon F_{v}\to\C^{\ts}$ 
 and  an embedding $\iota\colon L\into \C$,  is defined as 
\beqq
\gamma(\iota \rho,\psi_{v})^{-1}:= {L(\iota {\rm WD}(\rho))\over \eps(\iota {\rm WD}(\rho), \psi_{v}) L(\iota {\rm WD}(\rho^{*}(1)))},\eeqq
where $\iota {\rm WD}$ is Fontaine's functor \cite{fontaine} to complex Weil--Deligne representations.
We also define $\gamma(s, W, \psi_{v}):= \gamma(W\ot |\cd|^{s}, \psi_{v})$. 

Let $\Q_{p}^{\ab}$ be the abelian closure of $\Q_{p}$. If $W_{\chi}$ is  the Weil--Deligne representation corresponding to a smooth character $\chi\colon F_{v}\to \Q_{p}^{\rm ab, \ts}$, then for any $\sg\in G^{\ab}_{\Q_{p}}$ corresponding to $a\in \Q_{p}^{\ts}$ under the reciprocity map, we have
\beq \lb{equiv gamma}
\gamma(W_{\chi},\psi_{v})^{\sg} = \chi(a) \gamma(W_{\chi^{\sg}}, \psi_{v}).\eeq

For now until the final \S~\ref{sec final}, we fix an embedding 
\beq\lb{iotab}
\iota^{\ab}\colon \Q_{p}^{\ab}\into \C, \eeq by which we  identify the fixed standard character $\psi^{\infty}\colon  \A^{\infty}\to \C$ with one valued in $\Q_{p}^{\ab}$ (still denoted by $\psi^{\infty}$).

\subsection{Petersson product}\lb{Pet}

Let $\pi$ be an  ordinary  automorphic representation of $\G(\A)$ over a finite extension $L$ of $\Q_{p}^{\ab}$.
For $v \vert p$, let $\omega_{\pi,v}$, $ \alpha_{\pi ,v}\colon F_{v}^{\ts}\to L^{\ts}$ be the central character and, respectively, $\Up_{v}$-eigencharacter of $\pi_{v}$.

Define $\gamma( \ad (V_{\pi,p})(1)^{++}, \psi_{p}):=\prod_{v\vert p} \gamma( \ad (V_{\pi,v})(1)^{++}, \psi_{v})$, where $\ad (V_{\pi,p})(1)^{++}$ is the character $\omega_{\pi, v}^{-1}\alpha_{\pi, v}^{2}|\cd|^{2}$ of $F_{v}^{\ts}$.  For $\iota\colon L\into \C$, let 
 $$e_{p}(\ad(V_{\pi^{\iota}})(1)):= {\gamma(\iota \ad (V_{\pi,p})(1)^{++}, \psi_{p})^{-1} \zeta_{F, p}(2) \over L(1, \pi_{p}^{\iota}, \ad) }.$$ 
 
 Let $(, )$ denote the Petersson pairing 
$$(f, f') :=\int_{\A^{\ts}\bks\G(\A)} f(g)f'(g)\, dg$$
of automorphic forms on $\G(\A)$. 

\begin{lemm} \lb{pet expl} Let $\pi$ be an  ordinary cuspidal automorphic representation of $\G(\A)$ of weight $\uw$ over  a finite extension $L$ of $\Q_{p}^{\ab}$.
There exists a bilinear pairing
$$\la\, , \, \ra\colon \pi^{\ord}\ot_{L} N_{\uw^{\vee}}(L)  \to L$$
such that  for all $\iota\colon L\into \C$ extending $\iota^{\ab}=\eqref{iotab}$  and all sufficiently large $r\in \N^{S_{p}}$, 
\beq\lb{pettt}
\iota \la f ,g  \ra =
{ {|D_{F}|^{1/2} \zeta_{F}(2)
\over 
\omega_{\pi^{\iota}, p}(-1)  e_{p}(\ad(V_{\pi^{\iota}})(1)) \cdot 2^{\sum_{\sg}-1- w_{\sg}}\cdot    L(1, \pi^{\iota}, \ad)} 
 {q_{F,p}^{r}(w_{r,p} \Up_{p}^{-r} f^{\iota, \rm a}, {g}^{\iota})} }.
 \eeq
The pairing $\la\, , \, \ra$ satisfies the following properties.
\begin{enumerate}
\item For all $f\in \pi^{\ord}$, $g\in N_{\uw^{\vee}}(L)$, we have $\la f ,g \ra =\la f, e^{\ord} e^{\hol} g\ra$;
\item If $f_{0}\in \pi^{\infty}$, $f_{0}^{\vee}\in \pi^{\vee,\infty}$ are ordinary forms, new at places away from $p$,  holomorphic at the infinite places, and with first Fourier coefficients equal to $1$, then 
\beq \lb{eq pet ex} 
\la f_{0},  f_{0}^{\vee} \ra = c_{\pi^{\infty}} \eeq
for some constant $c_{\pi^{\infty}}\in L^{\ts}$ depending only on the Bernstein components and the monodromy of $\pi_{v}$ for all $v\nmid p\infty$.
\end{enumerate}
\end{lemm}
\begin{proof} The existence and \eqref{eq pet ex} follow from the  factorization of the Petersson inner product into parings in the Whittaker models \cite[Proposition 2.1]{CST}, together with the local calculations of \cite[Proposition 3.11]{CST} away from $p$ and \cite[Lemma A.3.3]{univ} at $p$.
 Since the elements $f^{\iota, \rm a}$ in the right hand side of \eqref{pettt}  are antiholomorphic, it is clear that the pairing factors through $e^{\hol}$. 
\end{proof}

\begin{prop}\lb{lala}
Let $\X_{\G}\subset \Y_{\G}$ be a Hida family of tame level $U^{p}$, let $S$ be a finite set of places such that $U^{Sp}$ is maximal, and let $\Pi= \Pi_{\X_{\G}}^{U^{Sp}}$. 
There is a unique $\OO({\Y_{\G, \Q_{p}^{\ab}}})$-bilinear pairing
$$\la\la\ , \ \ra\ra \colon \Pi\ot_{\OO(\Y_{\G})} {\cS}^{\curlyvee}(\Y_{\G})_{\Q_{p}^{\ab}} \to \cK(\X_{\G, \Q_{p}^{\ab}})$$
such that for all $x\in\X_{\G, \Q_{p}^{\ab}}^{\cl}$, corresponding to an ordinary representation $\pi=\pi_{x}$ over $\Q_{p}^{\ab}(x)$, 
and for all $\bff\in \Pi_{\Q_{p}^{\ab}}$, ${\bf g}\in {\cS}^{\curlyvee}(\Y_{\G})_{\Q_{p}^{\ab}}$, we have
$$\la\la \bff, {\bf g}\ra\ra(x) = \la \bff_{x},{\bf g}_{x}\ra.
$$
\end{prop}
\begin{proof} 
 The construction  is very similar to that of the pairing denoted by $H^{-1}l_{\lm}$ in \cite[p. 380]{Hi}. In this case, let $\bff_{0}^{\curlyvee}=\eqref{new bfff}$ be the normalised primitive form in $\Pi^{\curlyvee}$, let $U_{0}^{p}\subset \G(\A^{p\infty})$ be a maximal open compact subgroup fixing $\bff_{0}^{\curlyvee}$,  and let 
$$e_{\bff_{0}^{\curlyvee}}\colon \cS^{\curlyvee}\to \cK(\X_{\G}) \, \bff_{0}^{\curlyvee}$$ 
be the unique $\cH_{U_{0}^{p}}^{\rm sph}$-equivariant idempotent that factors through the idempotent projection $\cS^{\curlyvee}\to \cS^{U_{0}^{p},\curlyvee}$. Then we define $\la\la \bff_{0}, - \ra\ra$ by
 $$\la\la \bff_{0}, {\bf g}\ra\ra \bff_{0}^{\curlyvee} = c_{\X_{\G}}^{-1}\cdot  e_{\bff_{0}^{\curlyvee}}({\bf g})$$
where $c_{\X_{\G}}:=c_{\pi^{\infty}}$ for any automorphic representation $\pi$ such that $\pi^{U^{Sp},\ord}\cong\Pi_{|x}$.
for some  $x\in \X_{\G, \Q_{p}^{\ab}}^{\cl}$. Let us explain why this is well-defined independently of $x$. As noted before, $c_{\pi^{\infty}}$  only depends on the Bernstein component and the (rank of the) monodromy of $\pi_{x, v}$ for $v\nmid p\infty$. (In plain terms, the  rank of the monodromy is~$1$ if $\pi_{x, v}$ is a special representation and it is~$0$ otherwise.) The Bernstein component is an invariant of connected families. As for the rank of the monodromy, by the local-global compatibility result of Proposition \ref{LGC}, it is the rank of the  monodromy of the Weil--Deligne representation attached to $\cV_{\G |x}$. Since the latter is pure, the desired constancy along $\X_{\G}$ follows 
from  \cite[Proposition 3.3.1]{LLC}.

In general, we may write $\bff= T\bff_{0}$ for some Hecke operator $T$ supported away from $p$. We then define $\la\la \bff, {\bf g}\ra\ra:= \la\la \bff_{0}, T^{\curlyvee}{\bf g}\ra\ra$.  The interpolation property follows from the definitions, the interpolation property proved in \cite[Lemma 9.3]{Hi}, and  \eqref{eq pet ex}.
\end{proof}

\subsection{Waldspurger's Rankin--Selberg integral} We recall the local and global theory of Waldspurger's \cite{wald} integral representation of our $L$-function.

\subsubsection{Setup}\lb{421}
Let  $y\in \Y_{\H}^{\cl}(\C)$, let  $\chi=\chi_{y}$ be the corresponding character of $E^{\ts}\bks \A_{E}^{\ts}$, and  let $\ukappa_{\chi} \in \frak W^{\cl}(\C)$ be its weight, let $\ul$ be its numerical weight,  and let $\omega_{\chi}:= \chi_{|\A^{\ts}}$.

 Let $x\in \X_{\G}^{\cl}(\C)$, corresponding to a point $x_{0}\in \X_{\G}^{\cl}$ and an embedding $\iota\colon 
\Q_{p}(x)\into\C$. Let $\pi_{0}$ be the ordinary  automorphic representation of $\G(\A)$ over $\Q_{p}(x_{0})$ attached to $x_{0}$, and let $\pi=\pi_{0}^{\iota}$. We denote by 
 $$\ukappa_{\pi}\in \frak W^{\cl}(\C), \quad \uw, \quad \omega_{\pi}$$
 respectively the weight, numerical weight, and central character of $\pi$. 
  We let $\alpha=\ot_{v}\alpha_{v}\colon F_{p}^{\ts}\to \C^{\ts}$ be the character such that $\Up_{y}f^{\iota}=\alpha(t) f^{\iota}$ for any $f\in \pi_{0}^{\ord}$ and $t\in F_{p}^{\ts}$.
  Then 
 $$ \kappa_{\pi, 0} (z) = \omega_{\pi} (z) z^{w_{0}} , \qquad \kappa_{\pi}(t)= \alpha_{| U_{F, p}^{\circ}}(t) t^{(w+w_{0})/2}$$
are the decompositions of $\kappa_{\pi, 0}$ and $\kappa_{\pi}$ into a product of a smooth and an algebraic character.

Define, as in Corollary \ref{coro Iord 2}, a numerical weight $\uk$ and a smooth character $\kappa_{2}'$ of $U_{F, p}^{\circ}$ by
\beq\lb{kp al} 
\kappa_{\chi}' \kappa_{2}' &= \kappa_{\pi^{\vee}}'  = \alpha_{|U_{F, p}^{\circ}}, \qquad &\xi=\omega_{\pi}\omega_{\chi},\\
 k &= (w-2-  | l | -k_{0})/2, \qquad &k_{0}=w_{0}+l_{0}, 
\eeq
and let $\ukappa_{2}\in \frak W(\C)$ be the associated weight as in \eqref{kappa'}.

 For  $v\vert p$, we choose a Schwartz function $\Phi_{v}=\phi_{v}\in \calS({\bf V}_{v}\ts F_{v}^{\ts})$ as in  \eqref{phi1p}  and \eqref{kp1} (for $\phi_{1}$), and  \eqref{phi2p} and \eqref{kp al} (for $\phi_{2}$); then 
\beq\lb{phivp}
\phi_{v}(x, u) =\delta_{r, v}(x_{1})\one_{{\bf V}^{\circ}_{2, v}}(x_{2}) \delta_{U_{F, v}^{\circ}}(u) \alpha_{v}(u).\eeq
For $v\vert\infty$, let 
$\Phi_{v}=\Phi_{l_{0}, l, k_{0}, k, v}$ be a preimage, under the map \eqref{fock}, of $$\phi_{l_{0}, l, k_{0}, k, v} (x_{1}, x_{2}, u)=\phi_{1, l_{0}, l, v} (x_{1}, u) \phi_{2, k_{0}, k,v}(x_{2}, u), $$
 where  the factors are defined in \eqref{phi1inf}, \eqref{phi2inf}.

\subsubsection{Waldspurger's integral} The next proposition gives  an integral representation for the $L$-function we are interested in. We first define the local terms.
Let $f_{0}\in \pi_{0}^{\ord}$, let $f:= f_{0}^{\iota}$, and let 
 $$\baar{W}(g):=\int_{F\bks \A} f^{\rm a}\left(  \twomat 1x{}1 g\right)\psi(x)\, dx$$
  be the Whittaker function of $f^{\rm a}$ with respect to $\psi^{-1}$. It is related to the $q$-expansion \eqref{FW f} of  $f$ by
$$\baar{W}\left(\twomat y{}{}1\right)= W_{f}^{\C} (y).$$
We assume that $W\colon \G(\A)\to\C$ is factorisable as $\baar{W}=\ot_{v}\baar{W}_{v}$.

For $\Phi=\ot_{v}\Phi_{v}\in \calS({\bf V}\ts \A^{\ts})$, let 
$$R_{r,v}( \baar W_{v}, \Phi_{v}, \chi_{v}) := \int_{Z(F_{v})N(F_{v})\bks \G(F_{v})}\baar W_{v}(g) \delta_{\xi_{v}, r}(g) \int_{T(F_{v})} \chi_{v}(t)
r(gw_{r,v}^{-1})\Phi_{v}(t^{-1}, q(t))\,   dt\, dg,$$
where $\delta_{\xi, r}$ is as in \eqref{dlx}.
 Note that the integral $R_{r,v}$ does not depend on  $r\geq \underline{1}$ unless $v\vert p$; 
 we will accordingly simplify the notation  in these cases.
We also define normalised versions. For  $v\vert p\infty$, let $\Phi_{v}$ be as  fixed in \S~\ref{421}.
Then we put
\beqq
R_{v}^{\natural}(\baar W_{v},\Phi_{v}, \chi_{v})
 &:=
{ |d_{v}|^{-2} |D_{v}|^{-1/2}}
 {  \zeta_{F, v}(2)L(1, \eta_{v}\xi_{v})\over L(1/2, \pi_{E,v}\otimes \chi_{v})}
R_{v}( \baar W_{v},\Phi_{v},\chi_{v}) & \qquad \text{if $v\nmid p\infty$,} \\
R_{r,v}^{\dag}(\baar W_{v} ,\chi_{v}, \alpha_{v}) &: =
{ |d_{v}|^{-2} |D_{v}|^{-1/2}}
 {  \zeta_{F, v}(2)L(1, \eta_{v})\over L(1/2, \pi_{E,v}\otimes \chi_{v})}
q_{F, v}^{r}\alpha_{v}^{-r} R_{r,v}( \baar W_{v}, \Phi_{v}, \chi_{v}) &\qquad\text{if $v\vert p$,}\\
R_{v}^{\dagger}(\baar W_{v}, \chi_{v},  k)
&: = {  \zeta_{F, v}(2)L(1, \eta_{v})\over L(1/2, \pi_{E,v}\otimes \chi_{v})}
R_{v}( \baar W_{v},\Phi_{v}, \chi_{v})
 &\qquad \text{if $v\vert \infty$.}
\eeqq
By a result of Waldspurger (see \cite[Lemma 5.3.2]{pyzz}), for a place $v$ such that $\pi_{v}$ and $\chi_{v}$ are unramified, $\phi_{v}$ is standard, and $\baar W_{v}$ is unramified, we have 
\beq\lb{Rv unr}
R_{v}^{\natural}(\baar W_{v},\Phi_{v}, \chi_{v})=  \baar W_{v}(1).
\eeq

\begin{prop}\label{RSp} Let  $f_{0}\in \pi_{0}^{\ord}$ and assume that $f:=f_{0}^{\iota, \rm a}$ has a factorisable $\psi^{-1}$-Whittaker function $\baar W=\ot_{v} \baar W_{v}$.
Let $\phi^{p\infty}\in\calS({\bf V}^{p\infty}\times \A^{p\infty, \times})$. For sufficiently large $r=(r_{v})_{v\vert p}$, we have
\begin{multline*}
\iota  q_{F, p}^{r}\alpha(\vpi_{p})^{-r} ( f ,w_{r, p}^{-1} I(\phi^{p\infty}; \chi, \xi, \kappa_{2}', k))
 =
{|D_{F}|^{-1} |D_{E}|^{-1/2}}
 {L(1/2, \pi_{E}\otimes \chi)\over \zeta_{F}(2) L(1,\eta)}\\
 \cdot 
 \prod_{v\nmid p\infty}
  R_{r,v}^{\natural}(  \baar W_{v}, \phi_{v},\chi_{v})
   \prod_{v\vert p } 
R_{r,v}^{\dag}( \baar W_{v} ,\chi_{v},  \alpha_{v})
 \prod_{v\vert \infty } 
R_{v}^{\dag}(\baar W_{v} ,\chi_{v},  k),
\end{multline*}
where all but finitely many of the factors in the infinite product are equal to~$1$.
\end{prop}
\begin{proof} As in \cite[Proof of Proposition 3.5.1]{pyzz}, corrected in \cite[Appendix B, under {``\emph{Proposition 2.4.4.1}''}]{nonsplit} to include the factor $q_{F, p}^{r}$. 
\end{proof}

\subsubsection{Non-vanishing of the local integrals}
We recall a fundamental non-vanishing result for our zeta integrals for selfdual $\pi\boxtimes \chi$, as well as a useful refinement.
\begin{lemm} \lb{nnv}
Let $v\nmid p\infty$ be a place of $F$ and let $L$ be a field of characteristic zero.
Let $\pi_{v}$ be a smooth irreducible  representation of $\G(F_{v})$ over $L$, with central character $\omega_{\pi,v}$, and let $\chi_{v}\colon E_{v}^{\ts}\to L^{\ts}$ be a smooth character.
Assume the self-duality condition $\omega_{\pi, v}\chi_{|F_{v}^{\ts}}=\one$.

There exist
 \begin{itemize}
\item
a $4$-dimensional quadratic space ${\bf V}_{v} = \B_{v}$  over $F_{v}$ of the type described  in \S~\ref{sec bf V}, uniquely determined by
$$\eps(\B_{v}) = \eta_{v}\chi_{v}(-1)\eps(\pi_{E,v}\ot\chi_{v}),$$
\item a function $\baar W_{v}$ in the Whittaker model of $\pi_{v}$,
\item  a Schwartz function $\phi_{v}\in \calS({\bf V}_{v}\ts {F}_{v}^{\ts}, L)$,
\end{itemize}
such that  $$ {R}_{v}({\baar W}_{ v}, \phi_{v}, \chi_{v}) \neq 0.$$

If moreover all the data are unramified at a place  $v$ inert in $E$, it is possible to choose $\baar W_{v}$ and $\phi_{v}=\phi_{1, v}\phi_{2,v}$ such that $\phi_{2,v}(0, u)=0$ for all $u$ (condition \eqref{phi cond}).
\end{lemm}
\begin{proof} The argument  in \cite[Proof of Proposition 3.7.1, second paragraph]{pyzz} applies verbatim to prove the first statement. Let us prove the second one. We drop all subscripts $v$. Fix an isomorphism ${\bf V}_{2}\cong E$, and let us choose  $\baar W$ to be a new vector, $\phi_{1,v}$ to be the standard Schwartz function, and $$\phi_{2}(x_{2}, u)=\one_{\OO_{E}^{\ts}} (x_{2})\one_{\OO_{F}^{\ts}}(u).$$
Writing $\doteq $ for an equality up to nonzero scalars, by the Iwasawa decomposition 
$$R(\baar W, \phi, \chi)\doteq \int_{F^{\ts}}  \baar W(\smalltwomat y{}{}1) \int_{E^{\ts}}\chi(t)  \int_{\GL_{2}(\OO_{F})} r(g) \phi(t^{-1}y,  y^{-1}q(t)) \,  dg \, {d^{\ts} y} \, dt.  $$
Let $U_{0}(\vpi^{r})\subset U_{0}:=\GL_{2}(\OO_{F})$ be the set of matrices which are upper-triangular modulo $\vpi^{r}$. It is easy to verify that $\phi_{2}$ is invariant under $U_{0}(\vpi^{r})$ for some $r$, and that  $U_{0}= U_{0}(\vpi^{r})\sqcup \bigsqcup_{b\in \OO_{F,v}/\vpi^{t}} \smalltwomat 1{}{b}1 U_{0}(\vpi)$. Thus the   integral in $dg$ is a constant multiple of
\beqq 
{\ }&{\quad } \int_{\OO_{F}} r(w)[\psi(y^{-1}bq(tx')) \hat{\phi}(x', y^{-1}q(t))]_{|x=t^{-1}y}  \, db \\
&= 
\int_{\OO_{F}}  \int_{E\ts E} \psi( \Tr_{E/F} t  x_{1} ) \psi(y^{-1}q(t) \cd b  q(x)) \one_{\OO_{E}}(x_{1}) \widehat{\one_{\OO_{E}^{\ts}}}(x_{2}) \one_{\OO_{F}}^{\ts}(y^{-1}q(t))]  \,  dx db \\
&=\one_{\OO_{F}}^{\ts}(y^{-1}q(t))
\int_{\OO_{F}} \widehat{\one_{\OO_{E}}}(t)  \int_{E}\psi(y^{-1}q(t) bq(x_{2}) \widehat{\one_{\OO_{E}^{\ts}}}(x_{2}) \, dx_{2}\, db \doteq \one_{\OO_{E}}(t)\one_{\OO_{F}}^{\ts}(y^{-1}q(t))
\eeqq
where the last equality follows from interchanging the order of integration and observing that $\widehat{\one_{\OO_{E}^{\ts}}}(x_{2}) = \vol(\OO_{E}^{\ts})$ for $x_{2}\in \OO_{E}$.  

The last quantity equals  $\phi^{\circ}(t^{-1}y,  y^{-1}q(t))$ for the standard Schwartz function $\phi^{\circ}$; therefore the integral $R$ is a constant multiple of the unramified integral, in particular it is nonzero by \eqref{Rv unr}.
\end{proof}

\subsection{Evaluation of the integrals  at $p$ and $\infty$}
\lb{sec 43}
We explicitly compute the local integrals at the places $v \vert p\infty$.
\subsubsection{$p$-adic integrals}
 Define, for $v\vert p$,
\beq
e_{v}(V_{\pi_{E}\ot \chi}) &:=   { L(1/2, \pi_{E,v}\ot\chi_{v}) \over \zeta_{F,v}(2) L(1, \eta_{v})}  \prod_{w\vert v}\gamma(\chi_{w}\alpha_{\pi, v}|\cdot|\circ N_{E_{w}/F_{v}}, \psi_{v})^{-1}\\
e_{p}(V_{\pi_{E}\ot \chi}) &:=  \prod_{v\vert p} e_{v}(V_{\pi_{E}\ot \chi}).
\eeq

\begin{lemm} \lb{zeta p}
Let $v\vert p$, and assume that $\baar W_{v}$ is normalised by $\baar W_{v}(1)=1$. Then for any sufficiently large $r$ (depending on $\chi_{v}$, $\pi_{v}$) we have 
$$R_{r,v}^{\dag}(\baar W_{v},\chi_{v}, \alpha_{v}) ={\chi_{v}(-1) \over L(1, \eta_{v})} e_{v}(V_{\pi_{E}\ot \chi}). $$
\end{lemm}
\begin{proof}
By \cite[Proposition A.2.2]{pyzz} (with the discriminant factors corrected as in \cite[Appendix B]{nonsplit}), we have
$$R_{r,v}^{\dag}(\baar W_{v},\phi_{v}, \chi_{v}) =   {L(1/2, \pi_{E}\ot\chi_{v}) \over \zeta_{F,v}(2) L(1, \eta_{v})^{2} } Z_{v},
$$
where $Z_{v}$ are integrals defined in \cite[Lemma A.1.1]{pyzz}.  By \cite[Lemma A.1.1]{nonsplit}, we have $$Z_{v}=\chi_{v}(-1)\prod_{w\vert v}\gamma(\chi_{w}\alpha_{v}|\cdot|\circ N_{E_{w}/F_{v}}, \psi_{v})^{-1}.$$ The asserted formula follows.
\end{proof}

\subsubsection{Archimedean integrals}
We compute the local integrals $R_{v}^{\dag}$ when $v\vert \infty$. 
The \emph{standard antiholomorphic Whittaker function} for $\psi^{-1}$ of weight $(w_{0}, w)$ is
\begin{align}\label{antiholwhitt}
\baar W^{(w_{0}, w), {\rm a}}(\smalltwomat z{}{}z\smalltwomat yx{}1 r_{\theta}) 
=
 z^{w_{0}}  \one_{\R^{+}}(y) |y|^{(w+w_{0})/2}\psi(-x+iy)\psi(-w\theta).
\end{align}

\begin{lemm}\label{arch-int}
Let $v\vert \infty$, let  $\baar W_{v}$ be the standard antiholomorphic Whittaker function of weight $(w_{0}, w)$ for ${\psi}^{-1}$.
 Then
$$ R_{v}^{\dag} (\baar W_{v},\chi_{v}, k) 
=  i^{-k_{0}} 
 2^{-1-w} .
$$

\end{lemm}

\begin{proof} 
By the Iwasawa decomposition we can uniquely write any $g\in\GL_{2}(\R)$ as
$$g=\twomat 1 x{}1\twomat z{}{}z  \twomat y{}{}1 \twomat {\cos \theta} {\sin \theta}{-\sin\theta}{\cos\theta}$$
with $x\in\R$, $z\in \R^{\times}$, $y\in \R^{\times}$, $\theta\in \R/2\pi\Z$; the local Tamagawa measure is then $dg=dxd^{\times}z{d^{\times}y\over{|y|}}{d\theta\over 2}$.
Let $\Phi_{v}=\Phi_{l_{0}, l, k_{0}, k, v}$.
Dropping all subscripts $v$, since the weights match the integration over ${\rm SO}(2, \R)$ yields $1$, and we have
\begin{multline*}
R= R(\baar W, \Phi, \chi) =\int_{\R^{\ts}\ts (\R^{\ts} \bks \C^{\ts}) \ts  \R/2\pi\Z\ts{\R^{\ts}}}
\chi(tz)\omega_{\pi}(z) |y|^{(w+w_{0})/ 2} e^{-2\pi  y}  \xi^{-1}(z)|y| \Phi(yzt^{-1}, y^{-1}z^{-2}q(t))\\
 \, d^{\ts}z {d\theta\over 2}  {d^{\ts}y\over |y|} \, dt.
 \end{multline*}

By definition, $\omega_{\pi}\chi\xi^{-1}(z)=1$, so that the integration in $d^{\ts}z$ simply realises the map $\Phi\mapsto \phi$. 
Then
\beq\lb{RRR}
R &=
\pi \int_{\R^{\ts}}  \int_{\R^{\ts} \bks \C^{\ts}}  
\chi(t)  |y|^{(w+w_{0})/ 2} e^{-2\pi y}  |y|   \one_{\R^{+}}(y)  P_{k}(0) y^{( |l |+l_{0})/ 2} \chi(t)^{-1} e^{-2\pi y}
  {d^{\ts}y\over |y|} \, dt\\
&=
2\pi  P_{k_{0},k}(0) \int_{\R^{+}} 
  y^{(w+ | l | +w_{0}+l_{0}) / 2}     e^{-4\pi y}
  {d^{\ts}y} . 
\eeq
where 
   $2=\vol(\R^{\ts}\bks \C^{\ts})$. 

Recall from  \eqref{kp al} and  \eqref{Lag def}
that   $k_{0}=w_{0}+l_{0}$,  $k=(w-2- | l | -k_{0})/2$
and  that $P_{k_{0}, k}(0) =  (2\pi i)^{-k_{0}}  (4\pi)^{-k} (k+k_{0})!$. Then after a change of variables we have 
\beqq 
R&= (2\pi)^{1-k_{0}} i^{-k_{0}} (4\pi)^{-(w- | l | - k_{0} -2 )/2} \Gamma({w-| l |  +k_{0}\over 2}) (4\pi)^{-(w+| l |+k_{0}) / 2} \Gamma({w+| l |+k_{0}\over 2})\\
&= i^{-k_{0}}
  2^{-1-w}  \pi^{2} 
\Gamma_{\C}({w-l   +k_{0}\over 2}) \Gamma_{\C}({w+l+k_{0}\over 2}).
\eeqq
Now the result follows from identifying
 $$\pi^{2} \Gamma_{\C}({w-l   +k_{0}\over 2}) \Gamma_{\C}({w+l+k_{0}\over 2}) = {L(1, \pi_{E, v}\ot\chi)\over \zeta_{F, v}(2) L(1, \eta_{v})}.$$
\end{proof}

\subsection{Interpolation of the local zeta integral}\lb{s44}
\lb{sec 44} Let $\X=\X_{\G} \hat{\ts}\X_{\H}$ be a Hida family for $\G\ts\H$, let $v\nmid p\infty$ be a place of $F$,  and let $\Pi_{v}:=\Pi(\cV_{\G,v})$
be as in  \S~\ref{sec LG}. Let $\X_{\G}^{(v)}\subset \X$ be the open subset containing $\X_{\G}^{\cl}$ over which $\Pi_{v}$ is defined and let $\X^{(v)}=\X_{\G}^{(v)}\hat{\ts} \X_{\H}$. Let $\cW_{v}$ be the $\psi_{v}$-Whittaker model of  $\Pi_{v, \Q_{p}^{\ab}}$, which exists since $\Pi_{v}$ is co-Whittaker (see \cite[\S~4.2]{LLC}); it is $\OO_{\X^{(v)}}[\G(F_{v})]$-isomorphic
to the tensor product of $\Pi_{v}$ and an invertible sheaf with trivial $\G(F_{v})$-action. The space $\cW_{v}$ is, as usual, a space of functions on $\G(F_{v})$, $\psi_{v}$-invariant under the action of the unipotent subgroup $N(F_{v})$.  
For any $x\in \X_{\G}^{\cl}$ and any ${\bf W}_{v}\in \W_{v}$, the twisted  specialisation 
$$\baar{\bf W}_{v|x}(g):= {\bf W}_{v|x}\left(\twomat {-1}{}{}{1} g  \twomat {-1}{}{}{1}\right)$$
 belongs to the $\psi^{-1}$-Whittaker model of $\pi_{x}$.

\begin{prop} \lb{bf R} Let $v\nmid p\infty$. There exists an $\OO(\X^{(v)}_{\Q_{p}^{\ab}})$-linear map
 $${\bf R}_{v}\colon \W_{v}\ot_{\OO({\X^{(v)}_{\G, \Q_{p}^{\ab}})}} \OO{(\X^{(v)}_{\Q_{p}^{\ab}})} \ot_{\Q_{p}} \calS({\bf V}_{v}\ts F_{v}^{\ts})  \to \OO(\X^{(v)}_{\Q_{p}^{\ab}})$$
such that for all ${\bf W}_{v}\in \W_{v}$, all $\phi_{v}\in \calS({\bf V}_{v}\ts F_{v}^{\ts})$,  and all $(x, y)\in \X_{\Q_{p}^{\ab}}^{\cl}(\C)$,
  with underlying embedding $\iota\colon \Q_{p}^{\ab}(x, y)\into \C$, we have
$${\bf R}_{v}({\bf W}_{v}, \phi_{v})(x, y) = R_{v}^{\natural}(\iota\baar{\bf W}_{v|x}, \iota\phi_{v}, \chi^{\iota}_{y, v}).$$
\end{prop}
\begin{proof}
In fact, we may prove a stronger statement by replacing   $\X_{\H}$ by (its image in) $\Y_{\H, v}$, or equivalently  any connected component $\Y_{\H, v}^{\circ}$ thereof (which is an \'etale  torsor for  ${\bf G}_{m, \Q_{p}}^{\{w\vert v\}}$, the action being  induced by multiplication  by the uniformisers in $E_{v}^{\ts}=\prod_{w\vert v}E_{w}^{\ts}$).

The  proof is largely similar to that of  \cite[Proposition 5.2.3]{LLC} (whose statement is corrected in Appendix \ref{app B}); we refer to \emph{loc. cit.} and the sections preceding it for more details on the notions we use. Since $\W_{v}\cong \Pi(\cV_{{\G},v})$ is in the image of the local Langlands correspondence, there exists an irreducible component $\frak X^{\circ}$ of the extended Bernstein variety of \cite[\S~3.3]{LLC} and  a map $\X^{(v)}_{\G}\to \frak X^{\circ}$, such that $\cW_{v}$ is a quotient of the universal co-Whittaker module over $\frak X^{\circ}$.  We may further extend scalars to $\C$ and replace $\frak X^{\circ}$ by a cover of the form $$\wtil{\frak X}^{\circ}= {\bf G}_{m}^{d}$$ for $d=1$ or $2$; then the pull-back $\wtil{\cW}_{v}$ of the universal co-Whittaker module has one of the following shapes:
\begin{enumerate}
\item[(a)]
$\wtil{\cW}_{v}={\rm Ind}_{P_{v}}^{\G(F_{v})}(\beta_{1}\boxtimes \beta_{2})$, where 
 $d=2$ and  $\beta_{i}\colon F_{v}^{\ts}\to  \OO(\wtil{\frak X}^{\circ})^{\ts}$ are the universal characters; 
 \item[(b)]  $\wtil{\cW}_{v}={\rm St}\ot \beta_{1}$, where  
 $d=1$ and  $\beta_{1}\colon F_{v}^{\ts} \to\OO(\wtil{\frak X}^{\circ})^{\ts}$ is the universal character; 
 \item[(c)]  $\wtil{\cW}_{v}=\pi_{0}\ot \beta_{1}$ where  $\pi_{0}$ is a complex supercuspidal representation of $\G(F_{v})$, 
 $d=1$,  and  $\beta_{1} \colon F_{v}^{\ts} \to\OO(\wtil{\frak X}^{\circ})^{\ts}$ is the universal character. 
 \end{enumerate}
 In all cases, we need to show that for every ${\bf W}_{v}\in \wtil{\cW}_{v}$,  there is an element ${\bf R}_{v}({\bf W}_{v}, \phi_{v}) \in \OO( \wtil{\frak X}^{\circ}  \ts \Y_{\H, v})$ such that 
 $$  {\bf R}_{v}({\bf W}_{v}, \phi_{v}) (x, y)= L(1/2, \pi_{x, E, v}\ot\chi_{y, v})^{-1} R_{v}(\phi_{v}, \baar{\bf W}_{v|x}, \chi_{y,v})$$
 for all $x,y$;
 in other words, that the power series in $X_{i}^{\pm 1}:=\beta_{i}(\vpi_{v})^{\pm 1}$  and $Y_{w}^{\pm 1}:=\chi_{\rm univ}(\vpi_{w})^{\pm 1}$ obtained from the integral defining $R_{v}$ is a Laurent-polynomial multiple of the inverse of the Laurent polynomial $L(1/2, \pi_{x, E, v}\ot\chi_{y, v})$. This is proved by the same argument as  in \cite[Proof of Proposition 3.6.1]{pyzz}: since $\wtil{\cW}_{v}$ is torsion-free, it embeds in the representation  $\wtil{\cW}_{v}\ot \mathscr{K}( \wtil{\frak X}^{\circ})$ over the field $\mathscr{K}( \wtil{\frak X}^{\circ})$, so that the usual explicit description of  the Kirillov model used in \emph{loc. cit.} applies.
\end{proof}

\subsection{The $p$-adic $L$-function} \lb{sec 45}
Let
 $$\X\subset \Y_{\G}\hat{\ts }\Y_{\H}$$ 
 be a Hida family with $\X^{\rm sd}\neq \emptyset$, of tame level $U^{p}=U_{\G}^{p}\ts U_{\H}^{p}$. Let $S$ be a finite set of places of~$F$, disjoint from $S_{p\infty}$ and containing all those at which the tame level of $\X$ is not maximal, and let $\Pi:=\Pi^{U_{\G}^{Sp}}_{\X_{\G}}$.   
 If $\X'$ is an (ind-)scheme over $\Q_{p}^{\ab}$, we define $$\X'_{/\Q_{p}^{\ab}}(\C)\subset \X'(\C)$$ to be the subset of geometric points over $\iota^{\ab}$ (that is, those such that the composition $\Spec \C\to \X'_{\Q_{p}^{\ab}}\to \Spec \Q_{p}^{\ab}$ is $\iota^{\ab, \sharp}$). 

 \subsubsection{Whittaker models and $q$-expansions in families}
For $v\in S$, let $\X_{(\G)}^{(v)}\subset \X_{(\G)}$ and $\W_{v}$ be as in \S~\ref{s44}, and let $\X_{(\G)}':=\bigcap_{v\in S} \X_{(\G)}^{(v)}$; it contains $\X_{(\G)}^{\cl}$. 
 
\begin{lemm} There is an isomorphism of $\OO_{\X_{\G, \Q_{p}^{\ab}}'}[\G(F_{S})]$-modules
\beq\lb{bW1}
{\bf W}_{-, S}\colon \Pi_{\Q_{p}^{\ab}} &\stackrel{\cong}{\longrightarrow}\bigotimes_{v\in S} \cW_{v},\\
\bff&\mapsto {\bf W}_{\bff, S}=\ot_{v}{\bf W}_{\bff,v}
\eeq
such that for  all classical points $x\in \X_{\G, \Q_{p}^{\ab}}^{\cl}$ and all $a_{S}\in F_{S}^{\ts}$,  we have 
\beq\lb{bW2}
{\bf W}_{\bff, S}\left(\twomat {a_{S}}{}{}1\right)(x) = W_{\bff(x)^{}}(a_{S}1^{S\infty}),\eeq
where the right-hand side is the $p$-adic $q$-expansion of $\bff (x)$ defined in  \S~\ref{232}.
\end{lemm}
\begin{proof}
By  Proposition \ref{LGC} and \cite[Theorem 4.4.3]{LLC}, after possibly shrinking $\X_{\G}'$ there exist an invertible  sheaf $\cW^{S, U^{S}}$ over $\X'_{\G, \Q_{p}^{\ab}}$  with trivial $\G(F_{S})$-action and  an $\OO_{\X_{\G, \Q_{p}^{\ab}}'}[\G(F_{S})]$-isomorphism
\beq\lb{bW}
{\bf W}\colon \Pi_{\Q_{p}^{\ab}} &\stackrel{\cong}{\longrightarrow} \cW^{S, U^{S}}\ot \bigotimes_{v\in S} \cW_{v},
\eeq
unique uo to $\OO_{\X_{\G, \Q_{p}^{\ab}}'}^{\ts}$, that we may write locally as 
\beqq
\bff &\longmapsto {\bf W}^{S}(1^{S})\ot{\bf W}_{\bff, S} = {\bf W}_{\bff}^{S}(1^{S})\ot\ot_{v\in S}{\bf W}_{\bff, v}\eeqq
where ${\bf W}^{S}(1^{S})$ is a section trivialising $\cW^{S, U^{S}}$. 

For $v\notin S\cup S_{p\infty}$ and $x\in \X_{\G}^{\cl}$, let $\lm_{x,v}\colon F_{v}^{\ts}\to \Q_{p}(x)$ be the smooth function such that $W_{v}(a)=\lm_{x,v}(a) W_{v}(1)$ for any spherical element $W_{v}$ in the Kirillov model of $\pi_{x,v}$; by the standard formulas (see for instance \cite[p. 190]{wald}), there are functions $\lm_{v}\colon F_{v}^{\ts}\to \OO(\X_{\G})$ such that $\lm_{v}(x)=\lm_{x,v}$ for all $x\in \X_{\G}^{\cl}$. Let $\lm^{Sp}:=\ot_{v\notin S\cup S_{p\infty}}\lm_{v}$, and let $\alpha^{\circ}_{p}\colon F_{p}^{\ts}\to \OO(\X_{\G})^{\ts}$ be the $\Up_{p}^{\circ}$-eigencharacter.
 Then we may define  a pair of injective maps in $\Hom_{\cH_{\G}^{\ord}} (\Pi,  \OO_{\X_{\G}'}^{\A^{\infty, \ts}})$ by
\begin{gather*}
\bff \mapsto (W_{\bff}(a)),
\qquad  \qquad
\bff \mapsto (\alpha^{\circ}(a_{p}) \lm^{Sp}(a^{Sp}) {\bf W}_{\bff, S}(a_{S})),
\end{gather*}
where the former arises from \eqref{qexp Y} and interpolates the $q$-expansions $(W_{\bff_{x}}(a))$ for $x\in \X_{\G}^{\cl}$.
By \cite[Lemma 4.2.5]{LLC}, the maps differ by a scalar in $\OO_{\X_{\G}'}$. 
 It follows that the invertible sheaf $\cW^{S,U^{S}}$ is trivial, and that from \eqref{bW} we may deduce an isomorphism \eqref{bW1} normalised so as to satisfy \eqref{bW2}
\end{proof}

\subsubsection{Definition of the $p$-adic $L$-function and interpolation property}
 For each classical point $(x,y)\in \X^{\cl, {\rm sd}}$ and each place $v\nmid p\infty$, let ${\bf V}_{(x, y),v}$ be the quadratic space  given by the application of  Lemma \ref{nnv}  to $\pi_{x} $ and $\chi_{y}$.
 \begin{lemm} The quadratic space ${\bf V}_{v}= {\bf V}_{(x, y),v}$ is independent of $(x,y)\in \X^{\cl, {\rm sd}}$.
 \end{lemm}
 \begin{proof}
This follows from the characterisation in  \eqref{eps B} and the constancy results for  epsilon factors of  \cite[Corollary 5.3.3]{LLC}.
  \end{proof}
Let ${\bf V}^{ p\infty}:= \ot_{v\nmid p\infty}{\bf V}_{v}$, and assume that $S$ is not disjoint from the set $S'$ of inert places $v$ where $U^{p}$ is maximal. Let
$$\mathcal{A}\subset (\Pi_{\mathscr{K}(\X_{\G})} -\{0\})\ts \calS({\bf V}^{p\infty}\ts \A^{p\infty, \ts})$$
be the set of of those pairs $(\bff, \phi^{p\infty})$ such that $\phi^{Sp\infty}$ is standard,  \eqref{phi cond} holds at an inert place $v\in S\cap S'$, and the meromorphic function  ${\bf R}_{v}({\bf W}_{\bff,   v}, \phi_{v})$ on $\X$ is nonzero for all $v\in S$.

For  $(\bff, \phi^{p\infty})\in \mathcal{A}$, we define
  a meromorphic function 
\beq\lb{Lp phi}
\cL_{p}(\cV, \bff, \phi^{p\infty})& \in \mathscr{K}(\X_{\Q_{p}^{\ab}})
\\
\cL_{p}(\cV,\bff,  \phi^{p\infty})( x, y) & := C {\la\la \bff_{x}, \bI^{\ord, \curlyvee}(\phi^{p\infty}; x, y) \ra \ra \over
 \prod_{v\in S} {\bf R}_{v}({\bf W}_{\bff|x,   v}, \phi_{v}, \chi_{y,v})}, 
\eeq
where  we still denote by ${\bf I}^{\ord, \curlyvee}(\phi^{p\infty})$ the restriction to $\X$ of  the $(\Y_{\G}\hat{\ts} \Y_{\H})$-adic form 
 of \eqref{I x y}, and 
$$C=C(x, y):= \omega^{p\infty}_{x}\omega_{y}^{p\infty}(-1)  L(1, \eta_{p}) {|D_{F}|^{-1/2} \zeta_{F}(2)  \over \pi^{[F:\Q]}}. $$
is a constant in $\Q^{\ts}$; here, $\omega_{x}=\omega_{\pi_{x}}$
 and $\omega_{y}=\omega_{\chi_{y}}$.  Note that the (base-change of the) functional $\la\la {\bff}, -\ra\ra$ may be applied to $\bI^{\ord, \curlyvee}(\phi^{p\infty})$, thanks to Lemma \ref{Z-adic}.

\begin{prop} \lb{451}
The collection  
$$(\cL_{p}(\cV, \bff, \phi^{p\infty}))_{(\bff, \phi^{p\infty})\in \mathcal{A}}$$
of meromorphic functions on $\X_{\Q_{p}^{\ab}}$ has the following properties.
\begin{enumerate}
\item
Let $(x, y)\in \X^{\cl}_{\Q_{p}^{\ab} /\Q_{p}^{\ab}}(\C)$ have contracted weight $(k_{0}, w, l)$ satisfying
\beq\lb{wt cond}
 | l_{\tau}| \leq w_{\tau}-2, \qquad |k_{0}|\leq w_{\tau} - 2- | l_{\tau} |. \eeq
If $(x, y)$ is outside  the polar locus of $\cL_{p}(\cV, \bff, \phi^{p\infty})$, we have 
\beq\lb{int Lp2}
\cL_{p}(\cV,\bff, \phi^{p\infty})( x, y) = 
 e_{p\infty}(V_{(\pi, \chi)})
\cdot  \cL(V_{(\pi, \chi)}, 0),\eeq
where  $\pi=\pi_{x}$, $\chi=\chi_{y}$.
\item 
For each $(x, y)\in \X^{\cl, {\rm sd}}$, there is a pair $(\bff, \phi^{p\infty})\in \mathcal{A}$ such that $\cL_{p}(\cV,\bff, \phi^{p\infty})$ does not have a pole at $(x, y)$. 
\end{enumerate}
\end{prop}
Note that the right-hand side of \eqref{int Lp2} is the same as in \eqref{interpol Lp intro} and independent of $(\bff, \phi^{p\infty})$. This will enable us to glue the various $\cL(\cV, \bff, \phi^{p\infty})$ into the sought-for $p$-adic $L$-function.
\begin{proof} The second statement follows from Lemma \ref{nnv}.

It remains to prove the interpolation property. 
Abbreviate $\cL_{p}=\cL_{p}(\cV, \bff, \phi^{p\infty})$, and let $\baar{W}_{v}:=\baar{\bf W}_{\bff|x}$, $\alpha=\alpha_{\pi}$. Denote by $(x_{0}, y_{0})\in \X^{\cl}$ and $\iota \colon \Q_{p}^{\ab}(x_{0}, y_{0})\into \C$ the data corresponding to $(x, y)$. Let $k_{x,y}$ and $\xi_{x,y}$ (respectively $k_{x_{0}, y_{0}}$, $\xi_{x_{0}, y_{0}}$, $\ukappa_{2}=\ukappa_{2, x_{0}, y_{0}}$) be defined by \eqref{kp al} (respectively, by the analogous formulas for the objects attached to $(x_{0},y_{0})$ instead of $(x,y)$).

By the definitions  and  the defining property of $\la\la \, , \, \ra\ra$ in Proposition \ref{lala}, and of   ${\bf R}_{v}$ in Proposition \ref{bf R}, we have
\beqq \cL_{p}(x, y)
& = C\cdot {\iota\la \bff_{x_{0}}, e^{\ord} {\rm I}(\phi^{p\infty}; \chi_{y_{0}}, \xi_{x_{0}, y_{0}}, k_{x_{0}, y_{0}}\ra
\over 
 \prod_{v\in S}
  R_{v}^{\natural}( \iota \baar W_{v}, \phi_{v},\chi_{v})}\\
& = C\cdot  {|D_{F}|^{1/2}\zeta_{F}(2) \cdot  |D_{F}| |D_{E}|^{1/2} \over 
\omega_{x, p}^{}(-1)\cd e_{p}( \ad(V_{\pi})(1)) \cdot 2^{\sum_{v\vert\infty} -1-w_{v}}}  
 \cdot   
 {q_{F, p}^{r}\alpha_{\pi}^{-r}(\bff_{x}^{\rm a},  w_{r,p}^{-1}I_{r}(\phi^{p\infty}; \chi,\xi_{x, y} , \kappa_{2}',  k_{x, y}) 
\over
 L(1, \pi^{}, \ad) \cdot \prod_{v\in S}
  R_{v}^{\natural}( \iota \baar W_{v}, \phi_{v},\chi_{v})}\eeqq
  where $r\in (\Z_{\geq 1})^{S_{p}}$ is sufficiently large, and the  second equality follows from  the interpolation properties of ${\bf I}^{\ord, \curlyvee}$ in \eqref{I ord int 2}, and of $\la\ ,  \ \ra$ in Lemma \ref{pet expl}.
  
  Using first Waldspurger's integral representation as in Proposition \ref{RSp}, and then
 the calculations of local integrals in Lemma \ref{zeta p} and  Lemma \ref{arch-int}, we find 
 \beqq \cL_{p}(x, y)
&=  C\cd 
 {\prod_{v\vert p} R_{r,v}^{\dag}(\baar W_{v}, \chi_{v}, \alpha_{v}) \over 2^{\sum_{\tau}-1- w_{\tau}} \omega_{x,p}^{}(-1)\cd e_{p}(\ad(V_{\pi})(1))}
\cdot  {|D_{F}|^{1/2} L(1/2, \pi_{E}\otimes \chi)\over  L(1,\eta) L(1, \pi,\ad)}
 \prod_{v\vert \infty}
  R_{v}^{\dag}( \baar W_{v},\chi_{v}, k_{v})
\\
&= C\cd i^{-k_{0}[F:\Q]} 
{ \omega_{x,p}^{}\omega_{y, p}^{\rm sm}(-1)}{ e_{p}(V_{\pi_{E}\ot\chi})  \over L(1, \eta_{p})\cd  e_{p}(\ad(V_{\pi})(1))}
\cdot  {|D_{F}|^{1/2} L(1/2, \pi_{E}\otimes \chi)\over  L(1,\eta) L(1, \pi,\ad)}.
\\
&= i^{k_{0}[F:\Q]} \cd
{ e_{p}(V_{(\pi, \chi)})}
\cdot  \cL(V_{(\pi, \chi)}, 0),
\eeqq
as desired. 
\end{proof}

\begin{rema} \lb{coates-pr} 
The interpolation factors $e_{p\infty}(V_{(\pi, \chi)})$ are easily seen to agree with the predictions of  Coates and Perrin-Riou (see \cite{coates}) for a (cyclotomic) $p$-adic $L$-function attached to the `virtual motive' \eqref{virtual}, up to a subtlety that we now explain. With the notation used in \eqref{ev}, for $v\vert p$ consider the $G_{F_{v}}$-representations 
\beqq
\ad(V_{\pi,v})(1):=\End^{0}(V_{\pi})(1)&\supset \ad(V_{\pi,v})(1)^{+}:= \Ker [\ad(V_{\pi})(1)\to \Hom(V_{\pi,v}^{+}, V_{\pi,v}^{-})(1)]\\
&\supset \ad(V_{\pi,v})(1)^{++} =\Hom(V_{\pi,v}^{-}, V_{\pi,v}^{+})(1))  \eeqq
where `$0$' denotes trace-$0$ elements, and  the cokernel of the second containment is isomorphic to the cyclotomic character. Then  \eqref{ev} differs from the ratio of the $v$-adic Coates--Perrin-Riou factors for the hypothetical $L$-functions of $V_{\pi} \ot {\rm Ind}_{G_{F}}^{G_{E}}{V_{\chi}}$ and of $\ad(V_{\pi})(1)$  by the appearance of $\gamma(\iota\ad(V_{\pi,v})(1)^{++}, \psi_{v})^{-1}$ in place of  $\gamma(\iota\ad(V_{\pi,v})(1)^{+}, \psi_{v})^{-1}$. This discrepancy removes the trivial zero $\gamma(\C(1), \psi_{v})^{-1}$ from the latter inverse gamma factor.
\end{rema}

\subsubsection{Rationality and completion of the proof of Theorem {\ref{thm A}}} \lb{sec final}
 By  Proposition \ref{451}
and the density of classical points, the functions $\cL_{p}(\cV, \bff, \phi^{p\infty})=\eqref{Lp phi}\in \mathscr{K}(\X_{\Q_{p}^{\rm ab}})$
  glue to a function $$\cL_{p}(\cV) \in \mathscr{K}(\X_{\Q_{p}^{\rm ab}})$$  which satisfies the required interpolation property, and whose polar locus does not meet the set $\X^{\cl, {\rm sd}}$. All that is left to show is that $\cL_{p}(\cV)$ descends to $\cK(\X)$. It will be a consequence of the following. 
\begin{prop}\lb{sh alg} Let $\X^{\cl, \parallel} \subset \X^{\cl}$ be the sub-ind-scheme of those $(x, y)$
corresponding to a representation $\pi\boxtimes \chi$ whose contracted weight $(k_{0}, w, l)$ satisfies \eqref{wt cond} and is parallel.\footnote{That is, $w_{\tau}$ is independent of $\tau\in \Sg_{\infty}$ and so is $l_{\tau}$. Without this condition, we may have a slightly weaker result.} 
There is a function
 $${\rm L}\in \OO(\X^{\cl , \parallel})$$
 such that for any $z=(x, y) \in \X^{\cl, \parallel}(\C)$ corresponding to a point $z_{0}\in \X^{\cl,||}$ and an embedding $\iota \colon \Q_{p}(z_{0})\into \C$, with attached representation $\pi\boxtimes \chi$, we have
\beq\lb{toto}
{\rm L}(z)=\iota{\rm L}(z_{0}) = i^{-(1+k_{0})[F:\Q]}
\gamma(1, \eta^{\infty}\omega^{\infty}, \psi^{\infty})^{-1}   {\zeta_{F}(2) L(1/2, \pi_{E}\ot \chi)\over \pi^{[F:\Q]} L(1, \pi, \ad)}. \eeq
Here, we denote by $\omega_{\pi}$ be the central character of $\pi$, let  $\omega_{\chi}:=\chi_{|\A^{\ts}}$,   let $\omega=\omega_{\pi}\omega_{\chi}$, and define  $\gamma(s,\omega'^{\infty}, \psi^{\infty}) := \prod_{v\nmid \infty} \gamma(s,\omega'_{v}, \psi_{v})$. 
\end{prop}
\begin{rema} The construction of this paper gives an alternative proof of this result. However, due to the occurrence of the  additive character $\psi$ in the definition of the form $I$ (via the Weil representation), keeping track of rationality requires some burdensome bookkeeping. 
\end{rema}
\begin{proof} This is a consequence of a well-known algebraicity theorem of Shimura \cite[Theorem 4.2]{shimura-hilb}, applied to the newform in the representation $\pi$ and the CM form attached to $\chi$, whose central character is $\eta\omega_{\chi}$.  (For  the comparison of Shimura's periods and adjoint $L$-values, see \cite[Proposition 1.11]{CST}.)
\end{proof}
\begin{coro}
The function $\cL_{p}(\cV)$ belongs to $\cK(\X) \subset \cK(\X_{\Q_{p}^{\ab}})$.
\end{coro}
\begin{proof}
We need to show that \beq\lb{sge}\cL(\cV)^{\sg}=\cL(\cV)\eeq for all $\sg\in \Gal(\Q_{p}^{\ab}/\Q_{p})$.
Let $\X_{\Q_{p}^{\ab}}^{\cl, \parallel, {\rm reg}}$ be the intersection of $\X_{\Q_{p}^{\ab}}^{\cl, \parallel} $ with the complement of the polar locus of $\cL_{p}(\cV)$. Since this set is dense in $\X_{\Q_{p}^{\ab}}$, it suffices to show that \eqref{sge} holds for  the restriction ${\rm L}_{p}(\cV)$ of $\cL_{p}(\cV)$ to ${\X_{\Q_{p}^{\ab}}^{\cl, \parallel, {\rm reg}}}$; in other words, that ${\rm L}_{p}(\cV)$
belongs to $\OO(\X^{\cl, \parallel, {\rm reg}})$.

By \eqref{int Lp2} and \eqref{toto}, 
$${\rm L}_{p}(\cV)(z) = 
 {i^{[F:\Q]}\gamma(1,\eta^{\infty}, \psi^{\infty}) \over {L(1, \eta)}} 
  \cd { \gamma(1, \eta^{\infty}\omega^{\infty}, \psi^{\infty})\over  \gamma(1, \eta^{\infty}, \psi^{\infty}) \gamma(1, \omega^{\infty}, \psi^{\infty})}\cd 
{1\over \gamma(1, \omega^{p\infty}, \psi^{p\infty})^{-1}}  \cd
{e_{p}(V_{(x, y)}) \over \gamma(1, \omega_{p}, \psi_{p})^{-1}} 
\cd {\rm L}(z).$$  We show that all factors belong to $\OO(\X^{\cl, \parallel, {\rm reg}})$:
\begin{itemize}
\item
by the class number formula and  standard results on Gau\ss\ sums, the ratio $ {L(1, \eta)/ i^{[F:\Q]} \gamma(1,\eta^{\infty})} $ is rational, as both numerator and denominator are rational multiples of $|D_{E/F}|^{-1/2}$;
 \item by \eqref{equiv gamma}, the second and fourth ratios are values of functions on $\OO(\X^{\cl, \parallel})$, as $e_{p}(V_{(x, y)})$ is a ratio of inverse gamma factors of characters whose ratio is $\omega_{p}$;
\item as $\X$ is connected and it contains points $z$ with $\omega_{z}=\one$, the character $\omega_{z,v}$ is unramified for all $z\in \X$ and all $v\nmid p$; thus for those $v$,  the quantity $\gamma(\omega_{v})$ is a ratio of $L$-values, hence the third factor is also the value of a function in $\OO(\X^{\cl, \parallel})$;
\item finally, ${\rm L}\in \OO(\X^{\cl,\parallel})$ by Proposition \ref{sh alg}.
  \end{itemize}
  This completes the proof of the corollary and of Theorem \ref{thm A}.\end{proof}

\appendix

\section{Reality shows and double-factorial identities} \lb{double fact}

Consider the identity 
\beqq (*)\qquad \sum_{k=0}^{n} {n \choose k} (2k-1)!! (2n-2k-1)!! = 2^n n!\eeqq
where we recall that
 $(2m-1)!! = 1 \cd 3 \cd 5 \cdots  (2m-1)$ is the number of perfect matchings (into pairs) of a $2m$-element set.  Since  $\Gamma(j+1/2) = {(2j-1)!! \over 2^{j}}\sqrt{\pi}$, the identity $(*)$ is equivalent to \eqref{dfg}. 
 
Quick analytic proofs of $(*)$ have appeared in \cite{AMM}, \cite[Theorem 3]{AMM2}. As we were not able to find a bijective proof in the literature, we give one here. Another bijective proof was communicated to the author by David Callan.

\medskip

A \emph{reality TV show format} is an algorithm whose inputs are called \emph{players' choices} and whose outputs are called \emph{outcomes} (the set of players is partitioned into two disjoint sets, the \emph{producers} and the \emph{participants}). A format is said to be \emph{bijective} if its set of players' choices  is in bijection with its set  of outcomes.

We will describe two bijective formats for reality TV shows, with different sets of players' choices but the same set of outcomes. In each case, there are $2n$ participants forming an ordered set of $n$ heterosexual couples;\footnote{These TV shows, for simplicity or close-mindedness, assume the gender binary.} there are two tropical islands, $Q$ and $H$, and in each case, the outcome is: 
\begin{itemize}
\item
 a new matching of the participants into $n$ disjoint couples (which may be homosexual or heterosexual), and
\item an assignment of each participant to either island $Q$ or island $H$, such that
\item  each person lives in the same island as both their old and their new partner.
\end{itemize}

\begin{enumerate}
\item[Show 1.]  The producers choose a set of couples, send all their members to island $Q$, and send all the other participants  to  island $H$. 
Within each island, people mingle until they form new disjoint couples (heterosexual or homosexual) as they wish. 
\item[Show 2.] The producers pick a permutation $\sigma \colon \{1, \ldots, n\} \to \{1, \ldots, n\} $, then they   do the following.
\begin{itemize}
\item Initialize: $i=1$ and the set-variable $C= \emptyset$ (where $C$ is for `cycle'; to be thought of as the set of couples embarked in the show's boat at a given time);
\item Process: 
\begin{enumerate}
\item  consider couple $i$, set $C_{\rm new}=C_{\rm old}\cup \{i\}$, and interview couple member $p_i$ where: if $C=\{i\}$, then $p_i$ is the woman; if $C\supsetneq \{i\}$, then $p_i$ is the one that  does not yet have a new partner. 

The possible  answers to the interview question are `$H$' and `$Q$'.
\item If $j = \sigma(i) \notin C$ and $p_i$ responds $H$ (resp. $Q$):
\begin{itemize}
\item
 rematch $p_i$ with the person of opposite (resp. same) sex of couple $j=\sigma(i)$. Set
 $i_{\rm new} = j$. Return to (a).
  \end{itemize}
\item  If $j = \sigma(i) \in C $ and $p_i$ responds $H $ (resp. $Q$):
\begin{itemize}
\item 
rematch $p_i $ with the unique non-rematched person of couple $j$, and send all members of the `original couples' in $C$
 to island $H$ (resp. $Q$); set $C_{\rm new}=\emptyset$;
\item if everyone has been rematched,  STOP. Else:   set $i_{\rm new}\in \{1, \ldots, n\}$ to be the smallest such that neither member of couple $i_{\rm new}$ has been rematched. Return to (a).
\end{itemize}
\end{enumerate}
\end{itemize}
\end{enumerate}

\begin{proof}[Proof of $(*)$]  The number of possible players' choices  in Show 1 is the left-hand side  of $(*)$.
The number of possible players' choices in Show 2 is the right-hand side  of $(*)$. But the shows are bijective with the same set of outcomes.
\end{proof}

\section{Errata to \cite{LLC}}\lb{app B}

The conclusions of the statements of  Lemma 5.2.2, Proposition 5.2.3, and Proposition 5.2.4 should respectively have $A[T^{\pm 1}]$, $\OO_{X}[T^{\pm 1}]$, and $\OO_{X}[T^{\pm1}]$ instead of $A[T]$, $\OO_{X}[T]$, and $\OO_{X}[T]$.

\bigskip

\backmatter
\addtocontents{toc}{\medskip}

\end{document}